\documentclass[11pt, reqno]{amsart}
\usepackage{fullpage}

\newif\ifHideFoot
\HideFoottrue  
\HideFootfalse  

\usepackage{enumitem, amssymb,amsmath,amsthm,amscd,mathrsfs,graphicx, color}
\usepackage[cmtip,all,matrix,arrow,tips,curve]{xy}
\usepackage{xr-hyper}
\usepackage{hyperref}

\makeatletter\@input{fakediagonal.tex}\makeatother
\makeatletter\@input{fakebloch.tex}\makeatother

\usepackage[usenames,dvipsnames]{xcolor}
\usepackage{xypic}
\hypersetup{colorlinks=true,citecolor=ForestGreen,linkcolor=Maroon,urlcolor=NavyBlue}
\usepackage{mathtools}
\usepackage{mathpazo}

\setlist[enumerate]{leftmargin=8.5mm}

\numberwithin{equation}{section}
\newtheorem{teo}{Theorem}[section]
\newtheorem{pro}[teo]{Proposition}
\newtheorem{lem}[teo]{Lemma}
\newtheorem{cor}[teo]{Corollary}

\newtheorem{teoalpha}{Theorem}

\theoremstyle{definition}
\newtheorem{dfn}[teo]{Definition}
\newtheorem{exa}[teo]{Example}

\theoremstyle{remark}
\newtheorem{rem}[teo]{Remark}

\ifHideFoot

\newcommand{\Yano}[1]{}
\newcommand{\Jeff}[1]{}
\newcommand{\Charles}[1]{}

\else

\newcommand{\marg}[1]{\normalsize{{
   \color{red}\footnote{{\color{blue}#1}}}{\marginpar[\vskip
   -.25cm{\color{Maroon}\hfill\thefootnote$\implies$}]{\vskip
    -.2cm{\color{Maroon}$\impliedby$\tiny\thefootnote}}}}}
\newcommand{\Yano}[1]{\marg{(Yano) #1}}
\newcommand{\Jeff}[1]{\marg{(Jeff) #1}}
\newcommand{\Charles}[1]{\marg{(Charles) #1}}

\fi

\def\res{{\bf\rm R}}
\def\sep{{\text{sep}}}
\def\cris{{\text{cris}}}
\DeclareMathOperator{\im}{im}

\newcommand{\til}[1]{{\widetilde{#1}}}

\def\cx{{\mathbb C}}

\def\rat{{\mathbb Q}}
\def\integ{{\mathbb Z}}

\def\cali{{\mathcal I}}
\def\iso{\cong}

\renewcommand{\bar}[1]{{\overline{#1}}}
\newcommand{\ubar}[1]{{\underline{#1}}}

\DeclareMathOperator{\aut}{Aut}
\DeclareMathOperator{\alb}{Alb}

\DeclareMathOperator{\End}{End}
\DeclareMathOperator{\Hom}{Hom}

\DeclareMathOperator{\spec}{Spec}
\DeclareMathOperator{\pic}{Pic}

\DeclareMathOperator{\Ab}{Ab}
\DeclareMathOperator{\A}{A}

\DeclareMathOperator{\chow}{CH}

\newcommand{\LKtrace}[1]{\underline{\underline{#1}}}

\title{A functorial approach to regular homomorphisms}

\author{Jeffrey D. Achter}
\address{Colorado State University, Department of Mathematics,
 Fort Collins, CO 80523,
 USA}
\email{j.achter@colostate.edu}

\author{Sebastian Casalaina-Martin }
\address{University of Colorado, Department of Mathematics,
 Boulder, CO 80309, USA }
\email{casa@math.colorado.edu}

\author{Charles Vial}
\address{Universit\"at Bielefeld, Fakult\"at f\"ur Mathematik, Germany}
\email{vial@math.uni-bielefeld.de}

\thanks{Research of the first and second authors is supported in part by grants from the Simons Foundation (637075 and 581058, respectively).}

\begin{document}

 \begin{abstract}
  Classically, regular homomorphisms have been defined as a replacement for Abel--Jacobi maps for smooth varieties
  over an algebraically closed field. In this work, we interpret regular
  homomorphisms as morphisms from the functor of families of algebraically
  trivial
  cycles to abelian varieties and thereby define regular homomorphisms
in the relative setting, \emph{e.g.}, families of schemes
parameterized by a smooth variety over a given field.
 In that general setting, we establish
  the
  existence of an initial regular homomorphism, going by the name of algebraic
  representative, for codimension-2 cycles on a smooth proper scheme over the
  base. This extends a result of Murre for codimension-2 cycles on a smooth
  projective scheme over an algebraically closed field. In addition,
  we prove
  base change results for algebraic representatives as well as
  descent properties for algebraic representatives along separable field
  extensions.
  In the case where the base is a smooth
   variety over a subfield of the complex numbers we identify
  the
  algebraic representative for relative codimension-2 cycles with a subtorus of
  the intermediate Jacobian fibration which was constructed in previous work.
 \end{abstract}

  \maketitle
 \setcounter{tocdepth}{1}

 \tableofcontents

 \section*{Introduction}

Abel--Jacobi maps provide a fundamental tool in the study of smooth complex projective varieties.
  For algebraic curves, the Abel--Jacobi map on divisors essentially encodes all of the data of the curve, while for higher dimensional varieties, the Abel--Jacobi map provides a basic instrument for studying higher codimension algebraic cycles.  The most well-studied examples are the Abel map, taking algebraically trivial divisors to the connected component of the Picard scheme, and the Albanese map, taking algebraically trivial $0$-cycles to the Albanese variety.  In fact, in these instances, the theory has been developed for smooth projective varieties over an arbitrary field.   One of the main goals of this paper is to develop a framework for defining generalizations of  the remaining intermediate Abel--Jacobi maps for smooth projective varieties over arbitrary fields.   At the same time, studying Abel--Jacobi maps for families of smooth complex projective varieties has also been extremely fruitful, leading for instance to the theory of motivated normal functions.   The theory of the Abel and Albanese maps has been generalized to the relative setting in the category of schemes, and another of our goals is to develop a framework for defining generalizations of  the remaining Abel--Jacobi maps in the relative setting of morphisms of schemes.
 In what follows, we recall the classical definition of a regular homomorphism, which has been used as a replacement for Abel--Jacobi maps over algebraically closed fields, explain the relationship with our new functorial definition of a regular homomorphism, and then discuss our main results on existence, and on base change in the relative setting.

\smallskip

 Let $X$ be a scheme of finite type over a field $K$. A codimension-$i$ cycle
 class $a\in \mathrm{CH}^i(X)$ is said to be \emph{algebraically trivial} if it
 is ``parameterized'' by a smooth integral scheme over $K$. Precisely, the cycle class $a$ is
 algebraically trivial if and only if there exist a smooth integral scheme $T$
 of finite type
 over $K$, a pair of $K$-points $t_0,t_1 \in T(K)$ and a cycle $Z \in
 \operatorname{CH}^i(T\times_K X)$, which is referred to as a family of
 codimension-$i$ cycle classes on
 $X$ parameterized by $T$, such that $a = Z_{t_1} - Z_{t_0}$, where $Z_{t_i}$
 denote the Gysin fibers. It is in fact equivalent to require $T$ to be a smooth
 quasi-projective curve over $K$\,; see \emph{e.g.},
 \cite[Prop.~3.10]{ACMVabtriv}.
 Henceforth, the subgroup of $\operatorname{CH}^i(X)$ consisting of
 algebraically
 trivial cycles will be denoted $\operatorname{A}^i(X)$.\medskip

 Let $k$ be an algebraically closed field and let $X$ be a scheme of finite type
 over $k$. Classically, a  \emph{regular homomorphism}   consists of a
 group homomorphism $\phi : \operatorname{A}^i(X) \to A(k)$ to the closed points
 of an abelian variety $A$ over $k$ satisfying the property\,:
 for a smooth integral scheme $T$ of finite type over $k$, a point $t_0\in T(k)$
 and a cycle class
 $Z\in \operatorname{CH}^i(T\times_k X)$,
 the map defined on $k$-points by
 $$T(k)\to A(k),\quad  t \mapsto \phi(Z_t-Z_{t_0})$$ is induced
 by a $k$-morphism $T\to A$.  The standard example of a regular homomorphism is the Abel--Jacobi map for a smooth complex projective variety.
 Observe that replacing $Z$ with $Z-
 (Z_{t_0}\times T)$, one can equivalently define a  regular homomorphism to
 consist of a group homomorphism $\phi : \operatorname{A}^i(X) \to A(k)$ to the
 closed points of an abelian variety $A$ over $k$ satisfying the property\,:
 for a smooth scheme $T$ of finite type over $k$ and a cycle class
 $Z\in  \operatorname{CH}^i(T\times_k X)$ that is fiberwise algebraically
 trivial,
 the map defined on $k$-points by
 $$T(k)\to A(k),\quad  t \mapsto \phi(Z_t)$$ is induced by a
 $k$-morphism $T\to A$. Our starting point is that a regular homomorphism can
 thus be seen as a morphism of functors
 $$\Phi : \mathscr{A}^i_{X/k} \to A$$ by the recipe $\Phi(k) = \phi$, where
 $\mathscr{A}^i_{X/k} $ is the functor  of ``families of algebraically trivial
 cycle classes'', \emph{i.e.}, the functor assigning to each scheme $T$ smooth and
 of
 finite type over $k$ the subgroup of $\operatorname{CH}^i(T\times_k X)$
 consisting of cycle classes that are fiberwise (relative to $T$) algebraically
 trivial.
 The advantage of this functorial approach is that now it makes sense to define
 regular homomorphisms without the restriction that $k$ be algebraically closed.
 With the functorial language, an \emph{algebraic representative} for a scheme
 $X$ of finite type over a field $K$, if it exists, is a morphism
 $$\Phi^i_{X} : \mathscr{A}^i_{X/K} \to \mathrm{Ab}^i_{X/K}$$ that is initial
 among all morphisms from $\mathscr{A}^i_{X/K}$ to abelian varieties over $K$.
 In
 particular, it is unique up to unique isomorphism. We refer to \S
 \ref{S:AlgRepFun} for definitions.   When $X$ is a smooth projective
 variety over $K$,  the algebraic representative for divisors is given
 by the Abel map to $(\operatorname{Pic}^0_{X/K})_{\text{red}}$, and
 the algebraic representative for $0$-cycles is given by the Albanese
 variety.

In the end, our goal is to extend the notion of regular homomorphism to the relative setting.  However, as it turns out, many of the arguments in the relative setting can be reduced to the case of schemes over a field, and to a base change of field statement.  Moreover, even in the case of schemes over fields, we can improve some of our previous results over fields  \cite{ACMVdcg} using our new framework.
In this spirit, our first main result is a
 base-change/descent result for algebraic representatives along separable field
 extensions\,:

 \begin{teoalpha}  \label{T2:mainalgrep}
  Let $X$ be a scheme of finite type over a  field $K$ and let
  $\Omega/K$ be a (not necessarily algebraic) separable field extension.
  Then an algebraic representative
  $\Phi^i_{X_\Omega} : \mathscr A^i_{X_\Omega/\Omega} \to
  \mathrm{Ab}^i_{X_\Omega/\Omega}$ exists if and only if an algebraic
  representative $\Phi^i_{X} : \mathscr A^i_{X/K} \to \mathrm{Ab}^i_{X/K}$
  exists.
  If this is the case, we have in addition\,:
  \begin{enumerate}[label=(\roman*)]
   \item There is a canonical isomorphism $\mathrm{Ab}^i_{X_\Omega/\Omega} \stackrel{\sim}{\longrightarrow}
   (\mathrm{Ab}^i_{X/K})_\Omega$\,;
   \item
   $\Phi^i_{X_\Omega}(\Omega) : \mathrm{A}^i(X_\Omega) \to
   \mathrm{Ab}^i_{X_\Omega/\Omega}(\Omega)$ is
   $\mathrm{Aut}(\Omega/K)$-equivariant, relative to the above identification.
  \end{enumerate}
 \end{teoalpha}

 Theorem~\ref{T2:mainalgrep} generalizes and strengthens our previous result \cite[Thms.~3.7 \&
 4.4]{ACMVdcg} not only by eliminating the assumption that $K$ be perfect, but also by
 showing that the natural homomorphism $\mathrm{Ab}^i_{X_\Omega/\Omega} \to
 (\mathrm{Ab}^i_{X/K})_\Omega$ is an isomorphism rather than merely a purely
 inseparable isogeny.
 Theorem~\ref{T2:mainalgrep} is proven in Theorem \ref{T:mainalgrep}.



 \medskip

We now turn our attention to regular homomorphisms in the relative setting.   The natural framework for our definition  (see Definition \ref{D:fun} and \S\ref{SSS:def}) is to consider a  morphism $X\to S$ of finite type where $S$ is a smooth separated scheme of finite type over a regular Noetherian scheme $\Lambda$.  Then, as in the case of schemes over fields discussed above, one can define the sheaf  $\mathscr A^i_{X/S/\Lambda}$ of  ``families of algebraically trivial cycle classes"  (Definition \ref{D:AlgFun}), and with this one can then define a regular homomorphism $\Phi:\mathscr A^i_{X/S/\Lambda}\to A$ as a morphism of functors to an abelian $S$-scheme~$A$ (Definition \ref{D:reghom}).   An algebraic representative is then defined as an initial regular homomorphism (Definition \ref{D:fun}).

  A natural starting point is to establish existence of algebraic representatives.
 In general, even for smooth projective  varieties
 over an algebraically closed field it  is still an open problem to decide
 whether algebraic representatives exist.  However, there are three cases where existence is known for smooth projective varieties over an algebraically closed field\,:  for cycles of dimension-$0$ and for cycles of
 codimension-$1$ or $2$.
Indeed, as mentioned above, for a smooth projective variety over an algebraically closed field, the algebraic representative in codimension-$1$ is given by
 the reduced connected component of the Picard scheme, while the algebraic
 representative in dimension-$0$ is given by the  Albanese variety. Murre established
 the existence
of  algebraic
 representatives for codimension-$2$ cycles for smooth projective varieties over an algebraically closed field in \cite{murre83}. In \cite{ACMVdcg}, we showed that such
 algebraic representatives admit  models over any perfect field.
The following theorem provides a uniform statement for the case $i=1,2,\dim X$, generalizing the results to the relative setting.
Note that even specializing to
 the case where
 $\Lambda = S = \spec K$ for some field $K$, the following theorem in particular provides a new result, namely
 the existence of algebraic representatives for $i=2$ for smooth proper varieties over any
 field.

 \begin{teoalpha}\label{T2:mainAb2}
  Let $S$ be a scheme that is smooth separated and of finite type over a
  regular
  Noetherian scheme~$\Lambda$.
  Suppose~$X$ is a
  smooth proper scheme of finite type over $S$. Fix  $i=1,2$, or $\dim_S X$.  Then $X/S/\Lambda$
  admits an algebraic representative $$\Phi^i_{X/S} : \mathscr{A}^i_{X/S/\Lambda} \to
  \operatorname{Ab}^i_{X/S}$$
  for codimension-$i$ cycles.

\noindent  Moreover,
  if $\Lambda = \spec K$ for some subfield $K\subseteq \cx$,
  then
  \begin{enumerate}[label=(\roman*)]
   \item  For any dominant $\Lambda$-morphism $S' \to S$ of smooth separated schemes of finite type over $\Lambda$, there is a canonical isomorphism  of
   $S'$-abelian schemes $$ \operatorname{Ab}^i_{X_{S'}/S'} \stackrel{\sim}{\longrightarrow} (
   \operatorname{Ab}^i_{X/S})_{S'}.$$
   \item   The homomorphism
   $\Phi^i_{X/S}(\Omega) : \A^i(X_\Omega) \to \operatorname{Ab}^i_{X/S}(\Omega)$
   is an isomorphism
   on  torsion  for $s:\spec \Omega = \spec \kappa(S)^{\sep} \to S$ a separable
   closure of a generic point of $S$.
  \end{enumerate}
 \end{teoalpha}

The most interesting case is the case $i=2$. The existence statement of Theorem~\ref{T2:mainAb2}
 is proven in
 Theorem~\ref{T:mainAb2},
 item~$(i)$ is proven in Theorem~\ref{T:basechangechar0} and item~$(ii)$ is
 proven in Theorem~\ref{T:mainAb2char0}.
 The latter is obtained by using the fact proved by
 Murre \cite[\S 10]{murre83} that $\Phi^2_{X_\Omega/\Omega}$ coincides with the
 Abel--Jacobi map when $\Omega = \cx$, which explains why we have to restrict to
 characteristic zero.

  We note nonetheless that even in the cases $i=1$ or $\dim_S X$,
  Theorem \ref{T:mainAb2} provides some interesting new
  results. First, consider the case $i=1$.  Then the theorem does not
  require the existence of the connected component of the Picard
  scheme $\operatorname{Pic}^0_{X/S}$ to get the existence of the
  algebraic representative (we discuss cases where
  $\operatorname{Pic}^0_{X/S}$ and the algebraic representative agree
  in Theorem \ref{T:AlgRep-Co1} and Remark \ref{R:Pic0Kred}).  In
  other words, in this case the theorem does not just reduce to the
  theory of Picard schemes, and we get a more general relative
  existence theorem for the algebraic representative. Next consider
  the case $i=\dim_S X$.   Then the theorem does not require the
  existence of Grothendieck's relative Albanese scheme \cite{FGA},
  which is predicated on the existence of
  $\operatorname{Pic}^0_{X/S}$, to get the existence of the algebraic
  representative.  Again, we discuss cases where
  $\operatorname{Alb}_{X/S}$ and the algebraic representative agree in
  Theorem \ref{T:AlbAb} and Lemma \ref{L:AlbAb}.    In other words, in
  this case, the theorem does not just reduce to the theory of the
  Albanese scheme, and we get a more general relative existence theorem for the algebraic representative.

\medskip

Having established existence results for the algebraic representative,
we mention that nowhere in the above discussion did we assert that the
algebraic representative is nontrivial.  For $i=1$ and $i=\dim_S X$,
there are well-established results on the dimensions of the Picard
scheme and the Albanese scheme that can be used.  However, for codimension-$2$ cycles, the question is much more subtle.  In particular, one can see that the main issue is to establish the existence of a single non-trivial regular homomorphism.  In positive characteristic, this is quite subtle, and we refer the reader to \cite{ACMVBlochMap} where we discuss this \emph{via} the geometric coniveau filtration (for the case of fibrations in conics over the projective plane see also \cite{beauville77}).
However, in characteristic~$0$, due to the base change result in Theorem \ref{T2:mainAb2},  the existence of nontrivial regular homomorphisms in the relative setting is ensured, \emph{via}  base-changing along a dominant morphism $\spec \cx \to S$ and then   by using the fact that the Abel--Jacobi map induces a regular homomorphism (Example~\ref{Ex:AJ}, Theorem \ref{T:geomNF}).
In fact, for $i=1,2,\dim X$, we can show that the Abel-Jacobi map provides an algebraic representative in the relative setting.
Precisely, thanks to our functorial approach to regular homomorphisms, we can
 extend
 \cite[\S 10]{murre83} to the relative setting and compare the base-change of
 the algebraic representative $\Phi^i_{X/S}$ to $\cx$ with normal functions
 attached to families of fiberwise algebraically trivial codimension-$i$ cycles.
 The following theorem, which can be read off Theorems~\ref{T:geomNF}
 and~\ref{T:MK},  is obtained by combining our previous work on normal functions
 attached to cycles that are fiberwise algebraically trivial \cite{ACMVsigma}
 with Theorems~\ref{T2:mainalgrep} and~\ref{T2:mainAb2}.

 \begin{teoalpha}\label{T2:NF}
  Let $S$ be a smooth separated  scheme of finite type over a field $K\subseteq \cx$ and
  let
  $f : X\to S$ be a smooth projective morphism.   Denote $\mathrm{J}^{2i-1}(X_\cx/S_\cx)$
  the relative intermediate Jacobian attached to the variation of Hodge
  structures
  $\mathrm{R}^{2i-1}(f_\cx)_*\integ$.
 Then
 \begin{enumerate}[label=(\roman*)]
\item{\em (The relative $AJ$ map is a regular homomorphism)} There exist a surjective regular homomorphism  $\Phi^i_{AJ_{X/S}} :  \mathscr{A}^i_{X/S}\to \mathrm{J}^{2i-1}_{a,X/S}$ and a canonical inclusion $\iota : (\mathrm{J}^{2i-1}_{a,X/S})_\cx \hookrightarrow\mathrm{J}^{2i-1}(X_\cx/S_\cx)$ of relative complex tori. The inclusion $\iota$ has the property that for any $Z \in
\mathscr{A}^i_{X_\cx/S_\cx}(S_\cx)$ the composition $\iota \circ
\Phi^2_{X_\cx/S_\cx}(S_{\mathbb C})(Z)$ coincides with the normal function $\nu_Z$ attached to~$Z$. In particular, restricting the co-domain of the normal function to
$(\mathrm{J}^{2i-1}_{a,X/S})_\cx$, the normal function $\nu_Z$ is algebraic and if
$Z$
is defined over $K$ then $\nu_Z$ is defined over~$K$.
\item{\em (The relative $AJ$ map is an algebraic representative for $i=1,2,\dim_SX$)} For $i=1,2,\dim_SX$, the natural homomorphism $\mathrm{Ab}^i_{X/S} \to \mathrm{J}^{2i-1}_{a,X/S}$, provided by Theorem~\ref{T2:mainAb2}, is an isomorphism.
 \end{enumerate}
 \end{teoalpha}

 Note that for $i=1$ or $\dim_S X$, Theorem~\ref{T2:NF} is classical and the inclusion $\iota$ is an isomorphism.\medskip

 Regarding item $(i)$ of Theorem~\ref{T2:mainAb2}, we have to restrict to
 characteristic zero in order to use the Faltings--Chai extension theorem for
 abelian schemes. However, a first step consists in using the
 N\'eron--Ogg--Shafarevich criterion, which holds without any restrictions on the
 characteristic.
 Consequently, we can show the following base-change result concerning the
 algebraic representative.

 \begin{teoalpha}\label{T2:dvr}
  Let $S/ \Lambda$ be either the spectrum of a DVR (with $S=\Lambda$), or a
  smooth quasi-projective variety over a field $K\subseteq \cx$, and denote
  $\eta$
  its generic point. Suppose $X$ is a smooth and proper scheme over $S$.  Fix $i=1,2$, or $\dim_S X$.
  Then
  the natural homomorphism of $\eta$-abelian varieties
  $$\mathrm{Ab}^i_{X_{\eta}/\eta} \longrightarrow (\mathrm{Ab}^i_{X/S})_{\eta}$$
  is an isomorphism.
 \end{teoalpha}

 In the case where $S$ is a DVR, this is Theorem~\ref{T:dvr}, while in the other
 case this is
 Theorem~\ref{T2:mainAb2}$(i)$ (by going to the limit over all open subsets of
 $S$) or can be seen as
 a combination of Theorems~\ref{T:geomNF} and~\ref{T:MK}.
 In the
 first case, the key point is to show that $ \mathrm{Ab}^2_{X_{\eta}/\eta}$
 admits a model over $S$\,; this is achieved \emph{via} the N\'eron--Ogg--Shafarevich
 criterion by showing (Proposition~\ref{P:BlochMap}) that $T_\ell
 \mathrm{Ab}^2_{X_{\eta}/\eta}$ admits a $\operatorname{Gal}(\eta)$-equivariant
 embedding as a $\integ_\ell$-module in the torsion-free quotient of
 $H^3(X_{\eta^{\sep}},\integ_\ell(2))$. In the second case, we can extend $
 \mathrm{Ab}^2_{X_{\eta}/\eta}$ to an abelian scheme over $S$ thanks to the
 Faltings--Chai extension theorem \cite[Cor.~6.8, p.185]{FC}.\medskip

 \subsection*{Acknowledgements}
We thank Olivier
Wittenberg for useful comments, and Kiran Kedlaya and Adrian Vasiu for
helpful conversations about crystals and $p$-divisible groups.

 \subsection*{Notation and Conventions} Given a field $K$, $K^{\sep}$ will
 denote
 an algebraic separable closure of $K$. Since we will deal with separable
 extensions of fields that are not necessarily algebraic,  we make it clear that
 a field $K$ is said to be \emph{separably closed} if the only separable and
 algebraic extension of $K$ is $K$ itself.
 Given a scheme $X$ of finite type over a field $K$,
 we denote $\mathrm{A}^i(X)$ the subgroup of $\operatorname{CH}^i(X)$ consisting
 of cycle classes that are algebraically trivial.

 Unless stated otherwise, $\Lambda$ denotes an integral regular Noetherian
 scheme, $S$ denotes an integral scheme that is smooth separated and of finite
 type over $\Lambda$, and $X$ denotes a scheme of finite type over~$S$. Given an
 integral scheme $T$ with function field $\kappa(T)$, we denote $\eta_T : \spec
 \kappa(T) \to T$ its generic point and we denote $\eta_T^{\sep} : \spec
 \kappa(T)^{\sep} \to T$ the corresponding separably closed point over $\eta_T$.

 \section{Regular homomorphisms as functors}\label{S:AlgRepFun}

 \subsection{The functor of algebraically trivial cycle classes}

 We fix an integral  regular Noetherian  scheme $\Lambda$,
 and an integral
 scheme $S$ that is smooth separated and of finite type over $\Lambda$.
 We will denote the  generic points of $S$ and $\Lambda$, respectively, as
 $\eta_S$ and $\eta_\Lambda$.

 \begin{rem}
  The assumption that  $S$ and
  $\Lambda$ be irreducible (and hence integral, since $\Lambda$ is regular and
  since $S$  is smooth over the regular scheme $\Lambda$) is not
  essential\,; all statements in the paper hold without this integrality
  assumption simply by working component-wise.
 \end{rem}

 We denote
 by $$\mathsf {Sm}_\Lambda/S$$ the
category whose objects are
separated
$S$-schemes that are smooth and of finite type over $\Lambda$ with structure map to $S$ being \emph{dominant}, and whose morphisms are given by morphisms of $S$-schemes.
 For later reference, we recall that any morphism $t:T'\to T$ in $\mathsf
 {Sm}_\Lambda/S$ -- in fact any morphism of smooth separated schemes of finite
 type
 over $\Lambda$ -- factors as
 $$
 \xymatrix{
  T' \ar[r]^<>(0.5){\gamma_t} & T'\times_\Lambda T \ar[r]^<>(0.5){pr_T} & T
 }
 $$
 where $\gamma_t$, the graph of $t$, is a regular embedding
 \cite[B.7.3]{fulton},
 and $\operatorname{pr}_T$ is the second projection, and is smooth by base
 change.   In particular, $t$ is locally complete intersection (lci) in the sense of \cite[B.7.6]{fulton},
 extended to the the situation of schemes over $\Lambda$, and if both $T$ and
 $T'$ are equidimensional over $\Lambda$, then $t$ is of relative dimension
 $d_T-d_{T'}$ , where $d_T=\dim_\Lambda T$ and $d_{T'}=\dim_\Lambda T'$ are the
 relative dimensions of $T$ and $T'$ over $\Lambda$.

 Let $X\to S$ be a
  morphism
 of finite type over $S$.
 For each $T\to S$ in $\mathsf {Sm}_\Lambda/S$, denote by $X_T$ the base-change
 $X\times _ST$.
 Note that $X_T$ is of finite type over $\Lambda$.    Following \cite[\S
 20]{fulton}, we have $\operatorname{CH}^i(X_T/\Lambda)$, which, abusing
 notation, we will henceforth simply write as $\operatorname{CH}^i(X_T)$.

 \begin{rem}[The case $S=\Lambda=\operatorname{Spec}K$]
  An important special case is when $S=\Lambda=\operatorname{Spec}K$ for some
  field~$K$, and $f:X\to \spec K$ is  smooth projective. The classical case is
  the case where we assume in addition
  that
  $K=k$ is algebraically closed (see \S \ref{SS:reg}).
 \end{rem}

 Recall that for any morphism $t:T'\to T$ in $\mathsf {Sm}_\Lambda/S$, being lci
 as described above, there is a refined Gysin pull-back \cite[p.395]{fulton}
 $$
 t^!:\operatorname{CH}^i(X_T)\to \operatorname{CH}^i(X_{T'}),
 $$
 which, as pointed out on \cite[p.395]{fulton}, satisfies all the fundamental
 properties of \cite[\S 3, \S 6]{fulton}.

 \begin{dfn}[Functor of algebraically trivial cycles] \label{D:AlgFun}
  In the notation above, for a nonnegative  integer $i$, the \emph{functor of
   codimension-$i$ algebraically trivial cycles on $X$ over $S$} is the
  contravariant functor
  $$\mathscr {A}^i_{X/S}
  :\mathsf {Sm}_\Lambda/S \longrightarrow \mathsf {AbGp},
  $$
  sometimes denoted $\mathscr A^i_{X/S/\Lambda}$ for clarity,
  to the category of abelian groups $\mathsf {AbGp}$ given by families of
  algebraically
  trivial cycles on $X/S$.
  Precisely, given $T$ in $\mathsf {Sm}_\Lambda/S$, we take  $\mathscr
  {A}^i_{X/S}(T)$ to be the group of cycles classes $Z \in
  \operatorname{CH}^i(X_T)$ such that
  for  every  separably closed field $\Omega$ and every
  $\Omega$-point $t:
  \operatorname{Spec}(\Omega) \to  T$ obtained as an inverse limit of dominant
  morphisms $(t_n:T_n \to T)_{n\in \mathbb N}$ in $\mathsf {Sm}_\Lambda/T$,
  \emph{i.e.}, $\operatorname{Spec}\Omega=\varprojlim_n T_n$,
  the cycle $Z_t:=t^!Z:=\varinjlim_{n}t_n^! Z \in
  \varinjlim_{n} \operatorname{CH}^i(X_{t_n})= \operatorname{CH}^i(X_t)$
  is algebraically trivial.
  The functor is defined on  morphisms $t:T'\to T$ in
  $\mathsf {Sm}_\Lambda/S$  \emph{via}  the refined Gysin pullback $t^!$ for lci
  morphisms  \cite[\S 6.6]{fulton}\,;    we will sometimes use the notation
  $Z_{T'}=t^!Z$.
 \end{dfn}

 We now establish a result that helps to clarify the meaning of
 Definition~\ref{D:AlgFun}.  The first observation is that if $\Omega$ is a
 separably closed
 field, and
 $t:
 \operatorname{Spec}(\Omega) \to  T$ is obtained as an inverse limit of
 morphisms $(t_n:T_n \to T)_{n\in \mathbb N}$ in $\mathsf {Sm}_\Lambda/T$, and 
 if $T$ is integral, then the image of $\operatorname{Spec}\Omega$ is the generic
 point of  $T$.
 Moreover, under the composition   $\spec \Omega \stackrel{t}{\to} T \to
 \Lambda$, the image of $\operatorname{Spec}\Omega$ is the generic point
 $\eta_\Lambda$, and $\Omega/\kappa(\eta_\Lambda)$ is separable\,;
 this is an immediate consequence of the fact
 \cite[\href{https://stacks.math.columbia.edu/tag/037X}{Tag
  037X}]{stacks-project} that
 a field extension $L/K$ is
 separable
 if and only if $L$ is a direct limit of rings that are smooth
 and of finite type over $K$.
 Note also that if $P\in |T|$ is the image of $t$ in the underlying topological
 space, then $\kappa(P)/\kappa(\eta_\Lambda)$ is separable.  This is a
 consequence of the fact that  if $L/E/K$ is a tower of field extensions, with
 $L/K$ separable, then $E/K$ is separable (\emph{e.g.},
 ~\cite[Prop.~8, p.V.116]{bourbakifields}).

 \begin{rem}
  Given any field extension $L/\kappa(\Lambda)$,
  in \cite[Rem.~3.2]{ACMVabtriv} we recalled the pullback of $Z\in
  \operatorname{CH}^i(X_T)$ over the morphism $ t:\operatorname{Spec}
  L\to T$, defining $Z_L=Z_{\bar t}\in \operatorname{CH}^i(X_{\bar t})$.
  In particular, if $L'/L$ is a further field extension then, if $Z_L$ is
  algebraically trivial, then so is $Z_{L'}$.
 \end{rem}

 The following lemma describes Definition~\ref{D:AlgFun} in terms of more easily
 identified fibers\,:

 \begin{lem}\label{L:WhatPoints}
  Let $T$ be an integral scheme in $\mathsf {Sm}_\Lambda/S$ with generic point
  $\eta_T: \kappa(T) \to T$ and let $Z$ be a cycle class in
  $\operatorname{CH}^i(X_T)$.
  The separable algebraic closure $\eta_T^{\sep}$ of the generic point $\eta_T$
  can be obtained as an inverse limit of morphisms $(t_n:T_n \to T)_{n\in
   \mathbb
   N}$ in $\mathsf {Sm}_\Lambda/T$.

 \noindent Moreover, the following conditions are equivalent\,:
  \begin{enumerate}[label=(\roman*)]
   \item $Z\in \mathscr A_{X/S}^i(T)$\,;
   \item $Z_{\eta_T^{sep}} \in \operatorname{CH}^i(X_{\eta_T^{\sep}})$ is
   algebraically trivial\,;
   \item  $Z_t \in \operatorname{CH}^i(X_t)$ is algebraically trivial for all
   points $t : \spec \kappa(\Lambda)^{\sep} \to T$\,;
   \item[(iii)']  $Z_t \in \operatorname{CH}^i(X_t)$ is algebraically trivial
   for
   some point $t : \spec \kappa(\Lambda)^{\sep} \to T$\,;
   \item [(iv)]$Z_t \in \operatorname{CH}^i(X_t)$ is algebraically trivial for
   all
   dominant points $t : \spec \Omega \to T$ with $\Omega$ separably closed.
  \end{enumerate}
 \end{lem}

 \begin{proof}
  The separable algebraic closure $\eta_T^{\sep}$ of the generic point $\eta_T$
  can be obtained as an inverse limit of morphisms $(t_n:T_n \to T)_{n\in
   \mathbb
   N}$ in $\mathsf {Sm}_\Lambda/T$ (see \emph{e.g.}
  \cite[\href{https://stacks.math.columbia.edu/tag/037X}{Tag
   037X}]{stacks-project}), yielding the first assertion, as well as the
  implication $(i)\Rightarrow (ii)$.
  The implication $(iv) \Rightarrow (i)$ is obvious, while the
  equivalence $(ii)\Leftrightarrow (iv)$ follows from the fact that any point
  $t$
  as in $(iv)$ factors through $\eta_T^{\sep}$.
  Hence
  items $(i), (ii)$ and $(iv)$ are equivalent.
  Items $(iii)$ and $(iii)'$ are equivalent\,:  Since $T$ is smooth over
  $\Lambda$, its base-change
  $T_{\eta_\Lambda^{\sep}}$ is smooth over $\spec \kappa(\Lambda)^{\sep}$ and it
  follows by definition of
  algebraic equivalence that $Z_t$ and $Z_s$ are algebraically equivalent for
  any
  two points   $t : \spec \kappa(\Lambda)^{\sep} \to T$ and $s : \spec
  \kappa(\Lambda)^{\sep} \to T$.

  We now prove $(ii) \Rightarrow (iii)'$.
  By definition of algebraic equivalence, there exist a smooth variety $V$ of
  finite type over $\spec \kappa(T)^{\sep}$, points $v_0,v_1 \in  V(
  \kappa(T)^{\sep})$, and a correspondence $Y\in
  \operatorname{CH}^i(X_{\eta_T^{\sep}} \times_{ \kappa(T)^{\sep}} V)$ such that
  $Z_{\eta_T^{\sep}} = Y_{v_1} - Y_{v_0}$. The data consisting of $V,v_0,v_1$
  and
  $Y$ is defined over a finite (separable) extension $L$ of $\kappa(T)$.
  Spreading out, we find an \'etale morphism $\widetilde{T} \to T$, a smooth
  family $\widetilde{V} \to \widetilde{T}$ with two sections $\tilde{v}_0$ and
  $\tilde{v}_1$, and a cycle class $\widetilde{Y} \in
  \operatorname{CH}^i(X_{\widetilde{T}} \times_{\widetilde{T}} \widetilde{V})$,
  whose generic fibers are the models of $V,v_0,v_1$ and $Y$ over $L$.
  Consider now a point $t : \spec \kappa(\Lambda)^{\sep} \to T$ in the image of
  $\widetilde{T} \to T$, and consider any point $\tilde{t} :  \spec
  \kappa(\Lambda)^{\sep} \to \widetilde{T}$ in the pre-image.
  Then, under the natural identifications, $Z_t$ coincides with~$Z_{\tilde{t}}
  = (\widetilde{Y}_{\tilde{v}_1})_{\tilde{t}} -
  (\widetilde{Y}_{\tilde{v}_0})_{\tilde{t}} =
  (\widetilde{Y}_{\tilde{t}})_{(\tilde{v}_1)_{\tilde{t}}} -
  (\widetilde{Y}_{\tilde{t}})_{(\tilde{v}_0)_{\tilde{t}}} $, and so $Z_t$ is
  algebraically trivial.

  To conclude, we prove
  $(iii)' \Rightarrow (ii)$. This follows from the simple observation that
  $\eta_T^{\sep}$ provides a point $\spec \kappa(T)^{\sep} \to
  T_{\kappa(T)^{\sep}}$
  and that the base-change of a point  $t : \spec \kappa(\Lambda)^{\sep} \to T$
  along the field extension $\kappa(T)^{\sep}/\kappa(\Lambda)^{\sep}$ provides a
  point $\spec \kappa(T)^{\sep} \to T_{\eta_T^{\sep}}$\,; hence
  $Z_{\eta_T^{\sep}}$ is
  algebraically equivalent to  $(Z_{t_{\kappa(T)^{\sep}}}) =
  (Z_t)_{\kappa(T)^{\sep}}$, which is by assumption itself algebraically
  trivial.
 \end{proof}

 \begin{cor}
  Suppose that $\Lambda=\operatorname{Spec}K$ for a field $K$,
  that $T$ is an integral scheme in $\mathsf {Sm}_K/S$, and $Z\in \mathscr
  A^i_{X/S}(T)$.
  Then, for any morphism $t:\operatorname{Spec}\Omega\to T$ with $\Omega$
  separably closed, we have $Z_t\in \operatorname{CH}^i(X_t)$ is algebraically
  trivial.
 \end{cor}

 \begin{proof}
  This is immediate from the lemma.
 \end{proof}

 \subsection{Regular homomorphisms and algebraic representatives}
 We keep the notation and conventions of the previous subsection.

 \begin{dfn}[Regular homomorphism]\label{D:reghom}
  Let $A/S$ be an abelian scheme, viewed also as
  the contravariant functor $\operatorname{Hom}_S(-,A):\mathsf {Sm}_\Lambda/S\to
  \mathsf
  {AbGp}$.
  A \emph{regular homomorphism from $\mathscr A^i_{X/S}$ to $A/S$} is a natural
  transformation of functors $\Phi:\mathscr A^i_{X/S}\to A$.
 \end{dfn}

 Let us parse the definition.  Given $T$ in $\mathsf {Sm}_\Lambda/S$, we
 obtain
 $\Phi(T):\mathscr A^i_{X/S}(T) \to A(T)$\,; in other words, given a cycle class
 $Z\in \mathscr A^i_{X/S}(T)$, \emph{i.e.}, a family of algebraically trivial
 cycle
 classes on $X$ parameterized by $T$, we obtain a $S$-morphism $\Phi(T)(Z):T\to
 A$.

 \begin{dfn}[Algebraic representative]\label{D:fun}
  An \emph{algebraic representative in codimension-$i$}  consists of an abelian
  $S$-scheme $\operatorname{Ab}^i_{X/S}$ together with   a natural
  transformation
  of  functors
  $$\Phi^i_{X/S}:\mathscr {A}^i_{X/S} \to \operatorname{Ab}^i_{X/S}$$
  over $\mathsf {Sm}_\Lambda/S$
  that is initial among all natural transformations of  functors $ \Phi:\mathscr
  {A}^i_{X/S} \to A$ to abelian schemes $A/S$\,:
  \begin{equation}\label{E:AlgRepDefE}
  \begin{multlined}
  \xymatrix{
   \mathscr {A}^i_{X/S} \ar[r]^{\Phi^i_{X/S}} \ar[rd]_{\Phi}
   &\operatorname{Ab}^i_{X/S}\ar@{-->}[d]_{\exists !}  \\
   & A\\
  }
  \end{multlined}
  \end{equation}
  In particular, if an algebraic representative exists, it is unique up to
  unique isomorphism.
 \end{dfn}

 \begin{rem} Since abelian $S$-schemes, and morphisms of abelian $S$-schemes as $S$-schemes,
  form a
  full
  subcategory of  $\mathsf {Sm}_\Lambda/S$, 
 by Yoneda a natural transformation of
  functors
  $\operatorname{Ab}^i_{X/S}\to A$ is equivalent to a morphism of 
  $S$-schemes; from the commutativity of \eqref{E:AlgRepDefE}, this must send the zero section to the zero section.  In particular, the morphism $\operatorname{Ab}^i_{X/S}\to A$ in \eqref{E:AlgRepDefE} is induced by a homomorphism of abelian $S$-schemes.
 \end{rem}

 \subsection{Remarks on the definition of regular
  homomorphisms}\label{SS:rem}

 We feel it would be useful here to provide some discussion of how we arrived at
 Definition~\ref{D:reghom}.  In fact, there are many possible variations on the
 definition, which lead to slight variations on the theory, and we find it to be
 useful to mention some pathologies that led to our choices here.  An underlying
 theme is our desire to use Fulton's foundations for cycles and intersection
 theory \cite{fulton}, and to ensure that the Abel--Jacobi map defines a regular
 homomorphism in the relative setting.

 \subsubsection{Defining the category $\mathsf {Sm}_\Lambda/S$}\label{SSS:def}
 In order to be able to manipulate fibers of cycles, we wanted to work in a setting where we
 could use the refined Gysin fibers \cite[\S 20.1]{fulton}.  For this we need points to
 be
 regularly embedded in the ambient space, and so insisting on parameter spaces
 $T$ that are smooth separated and of finite type over $\Lambda$ is expedient.
 This leads to considering the category of $S$-schemes that are smooth separated and of finite type over $\Lambda$ (i.e., without the assumption that the structure map to $S$ be dominant).
  For the abstract theory, the functor $\mathscr{A}^i_{X/S/\Lambda}$ extends naturally to
 this larger category and in fact most of the structural results
 (precisely all results up to Theorem~\ref{T:basechangechar0} excluded, except for Corollary~\ref{C:surjdef})
 in this paper go through
 unchanged in that setting.  However, without the requirement that the structure morphisms to $S$ be
 \emph{dominant}, showing the existence of interesting (\emph{i.e.},
 nonzero) regular homomorphisms, which is a central step in showing the theory
 is
 nontrivial, becomes much harder. For instance,  one can see that any time one
 has $f:X\to S$ a smooth projective morphism of complex algebraic manifolds,
 then
 if the geometric coniveau in a fiber jumps,  the Abel--Jacobi map  (see \S
 \ref{S:jac}) would fail to
 be a regular homomorphism\,; see
 \cite[Ex.~6.5]{ACMVdcg} and compare with Theorem~\ref{T2:NF}.
 This is the essential reason we require
 that
 objects in $\mathsf {Sm}_\Lambda/S$ have dominant structure morphism to $S$.

 \subsubsection{Sheafification of $\mathscr A^i_{X/S}$}\label{S:sheafify}
 Even in the setting of codimension-$1$ cycle classes, our functor $\mathscr
 A^1_{X/S}$ differs from the Picard functor.  Namely, assuming $f:X \to S$ is flat, it could be natural to
 define a functor $$\mathscr P^i_{X/S/\Lambda}:=\mathscr
 A^i_{X/S/\Lambda}/f^*\mathscr A^i_{S/S/\Lambda},$$ where we use flat pull back
 of cycle classes.  For instance, $\mathscr P^1_{X/S/\Lambda}$ agrees with the
 relative  Picard functor $\mathscr Pic^0_{X/S}$ of algebraically trivial line
 bundles.
 Even in the case $i=1$, this need not be a sheaf, and so one might prefer
 further to sheafify this functor, in the hope that its sheafification
 $(\mathscr P^i_{X/S/\Lambda})^\dagger$ might be representable by
 an abelian  scheme in $\mathsf {Sm}_\Lambda/S$.  Unfortunately,
 for $i>1$, there are examples where this sheaf is far from being representable
 by an abelian scheme (\emph{e.g.}, Mumford's theorem that for any complex
 projective surface $S$ with $h^{2,0}(S)\ne 0$, the group $\operatorname{A}^2(S)$
 is infinite-dimensional\,; see also Remark \ref{R:mumford}).
 Since any regular homomorphism $\Phi:\mathscr A^i_{X/S}\to A$ factors uniquely
 through a natural transformation
 $$
 \xymatrix@R=1em{
  \mathscr A^i_{X/S}  \ar[r] \ar[d]_\Phi& \mathscr P^i_{X/S} \ar[r]& (\mathscr
  P^i_{X/S})^\dagger \ar@{-->}[lld]\\
  A
 }
 $$
 we prefer to work with the functor $\mathscr
 A^i_{X/S}$, which although (again) far from being representable in general, is
 easier to work with.
 In light of this discussion, we view the algebraic representative as being
 something of a categorical
 moduli space for $\mathscr A^i_{X/S}$ (and indeed also for $\mathscr P^i_{X/S}$
 and its sheafification).

 \begin{rem}
  While preparing this manuscript, we became aware of the recent work of
  Benoist--Wittenberg~\cite{BW}.  
 In the case where $S=\Lambda=\operatorname{Spec}K$ and where $X$ is a \emph{geometrically rational smooth projective threefold}, they define in \cite[\S 2.3.2]{BW}  a functor for codimension-$2$ cycle classes and show \cite[Thm.~3.1]{BW} that  this functor is a sheaf representable by a group scheme $\mathbf{CH}^2_{X/K}$ over $K$ whose connected component of the identity is 
an abelian variety.
In this setting,
  it would be
  interesting to compare  their sheaf with the sheaf we work with here, namely
  $(\mathscr P^2_{X/K})^\dagger$, and to compare their abelian variety $(\mathbf{CH}^2_{X/K})^\circ$
  with the algebraic representative we construct in Theorem~\ref{T2:mainAb2} for codimension-2 cycles classes.  For instance, if $K$ is perfect, they have
  shown  \cite[Thm.~3.1$(vi)$]{BW}  that the two algebraic representatives agree.
  On the other hand, as is mentioned in \cite[Rmk.~3.2(i)]{BW}, 
  defining an algebraic representative in the sense of Benoist--Wittenberg for codimension-2 cycles 
under weaker hypotheses on the variety $X$
   could prove useful in extending Theorem~\ref{T2:mainalgrep} in the case of codimension-2 cycles to non-separable field extensions.
 \end{rem}

 \subsection{The classical definition of regular homomorphisms and algebraic
  representatives} \label{SS:reg}

 We now show that over algebraically closed fields our notion of regular
 homomorphisms and algebraic representative agrees   with the usual one
 (\emph{i.e.},
 \cite[Def.~1.6.1]{murre83} or \cite[2.5]{samuelequivalence}).

 Let $k$ be an algebraically closed field, and take
 $\Lambda=S=\operatorname{Spec}k$. In that case, Lemma~\ref{L:WhatPoints} takes
 the simple form\,:

 \begin{lem}
  Let $T$ be a smooth scheme of finite type over $k$ and let $Z\in
  \operatorname{CH}^i(T\times_k X)$. Then $Z\in  \mathscr
  A^i_{X/k}(T)$ if and only if 
  $Z_t \in \operatorname{A}^i(X)$ for all $t\in T(k)$.\qed
 \end{lem}

 For a smooth integral separated scheme $T$ of finite type over $k$ and a cycle
 class
 $Z\in \mathscr {A}^i_{X/k}(T)$, we therefore have a map  $$w_{Z}:T(k)\to
 \operatorname{A}^i(X)$$
 $$
 w_Z(t)=Z_t.
 $$
 On the other hand,
 for a smooth integral separated $k$-pointed scheme $(T,t_0)$ of finite type
 over $k$ and
 a cycle class $Z'\in \operatorname{CH}^i(T\times_k X)$, we have a map
 $$w_{Z',t_0}:T(k)\to \operatorname{A}^i(X)$$
 $$w_{Z',t_0}(t)=Z'_t-Z'_{t_0}.$$
 Note that if  $Z\in \mathscr {A}^i_{X/k}(T)$, then viewing $Z\in
 \operatorname{CH}^i(T\times_k X)$, we have $w_{Z,t_0}(t)=w_Z(t)-Z_{t_0}$.
 Conversely\,:

 \begin{lem}[Cycle maps with and without base points]
  Let $A/k$ be an abelian variety, and let $\phi:\operatorname{A}^i(X)\to A(k)$
  be a homomorphism of groups.
  The following are equivalent\,:
  \begin{enumerate}[label=(\roman*)]
   \item For every smooth integral separated scheme  $T$ of finite type over $k$
   and every
   cycle class $Z\in \mathscr {A}^i_{X/k}(T)$, the induced  map of sets
   $T(k)\stackrel{w_Z}{\to} \operatorname{A}^i(k)\stackrel{\phi}{\to} A(k)$ is
   induced by a morphism of $k$-schemes $\psi_Z:T\to A$.

   \item For every smooth integral separated $k$-pointed scheme $(T,t_0)$ over
   $k$ and
   every
   cycle class $Z\in \operatorname{CH}^i(X\times T)$, the induced  map of sets
   $T(k)\stackrel{w_{Z,t_0}}{\to} \operatorname{A}^i(k)\stackrel{\phi}{\to}
   A(k)$
   is induced by a morphism of $k$-schemes $\psi_{Z,t_0}:T\to A$.
  \end{enumerate}
 \end{lem}

 \begin{proof}
  Assume $(i)$.  If $Z\in \operatorname{CH}^i(T\times_k X)$, set $Z':=Z-
  (Z_{t_0}\times_k T)$.  The cycle $Z'$ belongs to $\mathscr {A}^i_{X/k}(T)$ by
  Lemma~\ref{L:WhatPoints} and on sets we have  $\psi_{Z,t_0}=\phi \circ w_{Z,t_0}
  =
  \phi\circ w_{Z'}= \psi_{Z'}$. Thus since $\psi_{Z'}$
  is induced by a morphism of schemes, so is $\psi_{Z,t_0}$.  Conversely, given
  $Z\in \mathscr A^i_{X/k}(T)$, we have in particular  $Z\in
  \operatorname{CH}^i(T\times_k X)$.
  From the definitions, we have  $\psi_Z =
  \psi_{Z,t_0} + \phi(Z_{t_0})$, and we are done.
 \end{proof}

 The classical definition of a regular homomorphism is a homomorphism $\phi$ as
 in $(ii)$ of the lemma above.
 The following lemma shows that our definition of regular homomorphisms, and
 hence of  algebraic representatives,  agrees with the usual one when working
 over an algebraically closed field.

 \begin{lem}[Regular homomorphisms and cycle maps]
  Let $A/k$ be an abelian variety.

  \begin{enumerate}[label=(\roman*)]
   \item Given a natural transformation of contravariant functors
   $\Phi:\mathscr
   A^i_{X/k}\to A$, the group  homomorphism
   $\phi_{\Phi}:=\Phi(\operatorname{Spec}k):\mathscr
   A^i_{X/k}(\operatorname{Spec}k)=\operatorname{A}^i(X) \to A(k)$
   has the property that
   for every smooth integral separated  scheme $T$ over $k$ and every cycle
   class $Z\in
   \mathscr A^i_X(T)$, the induced  map of sets  $T(k)\stackrel{w_{Z}}{\to}
   \operatorname{A}^i(k)\stackrel{\phi_\Phi}{\to} A(k)$ is induced by a morphism
   of
   $k$- schemes $\psi_{Z}:T\to A$.

   \item Conversely,  given a  group  homomorphism  $\phi: \operatorname{A}^i(X)
   \to A(k)$ with
   the property that
   for every smooth integral   variety $T$ over $k$ and every cycle class $Z\in
   \mathscr A^i_X(T)$, the induced  map of sets  $T(k)\stackrel{w_{Z}}{\to}
   \operatorname{A}^i(k)\stackrel{\phi}{\to} A(k)$ is induced by a morphism of
   $k$-schemes $\psi_{Z}:T\to A$,
   there is a natural transformation of functors $\Phi_{\phi}:\mathscr
   A^i_{X/k}\to A$ defined by $\Phi_\phi(T)(Z):= \psi_Z$.  The functor on
   morphisms
   is defined \emph{via} the refined Gysin pull-back.
  \end{enumerate}
  Moreover, given $\Phi$ as in $(i)$, we have $\Phi_{\phi_\Phi}=\Phi$, and given
  $\phi$ as in $(ii)$, we have $\phi_{\Phi_\phi}=\phi$.
 \end{lem}

 \begin{proof}
  This follows directly from the definitions.
 \end{proof}

 \begin{exa}[Abel--Jacobi maps are regular homomorphisms]
  \label{Ex:AJ}
  In the classical case where $X$ is smooth projective over $S=\Lambda=\spec \cx$, the image of Griffiths'
  Abel--Jacobi map $AJ : \operatorname{A}^i(X) \to \mathrm{J}^{2i-1}(X)$
  restricted to
  algebraically trivial cycles defines a subtorus $\mathrm{J}^{2i-1}_a(X)$
  naturally
  endowed with a polarization and hence defines a complex abelian variety. It is
  a
  classical result of Griffiths that the induced Abel--Jacobi map $AJ :
  \operatorname{A}^i(X) \to \mathrm{J}^{2i-1}_a(X)$ defines a regular
  homomorphism. Additionally, in the
  case where $i=1,2,$ or $\dim X$, then $AJ$ is in fact the algebraic
  representative\,; see \cite{murre83} and in particular \cite[Thm.~C]{murre83},
  but also Theorem~\ref{T:MK} below.
  However, for $2<i<\dim X$, the Abel--Jacobi map is not in general an algebraic
  representative\,; \emph{cf.}~\cite[Cor.~4.2]{OttemSuzuki}. We refer to \S
  \ref{SS:AJequi}
  below for more details on intermediate Jacobians.
 \end{exa}

 \section{Regular homomorphisms and base change}

 \subsection{Base change in $\mathsf {Sm}_\Lambda/S$}\label{S:SmSbasechange}

 In the notation above, let  $
 \Phi:\mathscr A^i_{X/S}\to A
 $  be a regular homomorphism
 and let $S'\to S$ be a morphism in $\mathsf {Sm}_\Lambda/S$.
   Then  the
 regular
 homomorphism $\Phi:\mathscr A^i_{X/S}\to A$  induces naturally a regular
 homomorphism
 \begin{equation}\label{E:BC-1}
 \Phi_{S'}:\mathscr A^i_{X_{S'}/S'} \to A_{S'}.
 \end{equation}
 Indeed,
 the forgetful functor
 $
 (\mathsf {Sm}_\Lambda/S')\to (\mathsf {Sm}_\Lambda/S)
 $
 induces, \emph{via} fibered products of functors, a functor
 $$
 \xymatrix@R=1.5em{
  \left(\mathscr A^i_{X/S}\right)_{S'} \ar[r] \ar[d] & \mathscr A^i_{X/S}
  \ar[d]\\
  (\mathsf {Sm}_\Lambda/S')\ar[r]&  (\mathsf {Sm}_\Lambda/S).
 }
 $$
 The first claim is that
 \begin{equation}\label{E:BC-2}
 \left(\mathscr A^i_{X/S}\right)_{S'}=\mathscr A^i_{X_{S'}/S'}
 \end{equation}
 as functors over $\mathsf {Sm}_\Lambda/S'$.
 To begin, let us show that  for every $T'\to S'$ in $\mathsf {Sm}_\Lambda/S'$
 we
 have
 $$
 \left(\mathscr A^i_{X/S}\right)_{S'}(T')= \mathscr A^i_{X/S}(T')=\mathscr
 A^i_{X_{S'}/S'}(T').
 $$
 The first equality follows immediately from the definition of the fibered
 product of functors. For the second equality, 
 by definition, if $Z\in \mathscr
 {A}^i_{X/S}(T')$, then  $Z \in
 \operatorname{CH}^i(X_{T'})$ has the property that
 for  every  separably closed field $\Omega$ and every
 $\Omega$-point $t':
 \operatorname{Spec}(\Omega) \to  T'$ obtained as an inverse limit of morphisms
 $(t_n:T_n' \to T')_{n\in \mathbb N}$ in $\mathsf {Sm}_\Lambda/T'$, the cycle class $Z_{t'}\in \operatorname{CH}^i(X_{t'})$ is algebraically trivial.
 Now since
 $X_{T'}=(X_{S'})_{T'}$, unraveling the definition, we see that the previous sentence is also the condition that $Z\in \mathscr A^i_{X_{S'}/S'}(T')$.  The
 functors also agree on morphisms since they are both defined \emph{via} the refined
 Gysin pullback.

 Now for $T'\to
 S'$, the diagram
 $$
 \xymatrix@R=1em{
  T'\ar[rrd]^{\Phi(T')(Z)} \ar[ddr] \ar@{-->}[rd]&&&\\
  &A_{S'}\ar[r] \ar[d]& A\ar[d]\\
  &S'\ar[r] &S
 }
 $$
 provides a natural transformation
 \begin{equation}\label{E:BC-3}
 \Phi_{S'}:\left(\mathscr A^i_{X/S}\right)_{S'}\to A_{S'}.
 \end{equation}
 Combining \eqref{E:BC-2} with \eqref{E:BC-3} defines the natural transformation
 \eqref{E:BC-1}.

 \begin{lem} \label{L:basechange}
  Let $\Phi:\mathscr A^i_{X/S}\to A$ be a regular homomorphism and let $S''\to
  S'\to S$  be morphisms in $\mathsf{Sm}_\Lambda/S$. Then $(\Phi_{S'})_{S''} =
  \Phi_{S''}$.
 \end{lem}
 \begin{proof}
  This follows directly from the definitions and from the base-change
  construction.
 \end{proof}

 \subsection{Inverse limits and base change}\label{S:FunLim-1}

 Both of the functors $\mathscr A_{X/S}^i$ and $A$  send inverse limits of
 varieties to direct limits of
 abelian groups.
 Therefore, a  natural transformation of functors $$\Phi:\mathscr A_{X/S}^i \to
 A$$  extends in a
 canonical way to the category of schemes over $S$ that can be obtained as
 inverse limits  of schemes in $\mathsf {Sm}_\Lambda/S$.
 For example, if
 $\Lambda=S=\operatorname{Spec}K$
 for some field $K$, and
 if   $L/K$ is
 a separable extension of fields, using
 \cite[\href{https://stacks.math.columbia.edu/tag/037X}{Tag
  037X}]{stacks-project},
 we obtain a map $\Phi(L):\mathscr A_{X/K}^i(L) \to A(L)$.

 In fact, since in this case any scheme in $\mathsf {Sm}_L/L$ is an inverse
 limit
 of schemes  in $\mathsf {Sm}_K/K$, we obtain an induced regular homomorphism
 \begin{equation}\label{E:PhiL}
 \Phi_L:\mathscr A_{X_L/L}^i \to A_L
 \end{equation}
 over $\mathsf {Sm}_L/L$.  This is a special case of the following lemma\,:

 \begin{lem}\label{L:limBC}
  Let $\Phi:\mathscr A^i_{X/S/\Lambda}\to A$ be a regular homomorphism, and
  suppose we have a cartesian diagram
  $$
  \xymatrix@R=1em{
   S'\ar[r] \ar[d]& S \ar[d]\\
   \Lambda'\ar[r]& \Lambda
  }
  $$
  where $\Lambda'$ is an integral  regular Noetherian  scheme,
  and $S'$ is an integral scheme that is smooth separated and of finite type
  over
  $\Lambda'$, which is obtained as the inverse limit of cartesian diagrams
  $$
  \xymatrix@R=1em{
   S'_n\ar[r] \ar[d]& S \ar[d]\\
   \Lambda'_n\ar[r]& \Lambda
  }
  $$
  where $\Lambda'_n$ is smooth separated of finite type over $\Lambda$,
  and $S_{n}'\to \Lambda_n'$ is smooth separated of finite type, so that
  $S'_n\to
  S$ is smooth separated of finite type, and so in particular in $\mathsf {Sm}_\Lambda/S$.
  Every  scheme in $\mathsf {Sm}_{\Lambda'}/S'$ is an inverse limit of schemes
  in
  $\mathsf {Sm}_\Lambda/S$, and therefore $\Phi$ defines a regular homomorphism
  \begin{equation*}
  \Phi_{S'}:\mathscr A_{X_{S'}/S'/\Lambda'}^i \to A_{S'}
  \end{equation*}
  over $\mathsf {Sm}_{\Lambda'}/S'$.
 \end{lem}

 \begin{proof}
  We only need to show that every  scheme in $\mathsf {Sm}_{\Lambda'}/S'$ is an
  inverse limit of schemes in $\mathsf {Sm}_\Lambda/S$.  We may as well work
  with
  affine schemes.  So, let $\Lambda = \operatorname{Spec}Q$, $\Lambda_n'=\operatorname{Spec}Q'$, $\Lambda ' = \operatorname{Spec}Q$, $S = \spec R$, $S_n'=\operatorname{Spec}R_n'$,  and $S' = \spec R'$.  We have that $R'$ is 
  obtained as the direct limit of the rings $R'_n$, which are  smooth and finite type over the Noetherian ring $Q_n$.  Then any  ring of finite type over $R'$,
  say of the form $R'[x_1,\dots,x_r]/I$ for some ideal $I$, which is smooth over $Q'$, is the direct limit
  of
  the $R$-algebras $R_n'[x_1,\dots,x_r]/I_n$, $I_n= I\cap R_n'[x_1,\dots,x_n]$, which are smooth over $Q_n'$ for sufficiently large $n$, say using the Jacobian criterion.  Indeed, we may apply the direct limit to the short exact sequence $0\to I_n\to R_n'[x_1,\dots,x_r]\to R_n'[x_1,\dots,x_r]/I_n\to 0$, and use exactness of the direct limit.
 \end{proof}

 We will often use this in the following form\,:

 \begin{cor}
  Let $\Phi:\mathscr A^i_{X/S}\to A$ be a regular homomorphism, let $\Omega$ be
  a
  separably closed field, and let
  $s:
  \operatorname{Spec}(\Omega) \to  S$ be obtained as an inverse limit of
  morphisms $(s_n:S_n \to S)_{n\in \mathbb N}$ in $\mathsf {Sm}_\Lambda/S$.
  Every  scheme in $\mathsf {Sm}_{\Omega}/\Omega$ is an inverse limit of schemes
  in $\mathsf {Sm}_\Lambda/S$, and therefore $\Phi$ defines a regular
  homomorphism
  \begin{equation*}
  \Phi_\Omega:\mathscr A_{X_\Omega/\Omega/\Omega}^i \to A_\Omega
  \end{equation*}
  over $\mathsf {Sm}_{\Omega}/\Omega$.
 \end{cor}

 \begin{proof}
  This follows immediately from the lemma.
 \end{proof}

 \begin{cor}
  Given  two cartesian diagrams
  $$
  \xymatrix@R=1em{
   S'' \ar[r] \ar[d]& S'\ar[r] \ar[d]& S \ar[d]\\
   \Lambda '' \ar[r]&\Lambda'\ar[r]& \Lambda
  }
  $$
  as in Lemma \ref{L:limBC}, we have $(\Phi_{S'})_{S''}=\Phi_{S''}$.
 \end{cor}

 \begin{proof}
  This follows immediately from the construction.
 \end{proof}

 \subsection{Base-change and equivariant morphisms}
 Let $X\to S$ and $S'\to S$ be  morphisms of schemes.  Given the action of a
 group $\Gamma'$ on $S'$ over $S$,  this can be extended to an action of
 $\Gamma$ on $ X_{S'}$
 as follows.  For each $\sigma'\in \Gamma'$, the commutativity of the diagram
 $\xymatrix{S'\ar[r]_{\sigma'}\ar @/^.7pc/[rr]& S'\ar[r]& S}$ gives an
 isomorphism
 between $X_{S'}$ and the pull-back $(X_{S'})_{\sigma'}$ of $X_{S'}$ along
 $\sigma'$.
 In other words, for each $\sigma'\in \Gamma'$, we obtain a cartesian diagram
 \begin{equation}\label{E:BLR-GD}
 \xymatrix@R=1em{
  X_{S'} \ar[r]^{\sigma'} \ar[d]& X_{S'} \ar[d]\\
  S' \ar[r]^{\sigma'}& S'
 }
 \end{equation}
 giving the action of $\Gamma'$ on $X_{S'}$.

 In what follows, we will want to consider a filtered system of $S$-schemes
 $S_n/S$ (indexed by an arbitrary filtered set), together with an $S$-scheme $S'$
 which is the cofiltered limit (\emph{i.e.}, inverse
 limit) of the $S_n$.  We will then assume that a group $\Gamma'$ acts on the
 cofiltered system $S_n/S$, inducing an action on  $S'$.  Note that we do not
 assume that $\Gamma'$ acts on each $S_n$ individually.   This induces an action
 of $\Gamma'$ on the cofiltered system $X_{S_n}/S_n$, \emph{via} the commutativity of
 the diagrams of the form
 $\xymatrix{S_n\ar[r]_{\sigma'}\ar @/^.7pc/[rr]& S_m\ar[r]& S}$, and the same
 construction as given above for the action of $\Gamma'$ on $S'$.

 \begin{exa}\label{E:R-Sm/K-filt}
  The main example we have in mind is where $S=\operatorname{Spec}K$ for some
  field $K$, $S'=\operatorname{Spec}L$ for some separable field extension $L/K$,
  $S_n=\operatorname{Spec}R_n$ is the filtered system of smooth $K$-algebras $R_n$
  of finite type contained in $L$, and $\Gamma'=\operatorname{Aut}(L/K)$.  To see
  that the $R_n$ form a filtered system, we can argue as follows.
  Suppose we have $R$ and $R'$ smooth $K$-algebras of finite type contained in
  $L$.
  Consider the fraction fields $K(R),K(R')\subseteq L$.  The field
  $K(R)K(R')\subseteq L$ is  separable over $K$ \cite[\href{https://stacks.math.columbia.edu/tag/030P}{Lem.~030P}]{stacks-project}.
  Thus there is some smooth $K$-algebra of finite type $R''\subseteq K(R)K(R')$
  with $K(R'')=K(R)K(R')$
  \cite[\href{https://stacks.math.columbia.edu/tag/037X}{Lem.~037X}]{stacks-project}.
  The inclusion $K(R)\subseteq K(R'')$ induces a map of rings $R\to R''_{f}$ for
  some localization at $f\in R''$, and similarly for $R'$.  Thus replacing $R''$
  with the localization $R_{ff'}$, we see that the system $R_n$ is filtered.  Now
  since $L$ is the localization at $(0)$ of a smooth $K$-algebra of finite type
  over $K$
  \cite[\href{https://stacks.math.columbia.edu/tag/07ND}{Lem.~07ND,~037X}]{stacks-project},
  it is clear that $L=\varinjlim R_n$.   Next we explain how
  $\Gamma'=\operatorname{Aut}(L/K)$ acts on the filtered system.  For this, we
  simply observe that for $\sigma'\in \operatorname{Aut}(L/K)$, then for any
  smooth $K$-algebra $R$ of finite type contained in $L$, we have that
  $\sigma'(R)$ is again a smooth $K$-algebra of finite type contained in $L$.
 \end{exa}

 \begin{dfn}[Equivariant regular  homomorphisms]\label{D:equivariant}
  Suppose we are given a cartesian diagram
  $$
  \xymatrix@R=1em{
   S'\ar[r] \ar[d]& S \ar[d]\\
   \Lambda'\ar[r]& \Lambda
  }
  $$
  as in Lemma~\ref{L:limBC}, such that    the filtered systems
  $S_n \to S$ and $\Lambda_n\to \Lambda$ from Lemma~\ref{L:limBC} admit
  compatible
  (in the obvious way) actions of a group $\Gamma'$.
  For each $\sigma'\in \Gamma'$, the action of the group on  the filtered system
  $X_{S_n }$
  induces in the limit natural
  transformations
  $\sigma'^!:\mathscr A^i_{X_{S'}/S'}\to \mathscr A^i_{X_{S'}/S'}$ and $\sigma'
  :A_{S'}\to
  A_{S'}$, defined in the obvious way.   We say \emph{a regular homomorphism
   $$\Phi':\mathscr A^i_{X_{S'}/S'/\Lambda'}\to A_{S'}$$ is
   $\Gamma'$-equivariant}
  if for all $\sigma' \in \Gamma'$ the following diagram is
  commutative\,:
  $$
  \xymatrix@R=1em{
   \mathscr A^i_{X_{S'}/S'}  \ar[r]^{\Phi'} \ar[d]^{\sigma'^!}& A_{S'}
   \ar@{<-}[d]^{\sigma'}\\
   \mathscr A^i_{X_{S'}/S'} \ar[r]^{\Phi'}& A_{S'}
  }
  $$
 \end{dfn}

 \begin{rem}
  Note the vertical arrow $\sigma'$ in the diagram above.  One could replace this
  isomorphism with the isomorphism  $\sigma'^*=\sigma'^{-1}$ pointing down in the
  diagram, but we prefer to leave the diagram as it is above.
 \end{rem}

 \begin{lem} \label{L:GalBC}
  In the notation of Definition~\ref{D:equivariant} and Lemma~\ref{L:limBC}, if
  $\Phi:\mathscr A^i_{X/S/\Lambda}\to A$ is a regular homomorphism, then
  the regular homomorphism
  $\Phi_{S'}:\mathscr A^i_{X_{S'}/S'/\Lambda'}\to A_{S'}$ is
  $\Gamma'$-equivariant.
 \end{lem}

 \begin{proof}
  This follows directly from the definitions and from the base-change
  construction
  of \S \ref{S:SmSbasechange}.
 \end{proof}

 \begin{rem}[Galois descent data] Later we will discuss descent of regular
  homomorphisms along Galois field extensions.  To clarify how this interacts
  with
  this discussion here, recall that a finite faithfully flat morphism of schemes
  $p:S'\to S$ is called
  a Galois covering if there is a finite group $\Gamma$ of $S$-automorphisms of
  $S'$ such that the morphism
  $$
  \Gamma \times S' \longrightarrow S'\times_S S'
  $$
  $$
  (\sigma,x)\mapsto (\sigma x,x)
  $$
  is an isomorphism, where $\Gamma\times S'$ is the disjoint union of copies of
  $S'$ parameterized by $\Gamma$.  The standard example is where $K'/K$ is a
  finite Galois extension of fields,  $p:\operatorname{Spec}K'\to
  \operatorname{Spec}K$ is the associated map of schemes, and
  $\Gamma=\operatorname{Gal}(K'/K)$.
  Given an $S'$-scheme $X'$, a descent datum for $X'$ over $S$ is equivalent to
  an
  action $\Gamma \times X'\to X'$ compatible with the action of $\Gamma$ on
  $S$\,;
  \emph{i.e.}, for each $\sigma\in \Gamma$, we have a cartesian diagram
  \eqref{E:BLR-GD} \cite[Exa.~B, p.139]{BLR}.
  Recall that if $X$ is a scheme over~$S$, then $X_{S'}$ is equipped with a
  descent datum over $S$, as described above.
 \end{rem}

 \section{Base change and descent along field extensions} Here we assume that
 $S=\Lambda
 =
 \spec (K)$ for some field~$K$.
 Recall from \S\ref{S:FunLim-1} that  if   $L/K$ is
 a separable extension of fields,
 then
 a  regular homomorphism  $\Phi:\mathscr A_{X/K/K}^i
 \to A$  induces in a
 canonical way a regular homomorphism
 $\Phi_L:\mathscr A_{X_L/L/L}^i \to A_L$.
 Combining Lemmas~\ref{L:limBC} and~\ref{L:GalBC}, we have in this setting\,:

 \begin{lem} \label{L:equivariant}
  Let $\Phi:\mathscr A^i_{X/K/K}\to A$ be a regular homomorphism. Let $L/K$ be a
  separable extension of fields.  Then
  \begin{enumerate}[label=(\roman*)]
   \item $(\Phi_L)_{L'} = \Phi_{L'}$ for any tower of separable field extensions
   $L'/L/K$.
   \item
   $\Phi_L$ is $\mathrm{Aut}(L/K)$-equivariant.
  \end{enumerate}
 \end{lem}
 \begin{proof}
  $(i)$ is a consequence of Lemma \ref{L:basechange}, and an elementary
  modification of the proof of Lemma~\ref{L:limBC}.
  $(ii)$ is an immediate consequence of Lemma~\ref{L:GalBC} and
  Example \ref{E:R-Sm/K-filt}.
 \end{proof}

 \subsection{Separable field extensions of separably closed fields}
 \label{S:algclosed}

 We assume here that $\Omega/k$ is a regular (\emph{i.e.}, primary and
 separable) field
 extension, and that
 $k$ is separably closed. Equivalently, $\Omega/k$ is a separable extension of a
 separably closed field $k$.
 For instance, if $k$ is algebraically
 closed, $\Omega$ is any field extension of $k$. We are
 interested in understanding the relationship between regular homomorphisms over
 $k$ and over $\Omega$.

 If $L/K$ is any field extension, recall that an $L/K$-\emph{trace} of an
 abelian variety $B$ over $L$ is a final object $(\LKtrace{B},\tau)$ (sometimes
 also denoted $(\mathrm{tr}_{L/K}(B),\tau)$) in the category of pairs $(A,f)$
 where $A$ is an abelian variety over $K$ and $f: A_L \to B$ is homomorphism of
 abelian varieties over $L$. If $\Omega/k$ is any \emph{primary} extension, then
 an
 $\Omega/k$-trace exists \cite[Thm.~6.2]{conradtrace} and is unique up to unique
 isomorphism. Concretely, if $B$ is an abelian variety over $\Omega$ and $A$ an
 abelian variety over~$k$, any homomorphism $A_\Omega \to B$ factors uniquely as
 $A_\Omega \to \LKtrace{B}_\Omega \stackrel{\tau}{\to} B$ and there are thus
 canonical bijections
 $$\Hom_\Omega(A_\Omega,B) = \Hom_\Omega(A_\Omega,\LKtrace{B}_\Omega) =
 \Hom_k(A,\LKtrace{B}).$$
 We direct the reader to \cite[\S\S 3.3.1-3.3.2]{ACMVdcg} for a
 review of the $\Omega/k$-trace and Chow rigidity for primary
 extensions as exposited in \cite{conradtrace}.
 \medskip

 Let now $X$ be a  scheme of finite type over a separably closed field  $k$
 and let $\Omega/k$ be a separable extension.
 We start by observing as in \cite[Step 2, Thm.~3.7]{ACMVdcg} that if
 $\Psi:\mathscr A^i_{X_\Omega /\Omega}\to B$ is a regular homomorphism (over
 $\mathsf {Sm}_\Omega/\Omega$), and $(\LKtrace B,\tau : \LKtrace{B}_\Omega \to
 B)$ is the $\Omega/k$-trace of $B$,
 then there is an
 induced regular homomorphism \begin{equation}\label{E:L/KtrRegHom}
 \LKtrace{\Psi} :\mathscr A^i_{X /k}\to \LKtrace B
 \end{equation}
 (over $\mathsf {Sm}_k/k$)
 defined in the following way.
 Take $T$ in $\mathsf {Sm}_k/k$ which we may assume to be integral (since $k$ is
 separably closed, $T$ is a disjoint union of integral schemes) and $Z\in
 \mathscr A^i_{X/k}(T)$. Since $\Omega/k$
 is regular and in particular separable, $\Omega$ is an inverse limit of rings
 that are smooth and of finite
 type over $k$\,; we thereby
 obtain a morphism $\Psi(T_\Omega)(Z_\Omega):T_\Omega\to B$.  Let $t_0\in T(k)$
 be a base point (which exists because $k$ is separably closed
 and $T$ is smooth
 over $k$) and let $\operatorname{alb}_{t_0} : T_\Omega \to
 \operatorname{Alb}_{T_\Omega/\Omega}$ be the associated Albanese morphism over
 $\Omega$ (which classically exists\,; see, \emph{e.g.}, the discussion in \cite{ACMValb}).
 Since  by definition $\operatorname{alb}_{t_0} $ is initial among morphisms
 from $T_\Omega$ to abelian varieties over $\Omega$ sending $t_0$ to $0$, we
 obtain a commutative diagram
 $$
 \xymatrix@R=1.5em{
  T_\Omega \ar[r]^{\Psi(T_\Omega)(Z_\Omega)} \ar[d]_{\operatorname{alb}_{t_0}}&
  B\\
  \operatorname{Alb}_{T_\Omega/\Omega} \ar[ru]_{\eta_{t_0}}&
 }
 $$
 From \cite[Thm.~A]{ACMValb}, since $\Omega/k$ is
 separable, the canonical homomorphism $ \operatorname{Alb}_{T_\Omega/\Omega} \to
 (\operatorname{Alb}_{T_k/k})_\Omega$ is an isomorphism.
 By the definition of the $\Omega/k$-trace
 together with Chow rigidity for primary extensions, we thereby obtain a
 commutative diagram
 $$
 \xymatrix@R=1.5em{
  T\ar@{-->}[r] \ar[d]_{\operatorname{alb}_{t_0}}& \LKtrace B\\
  \operatorname{Alb}_{T/k} \ar[ru]_{\eta_{t_0}}&
 }
 $$
 where the dashed arrow is the composition of the other arrows.  It is easy to
 see that while $\operatorname{alb}_{t_0}$ and $\eta_{t_0}$ depend on $t_0\in
 T(k)$, their composition, the dashed arrow, does not.  This is the morphism we
 define as $\LKtrace \Psi(T)(Z):T\to \LKtrace B$, and this in turn defines the
 natural transformation $\LKtrace \Phi:\mathscr A^i_{X/k}\to \LKtrace B$ on
 objects.  The definition on morphisms is similar.

 \begin{lem}[Trace for regular homomorphisms]\label{L:Pull-Trace}
  In the notation above\,:
  \begin{enumerate}[label=(\roman*)]
   \item
   Let $\Phi:\mathscr A^i_{X/k}\to A$ be a regular homomorphism, and let
   $$
   \Phi_\Omega:\mathscr A^i_{X_\Omega/\Omega}\to A_\Omega
   $$
   be the regular homomorphism defined in \eqref{E:PhiL} above.
   Then we have $\LKtrace{\Phi_\Omega}=\Phi$.

   \item Let $\Psi:\mathscr A^i_{X_\Omega /\Omega}\to B$ be a regular
   homomorphism,
   and let
   $$
   \LKtrace{\Psi} :\mathscr A^i_{X_k /k}\to \LKtrace B
   $$
   be the regular homomorphism defined in \eqref{E:L/KtrRegHom} above.
   Then $(\LKtrace \Psi)_\Omega$ fits into a commutative diagram
   $$
   \xymatrix@R=1em{
    \mathscr A^i_{X_\Omega/\Omega}  \ar[r]^{(\LKtrace \Psi)_\Omega} \ar@{=}[d]&
    \LKtrace{B}_\Omega  \ar[d]^\tau \\
    \mathscr A^i_{X_\Omega/\Omega} \ar[r]^\Psi & B
   }
   $$
  \end{enumerate}
 \end{lem}

 \begin{proof}
  This is clear from the construction of the trace of a regular homomorphism and
  from the universal property of the trace of an abelian variety.
 \end{proof}

 \subsection{Galois extensions of fields}\label{S:Galois}

 Let $L/K$ be a Galois extension of fields.

 \begin{dfn}[Galois-equivariant regular
  homomorphism]\label{D:Galois-equivariant}
  Let $A$ be an abelian variety over $K$.   We say \emph{a regular homomorphism
   $\Psi:\mathscr A^i_{X_{L}/L}\to A_{L}$ is Galois-equivariant} if it is
  $\mathrm{Aut}(L/K)$-equivariant in the sense of
  Definition~\ref{D:equivariant}.
 \end{dfn}

 If  $A$ is an abelian variety over $K$, and $\Psi:\mathscr A^i_{X_{L}/L}\to
 A_{L}$
 is a  Galois-equivariant regular homomorphism  (of functors over $\mathsf
 {Sm}_L/L$), then there is a   regular homomorphism
 \begin{equation}\label{E:underKreg}
 \underline \Psi:\mathscr A^i_{X/K}\to A
 \end{equation}
 (over $\mathsf {Sm}_K/K$) defined as follows.  If $T$ is in $\mathsf {Sm}_K/K$
 and  $Z\in \mathscr A^i_{X/K}(T)$, then we obtain an induced morphism
 $\Psi(T_{L})(Z_{L}):T_L\to A_{L}$ that satisfies Galois descent (see
 \cite[Rmk.~4.3]{ACMVdcg}).  Thus there is an induced morphism $T\to A$, which
 we
 define to be $\underline \Psi(T)(Z)$.  This defines  $\underline \Psi:\mathscr
 A^i_{X/K}\to A$ on objects.  The definition on morphisms is similar.

 \begin{lem}[Regular homomorphisms and base change from $K$ to
  $L$]\label{L:Pull-Desc}
  Suppose $X$ is a scheme of finite type over $K$,
  $L/K$
  is a Galois field extension, and $A$ is an abelian variety over $K$.
  \begin{enumerate}[label=(\roman*)]
   \item If $\Phi:\mathscr A^i_{X/K}\to A$ is a regular homomorphism, then
   $\Phi_{L}:\mathscr A^i_{X_{L}/L}\to A_{L}$ is a  Galois-equivariant regular
   homomorphism and we have $\underline{\Phi_{L}}=\Phi$.

   \item Suppose $\Psi:\mathscr A^i_{X_{L}/L}\to A_{L}$
   is a  Galois-equivariant regular homomorphism  (over $\mathsf {Sm}_L/L$),
   and
   $\underline \Psi:\mathscr A^i_{X/K}\to A$
   is the induced regular homomorphism defined in \eqref{E:underKreg} above.
   Then
   we have $(\underline \Psi)_{L}=\Psi$.
  \end{enumerate}
 \end{lem}

 \begin{proof}
  This is clear from the construction of $\underline \Psi$ for a
  Galois-equivariant regular homomorphism~$\Psi$.
 \end{proof}

 \section{Surjective regular homomorphisms and miniversal cycles}

 \subsection{Surjective regular homomorphisms}\label{S:surjreghom}

 Note that since $\mathsf {AbGp}$ is an abelian category, the category of
 contravariant functors $\mathsf {Sm}_\Lambda/S\to \mathsf {AbGp}$  is also an
 abelian
 category.
 Thus there is a natural notion of kernel,  image, etc., for regular
 homomorphisms.  This is \emph{not} the definition we want to use for
 surjectivity. Rather, we would like to define surjectivity as meaning
 ``surjective when restricted to algebraically closed
 fields''.
 Since our results focus on certain separable extensions, we adopt the
 following definition\,:

 \begin{dfn}[Surjective regular homomorphism] \label{D:surj}
  A regular homomorphism $\Phi:\mathscr A^i_{X/S}\to A$ is called
  \emph{surjective} if
  $\Phi_\Omega(\Omega):\mathscr{A}_{X/S}^i({\Omega})\to A(\Omega)$ is
  surjective for all separably closed points $s : \spec \Omega \to S$ obtained
  as inverse limits of morphisms to $S$ in $\mathsf {Sm}_\Lambda/S$  (see
  Lemma~\ref{L:WhatPoints}).
 \end{dfn}

 In Proposition~\ref{prop:surjdef}, we will see that the surjectivity of a
 regular homomorphism can be tested at the separable closure of the generic
 point of $S$.
 A regular homomorphism that is an epimorphism of functors is a surjective
 regular homomorphism (\emph{e.g.}, Proposition~\ref{P:univ-epi} and
 Corollary~\ref{C:MinVZK}).
 However, there are examples of surjective regular homomorphisms that are not
 epimorphisms of functors (Remark~\ref{R:Voisin-No-Uni}).

 \subsection{Miniversal cycles for surjective regular homomorphisms}
 The following notion of \emph{miniversal cycle} for regular homomorphism will
 prove to be crucial to our understanding of surjectivity.

 \begin{dfn}[Universal and miniversal cycle classes]
  Let $\Phi:\mathscr A^i_{X/S}\to A$ be a regular homomorphism. Then a cycle
  class
  $Z\in \mathscr A^i_{X/S}(A)$ is called \emph{universal}
  (resp.~\emph{miniversal}) for $\Phi$ if the induced morphism $\Phi(A)(Z):A\to
  A$
  is the identity (resp.~an isogeny).  If $\Phi=\Phi^i_{X/S}$ is an algebraic
  representative, we will often drop the reference to $\Phi$, and simply call
  $Z$
  a universal (resp.~miniversal) codimension-$i$ cycle class.
 \end{dfn}

 The following proposition shows that the existence of a universal cycle
 characterizes regular homomorphisms that are epimorphisms of functors.

 \begin{pro}\label{P:univ-epi}
  A regular homomorphism  is an epimorphism of functors if and only if admits a
  universal cycle class.
 \end{pro}

 \begin{proof}
  Let $\Phi:\mathscr A^i_{X/S}\to A$ be the  regular homomorphism.
  First assume there is a universal cycle class $Z$.  Then given any morphism $g:T\to
  A$, we have that $\Phi(T)(g^!Z)=g$.  Thus $\Phi$ is an epimorphism of
  functors.
  Conversely, if $\Phi$ is an epimorphism of functors, then
  any cycle $Z\in\Phi(A)^{-1}(\operatorname{Id}_A)$ is a universal cycle class.
 \end{proof}

 \subsection{Existence of miniversal cycles} \label{subsec:existence}

 Essential to the proof of Lemma~\ref{L:MinVZK} is the following result,
 extracted from \cite{ACMVabtriv}, regarding the field of definition of schemes
 parameterizing algebraically trivial cycles\,:

 \begin{pro}[\cite{ACMVabtriv}] \label{P:IMRN}
  Let $X$ be a scheme of finite type over a field $K$, let  $K^{\sep}$ be a
  separable closure of $K$ and let $\alpha \in
  \operatorname{A}^i(X_{K^{\sep}})$ be an algebraically trivial cycle class.
  Then there exist a smooth quasi-projective curve $C$ over $K$, a cycle $Z\in
  \operatorname{CH}^i(C\times_K X)$,
  and  points $t_0, t_1 \in C({K}^{\sep})$, such that $$\alpha = Z_{t_1} -
  Z_{t_0}
  \quad \in \mathrm{A}^i(X_{K^{\sep}}).$$
  Moreover, if $C$ can be chosen to be smooth projective over $K$, then there
  exist an abelian variety $A$ over $K$, a cycle $Z'\in
  \operatorname{CH}^i(A\times_K X)$,
  and  a point $t \in A({K}^{\sep})$ such that $$\alpha = Z'_{t} - Z'_{0}
  \quad \in \mathrm{A}^i(X_{K^{\sep}}).$$
 \end{pro}
 \begin{proof}
  The existence of a smooth quasi-projective curve $C$ over $K$ is given by the
  equivalence of $(i)$ and $(ii)$  for separable extensions in
  \cite[Thm.~4.11]{ACMVabtriv}. The ``moreover statement'' can be proved along
  the
  lines of the proof of  \cite[Prop.~3.14]{ACMVabtriv}.
 \end{proof}

 For the record, we have the following classical results concerning extensions of
 morphisms to abelian schemes.

 \begin{pro}[{\cite[Cor.~6, \S 8.4]{BLR}, \cite[I.2.7]{FC}}]\label{P:BLR}
  Let $S$ be a Noetherian scheme and let $A$ be an abelian $S$-scheme.
  \begin{enumerate}[label=(\roman*)]
   \item \cite[Cor.~6, \S 8.4]{BLR}  Any rational $S$-morphism $T\dashrightarrow A$
   from an $S$-scheme $T$ to $A$ is defined everywhere if $T$ is regular.
   \item \cite[I.2.7]{FC} Suppose $S$ is normal and let $B$ be another abelian
   $S$-scheme. Any homomorphism $B_U \to A_U$ defined on some dense open subset
   $U\subseteq S$ extends to a homomorphism $B\to A$. \qed
  \end{enumerate}
 \end{pro}

 \begin{rem}
  We will in fact only use Proposition \ref{P:BLR}$(ii)$ \cite[I.2.7]{FC} in the
  case where all schemes are taken over a fixed base field $K$, and $S$ is assumed
  to be smooth over $K$\,; in this special case Proposition \ref{P:BLR}$(ii)$
  follows immediately from Proposition \ref{P:BLR}$(i)$ \cite[Cor.~6, \S 8.4]{BLR}
  and rigidity \cite[Prop.~6.1]{mumfordGIT}.
 \end{rem}

 The following lemma is crucial.

 \begin{lem}[{Existence of miniversal cycle classes}]
  \label{L:MinVZK}
  Suppose that $\Phi:\mathscr A^i_{X/S}\to A$ is a regular homomorphism.  Then
  there is an $S$-abelian subscheme $i:A'\hookrightarrow A$ such that $\Phi$
  factors uniquely
  as
  $$
  \xymatrix{
   \Phi :\mathscr A^i_{X/S}\ar@{->>}[r]^<>(0.5){\Phi'}& A' \ar@{^(->}^{i}[r]&
   A\\
  }
  $$
  where $\Phi'$ is a surjective regular homomorphism.
  Moreover, for any surjective regular homomorphism  $\Phi:\mathscr
  A^i_{X/S}\twoheadrightarrow  A$
  there exists  $Z\in \mathscr A^i_{X/S}(A)$ such that the induced
  $S$-morphism $\Phi(A)(Z):A\to A$ is  given by  $r\cdot
  \operatorname{Id}_{A}$
  for some natural number $r$.
 \end{lem}

 \begin{proof}
  In case $\Lambda=S= \spec K$ with $K$ an algebraically closed field, this
  is \cite[Lem.~1.6.2 and Cor.~1.6.3]{murre83}.
  This was generalized to the case $\Lambda=S= \spec K$ with $K$ a perfect field
  in
  \cite[Lem.~4.9]{ACMVdcg}.
  We note that in \cite{murre83} and in  \cite[Lem.~4.9]{ACMVdcg}, $X$ is
  assumed
  further to be smooth and projective over $K$.   This however is not necessary as is
  apparent in the proof given below   for $S= \spec K$ with $K$ any field.

  We now treat the case $S= \spec K$ with $K$ any field. Let $A'$ be the abelian
  sub-variety of $A$ which is obtained as the abelian sub-variety generated by
  all
  images $\Phi(D)(Z) : D\to A$,
  where $D$ runs through all abelian varieties over $K$ and $Z \in \mathscr
  A^i_{X/K/\Lambda}(D)$
  runs through all cycles such that $Z_0 = 0 \in
  \operatorname{CH}^i(X)$.
  We immediately note that there exists an abelian variety $D$ over $K$ and a
  cycle class $Z \in \mathscr A^i_{X/K}(D)$ with the property that $\Phi(D)(Z) :
  D\to A$ has image $A'$. Better, there exists such an abelian variety $D$ such
  that $\Phi(D)(Z) : D\to A$ defines an isogeny onto its image~$A'$.
  Considering then a $K$-homomorphism $A'\to D$ such that the composition $A'
  \to D \to A'$ is multiplication by a non-zero integer $r$, we obtain by
  pull-back a cycle $Z' \in \mathscr A^i_{X/K}(A')$ with the property that
  $\Phi(A')(Z') : A' \to A$ factors as $A' \stackrel{\cdot r}{\rightarrow} A'
  \hookrightarrow A$.  It is important to note (using Proposition~\ref{P:IMRN})
  that, if $\Phi$ is surjective, then $A'=A$\,; this will provide the
  uniqueness statement regarding the factorization of $\Phi$, and also the
  statement regarding the existence of miniversal cycles for surjective regular
  homomorphisms.

  We now claim that $\Phi$ factors through   $A'
  \hookrightarrow A$, thereby defining $\Phi':\mathscr A^i_{X/K/\Lambda}\to A'$.
  Let $\gamma \in A(K^{\sep})$ be in
  the image of $\Phi(K^\sep)$.
  Let $\xi \in \mathscr A^i_{X/K}(K^\sep) = \A^i(X_{K^{\sep}})$ be some
  algebraically trivial cycle
  class with $\Phi(K^\sep)(\xi) = \gamma$.  Then by Proposition~\ref{P:IMRN}
  there exist a smooth quasi-projective (but not
  necessarily complete) curve $C/K$, a
  cycle $Z \in \operatorname{CH}^i(C\times X)$, and points $t_0, t_1 \in
  C(K^{\sep})$ with $\Phi(Z_{t_1}-Z_{t_0})= \gamma$
  and
  $\Phi(Z_{t_0}-Z_{t_0})= 0_A$, the origin of~$A$.  To ease notation slightly,
  we
  let $\psi_Z = \Phi(C)(Z):C \to A$.
  If $C$ admits a \emph{smooth} projective model, then by
  Proposition~\ref{P:IMRN}
  the cycle $\xi$ is parameterized by an abelian variety over~$K$ and  we
  therefore already
  have $\psi_Z(C) \subset A'$ by construction of $A'$.
  Otherwise, we continue as follows.  Find
  $n$ sufficiently large that the Frobenius twist $C^{(p^n)}$ admits a
  smooth projective model ({\em e.g.} \cite[Lemma 1.2]{schroer09}).
  The iterated relative Frobenius $F^n_{C/K}:
  C \to C^{(p^n)}$ is a universal homeomorphism.
  Let $\til Z :=
  (F^n_{C/K}\times \operatorname{id})_* Z \in
  \operatorname{CH}^i(C^{(p^n)}\times
  X)$, and let $\psi_{\til Z} := \Phi(C^{(p^n)})(Z)$.  If $Q$ is a point of $C$
  and $\til Q = F^n_{C/K}(Q)$, then  the
  projection formula shows that there is an equality $\til Z_{\til Q} =
  p^n Z_Q$.  Consequently, we have a commuting diagram
  \[
  \xymatrix@R=1.8em{
   C \ar[rr]^{\psi_Z} \ar[d]^{F^n_{C/K}}  \ar@/^3ex/^{\bar\psi_Z}[rrrr]
   && A \ar[d]^{[p^n]_A} \ar[rr] &&
   \bar A \ar[d]^{[p^n]_{\bar A}}\\
   C^{(p^n)} \ar[rr]^{\psi_{\til Z}} \ar@/_3ex/_{\bar\psi_{\til Z}}[rrrr] &&A
   \ar[rr] && \bar A
  }
  \]
  where $\bar A$ is the quotient abelian variety $A/A'$.
  Since $C^{(p^n)}$ admits a smooth projective model, $\psi_{\til
   Z}(C^{(p^n)}) \subset A'$, and thus $\bar\psi_{\til Z}(C^{(p^n)}) =
  0_{\bar A}$.  Since $[p^n]_{\bar A}$ is a finite morphism,
  $\bar\psi_{Z}(C)$ is zero-dimensional, and $\psi_Z(C)$ is contained in a
  unique translate of $A'$.  Because $\psi_Z(C)$ passes through
  $0_A \in A'$, we must have $\psi_Z(C) \subset A'$.  In particular,
  $\gamma \in A'(K^\sep)$.

  Having completed the case $S=\operatorname{Spec}K$, we consider the general
  case.  This is achieved by spreading from the generic fiber.  First we recall
  two general facts about abelian schemes.  First,
  denoting $\eta_S$ the generic point of $S$, we recall that if $A$ is an
  abelian
  scheme over $S$ and if
  $(A')^\circ$
  is an $\eta_S$-abelian subvariety of $A_{\eta_S}$, then $(A')^\circ$ is the
  generic fiber of an abelian subscheme $A'$ of~$A$.
  Indeed, let
  $(A')^\circ$ be an abelian subvariety of $A_{\eta_S}$, where $A/S$
  is an abelian scheme
  over $S$. Then there is a
  homomorphism $f:A_{\eta_S} \to  A_{\eta_S}$ with image $(A')^\circ$.
  Since $S$ is normal, being  smooth over the regular scheme $\Lambda$,
  $f$ extends uniquely by  Proposition~\ref{P:BLR}$(ii)$ \cite[I.2.7]{FC}
  to a
  homomorphism $A \to
  A$. Its image
  defines an abelian subscheme $A'$ over $S$ whose generic fiber is
  $(A')^\circ$.
  Second we recall the elementary fact that if $f,g:A\to B$ are two morphisms of
  separated $S$-schemes that agree on the generic fiber, then they agree over $S$.

  Proceeding with the proof of the general case, we start with the regular
  homomorphism $\Phi:\mathscr A^i_{X/S/\Lambda}\to A$.
  Using Lemma~\ref{L:limBC}, we obtain a regular homomorphism
  $\Phi'_{\eta_S} : \mathscr{A}^i_{X_{\eta_S}/\eta_S/\eta_\Lambda} \to
  A_{\eta_S}$.  From the previous case, we see that there exist an abelian
  subvariety $A'_{\eta_S}\subseteq A_{\eta_S}$ and a cycle
  $Z_{\eta_S}  \in \mathscr{A}^i_{X_{\eta_S}/\eta_S/\eta_\Lambda}(A')$ that is
  miniversal  for $\Phi'_{\eta_S} :
  \mathscr{A}^i_{X_{\eta_S}/\eta_S/\eta_\Lambda} \to
  A'_{\eta_S}$\,; \emph{i.e.}, $\Phi'_{\eta_S}(A'_{\eta_S})(Z_{\eta_S}) :
  A'_{\eta_S} \to A'_{\eta_S}$ is multiplication by some
  natural number $r$.

  Using the first observation on abelian varieties above, we find that
  $A'_{\eta_S}$ is the generic point of a sub-abelian $S$-scheme $A'\subseteq
  A$.
  We next claim that $\Phi$ factors as
  $$
  \xymatrix{
   \Phi :\mathscr A^i_{X/S}\ar@{->}[r]^<>(0.5){\Phi'}& A' \ar@{^(->}^{i}[r]&
   A.
  }
  $$
  We define $\Phi'$ by restricting to the generic point, and then use
  Proposition~\ref{P:BLR}$(i)$
  \cite[Cor.~6, \S 8.4]{BLR} to get extension.  More precisely, given $T\to
  S$ in $\mathsf {Sm}_{\Lambda}/S$, and $\zeta\in \mathscr A^i_{X/S/\Lambda}(T)$,
  we
  define $\Phi'(T)(\zeta):T\to A'$ by extending
  $\Phi'_{\eta_S}(T_{\eta_S})(\zeta_{\eta_S}):T_{\eta_S}\to A'_{\eta_S}$. By
  construction, using our second observation about abelian varieties above, we
  have $\Phi(T)(\zeta):T\to A$ factors through $\Phi'(T)(\zeta)$.   This also
  provides the uniqueness of $A'$ and the factorization.

  Finally,  any spread of $Z_{\eta_S}$ to a cycle $Z \in
  \mathscr{A}^i_{X/S}(A')$ that restricts to $Z_{\eta_S}$
  induces an $S$-morphism $\Phi'(
  A')(Z) : A' \to  A'$ whose restriction to $\eta_S$
  is
  multiplication by $r$. In particular, from our second observation on abelian
  varieties, the $S$-morphism $\Phi'(
  A')(Z)$ is multiplication by~$r$.
  The surjectivity of $\Phi'$ can then be easily established using the
  miniversal cycle, since one can simply take appropriate $\Omega$-points of
  $A'$
  to prove surjectivity.
 \end{proof}

 \begin{cor}\label{C:MinVZK}
  A regular homomorphism is surjective if and only if it admits a miniversal
  cycle class.
 \end{cor}

 \begin{proof}
  If a regular homomorphism is surjective, it follows from Lemma~\ref{L:MinVZK}
  that it admits a miniversal cycle class.  Conversely, if a regular homomorphism
  admits
  a miniversal cycle class, it is immediate that it is surjective.
 \end{proof}

We include the following example for clarity.  
 \begin{exa}\label{Exa:NotIsoFunct}
 There are examples of surjective regular homomorphisms $\Phi:\mathscr A^n_{X/S}\to A$ that are epimorphisms of functors, and isomorphisms on points, but  are not isomorphisms of functors.
  For instance, take $X=\mathbb P^2_k$, for an algebraically closed field $k$.
  Then $\operatorname{Ab}^2_{X/k}=0$, the trivial abelian variety.  The morphism
  $\Phi^2_{X/k}: \mathscr A^2_{X/k} \to  \operatorname{Ab}^2_{X/k}$ is clearly a
  surjective regular homomorphism, which is an epimorphism of functors, and an
  isomorphism on $k$-points.  However, let  $T$ be a smooth projective curve of
  positive genus, let $\alpha\in \operatorname{A}_0(T)\cong
  {\operatorname{Alb}}_{T/k}$ be any nontrivial cycle class,  and take $Z=
  \operatorname{pr}_1^*\alpha\in  \operatorname{A}^2(T\times_k X)$, where
  $\operatorname{pr}_1$ is the first projection.
  We obtain a morphism $\Phi^2_{X/k}(T):\mathscr A^2_{X/k}(T)\to
  \operatorname{Ab}^2_{X/k}(T)=0$.   Since $Z\in \mathscr A^2_{X/k}(T)$ is
  non-trivial  (\emph{e.g.}, \cite[Thm.~3.3]{fulton}), we see that  $\Phi^2_{X/k}(T)$ is not an isomorphism, and therefore
  $\Phi^2_{X/k}$ is not an isomorphism  of functors.
  We note that in this special case of $X=\mathbb P^2_k$, there is an obvious functor that is representable by $\operatorname{Ab}^2_{X/k}$.  Namely,
if one considers the functor $\mathscr P^2_{X/k}$ given on $T$ by taking the quotient of $\mathscr A^2_{X/k}(T)$ by the classes pulled back from $T$, then $\mathscr P^2_{X/k}=\operatorname{Ab}^2_{X/k}$.  More generally, for a rational variety $X$ over a field $K$,
Benoist--Wittenberg have defined a functor of algebraically trivial cycle classes that is representable by $\operatorname{Ab}^2_{X/K}$ \cite{BW}; see \S\ref{S:sheafify}  for  more discussion of these topics.
\end{exa}

 \subsection{Surjective regular homomorphisms and base change}
 We have the general fact that the notion of surjectivity for regular
 homomorphisms is invariant under base-change\,:

 \begin{pro}\label{P:surjbasechange0}
  In the notation of Lemma~\ref{L:limBC},
  if  $\Phi:\mathscr A^i_{X/S/\Lambda}\to A$ is a surjective regular
  homomorphism, then
  the regular homomorphism $\Phi_{S'}:\mathscr A^i_{X_{S'}/S'/\Lambda'}\to
  A_{S'}$ is surjective.
 \end{pro}
 \begin{proof}
  Taking limits, we reduce to the case where $S' \to S$ is a  morphism in $
  \operatorname{Sm}_\Lambda/S$.
  By Corollary~\ref{C:MinVZK}, a regular homomorphism is surjective if and only
  if it
  admits a miniversal cycle. A miniversal cycle $Z\in \mathscr A^i_{X/S}(A)$ for
  $\Phi$ provides by base-change along $S'\to S$ a miniversal cycle $Z_{S'} \in
  \mathscr A^i_{X_{S'}/S'}(A_{S'})$ for $\Phi_{S'}$.
 \end{proof}

 \subsection{Surjective regular homomorphisms, revisited}

 In the proof of Lemma~\ref{L:MinVZK}, we only used that the regular
 homomorphism
 $\Phi:\mathscr A^i_{X/S}\to A$ is surjective on the separable closure of the
 generic point of $S$.
 The aim of this section is to show Proposition~\ref{prop:surjdef} (see also
 Corollary~\ref{C:surjdef} in the next subsection) saying not
 only that in
 Definition~\ref{D:surj} we could have required surjectivity only at those
 points, but in fact only at $\ell$-primary torsion points for some prime $\ell$
 different from the characteristic exponent of the generic point of $S$.
 \medskip

 First we have the general lemma\,:

 \begin{lem}[Chow groups and purely inseparable extensions]\label{L:rad}
  Let $X$ be a scheme of finite type over a field $K$ of characteristic exponent
  $p$ and let $L/K$ be a purely inseparable extension. Denote $f: X_L \to X$ the
  natural projection. Then the proper push-forward $f_*:
  \operatorname{CH}^i(X_L)
  \to \operatorname{CH}^i(X)$ and the flat pull-back $f^* :
  \operatorname{CH}^i(X)
  \to \operatorname{CH}^i(X_L)$ are isomorphisms on prime-to-$p$ torsion.
 \end{lem}
 \begin{proof}
  Indeed, if $L/K$ is finite of degree, say, $p^r$, then on Chow groups we have
  $f_*f^* = p^r$ and $f^*f_* = p^r$.
 \end{proof}

 Second we have the following lemma, which uses the existence of miniversal
 cycles (Lemma~\ref{L:MinVZK})\,:

 \begin{lem}\label{L:Tor}
  Let $\Phi:\mathscr A^i_{X/K}\to A$
  be a surjective regular homomorphism.
  Then there exists a natural number $r$ such that for any separable extension
  $\Omega/K$ that is separably closed, and for any $N$ invertible in
  $K$, the induced homomorphism
\[
\xymatrix@C+2pc{
\Phi_{\Omega,r}[N]: \A^i(X_{\Omega})[(r,N)N]
\ar[r]^-{\Phi_\Omega[(r,N)N]} & A(\Omega)[(r,N)N] \ar[r] & A(\Omega)[N]
}
\]
is surjective, where $(r,N) = \gcd(r,N)$.
 \end{lem}

 \begin{proof}
  The lemma is proved  in the case where $\Omega$ is algebraically closed in
  \cite[Lem.~3.2]{ACMVdmij}.  Let $p$ be the characteristic
  exponent of $K$.
  First we claim that if  $Z\in \operatorname{CH}^i(A_\Omega \times_\Omega
  X_\Omega)$, then the
  map $w_Z : A(\Omega) \to \operatorname{A}^i(X_\Omega)$, $a\mapsto Z_a-Z_0$ is
  a
  homomorphism on prime-to-$p$ torsion\,; more precisely, for each natural
  number
  $N$ coprime to~$p$, $w_Z$
  restricted to $A(\Omega)[N]$ gives a homomorphism $w_Z[N]:A(\Omega)[N]\to
  \operatorname{A}^i(X_\Omega)[N]$. Indeed, the map $w_Z$ factors as $A(\Omega)
  \stackrel{\tau}{\longrightarrow} \operatorname{A}_0(A_\Omega)
  \stackrel{Z_*}{\longrightarrow} \operatorname{A}^i(X_\Omega)$, where $\tau(a)
  =
  [a] - [0]$ and $Z_*$ is the group homomorphism induced by the action of the
  correspondence~$Z$. The claim then follows from the fact that $\tau_{\bar
   \Omega}
  : A(\bar \Omega) \to \operatorname{A}_0(A_{\bar \Omega}), a\mapsto [a]-[0]$ is
  an isomorphism on torsion by  \cite[Prop.~11]{Beau83}, where $\bar \Omega$
  denotes an algebraic closure of $\Omega$, from the fact that
  $A(\Omega)\hookrightarrow A(\Omega')$ is
  an isomorphism on prime-to-$p$ torsion for any purely inseparable extension
  $\Omega' /\Omega$, and from the fact that the base change
  homomorphism $ \operatorname{A}_0(A_{ \Omega}) \to
  \operatorname{A}_0(A_{\Omega'}) $ is an isomorphism
  on prime-to-$p$ torsion for any purely inseparable extension $\Omega' /\Omega$
  by Lemma~\ref{L:rad}.

  Second, let
  $Z\in \mathscr A^i_{X/K}(A)$ be a miniversal cycle provided by
  Lemma~\ref{L:MinVZK} such that the induced morphism
  $\Phi(A)(Z):A\to
  A$ is  given by  $r\cdot \operatorname{Id}_{A}$ for some natural
  number $r$.  Let $N$ be a natural number relatively prime to $p$.
  By
  the claim, the composition of group
  homomorphisms
  $$A(\Omega)[N] \stackrel{w_Z[N]}{\longrightarrow} \operatorname{A}^i(X_\Omega)[N]
  \stackrel{\Phi_{\Omega}[N]}{\longrightarrow} A(\Omega)[N]$$
is multiplication by $r$.  In particular, $\ker w_Z[N] \subset
A(\Omega)[(r,N)]$, and the composition $\Phi_{\Omega,r}[N]$ is
surjective.
 \end{proof}

 \begin{pro}\label{prop:surjdef}
  Let $\Phi:\mathscr A^i_{X/S}\to A$ be a regular homomorphism.
  Denote $\eta_S$ and $\eta_\Lambda$ the respective generic points of
  $S$ and $\Lambda$. Let $p$ be the characteristic exponent of $\eta_S$, and let $\Omega$ be a separable closure of $\kappa(\eta_S)$.
  The following conditions are equivalent\,:
  \begin{enumerate}[label=(\roman*)]
   \item $\Phi:\mathscr A^i_{X/S}\to A$ is surjective.
   \item $\Phi_{\Omega'}(\Omega'): \mathrm{A}^i(X_{\Omega'})\to A_{\Omega'}(\Omega')$
   is
   surjective for all
   separably closed points $s : \spec \Omega' \to S$ obtained
   as inverse limits of morphisms to $S$ in $\mathsf {Sm}_\Lambda/S$.

   \item $\Phi_{\Omega}(\Omega) : \mathrm{A}^i(X_{\Omega})\to A_\Omega(\Omega)$
   is
   surjective.
   \item $\Phi_{\Omega}(\Omega) : \mathrm{A}^i(X_{\Omega})\to A_\Omega(\Omega)$
   is
   surjective on $\ell$-primary torsion for some prime $\ell\neq p$.

      \item $\Phi_{\Omega}(\Omega) : \mathrm{A}^i(X_{\Omega})\to A_\Omega(\Omega)$ contains 
 the $\ell$-torsion in its image  for some prime $\ell\neq p$.
  \end{enumerate}
 \end{pro}
 
 \begin{proof} The equivalence of $(i)$ and $(ii)$ is simply
  Definition~\ref{D:surj}, while the implication $(ii)\Rightarrow (iii)$ is
  obvious (since $S$ is smooth over $\Lambda$, $\eta_S$ is separable over
  $\eta_{\Lambda}$\,; see Lemma~\ref{L:WhatPoints}). The implication
  $(iii)\Rightarrow (ii)$ follows from
  Lemma~\ref{L:Pull-Trace} and Proposition~\ref{P:surjbasechange0}. The
  implication $(iii) \Rightarrow (iv)$ is Lemma~\ref{L:Tor}, while the
  implication
  $(iv) \Rightarrow (iii)$ follows from the density of $\ell$-power torsion
  points
  in~$A$ and from the fact implied by Lemma~\ref{L:MinVZK} that the image of
  $\Phi_{\Omega}(\Omega)$ consists of the $\Omega$-points of an abelian
  subvariety
  of~$A$. Finally, the implication $(iv) \Rightarrow (v)$ is trivial, while $(v)$ implies $(iii)$ because the dimension of the abelian subvariety of $A$ which contains the image of $\Phi_\Omega(\Omega)$ can be read off from the rank of its $\ell$-torsion.
 \end{proof}

 \subsection{Surjective regular homomorphisms and descent}
 Propositions~\ref{P:surjtrace} and~\ref{P:surjGal} below show that surjectivity
 for regular homomorphisms descends in the
 situations described in \S\S \ref{S:algclosed} and~\ref{S:Galois}.\medskip

 First we show that the trace \eqref{E:L/KtrRegHom} of a
 surjective regular homomorphism is surjective.
 \begin{pro}\label{P:surjtrace}
  Let $k$ be a
  separably closed field
  and let $\Omega/k$ be a separable field
  extension.
  Suppose $\Psi:\mathscr A^i_{X_\Omega /\Omega}\to B$ is a surjective
  regular homomorphism. Then
  the regular homomorphism   $$
  \LKtrace{\Psi} :\mathscr A^i_{X/k}\to \LKtrace B
  $$
  defined in \eqref{E:L/KtrRegHom} is surjective.
 \end{pro}

 \begin{proof} Let $p$ be the characteristic exponent of $k$. First we claim
  that $\operatorname{A}^i(X) \to \operatorname{A}^i(X_K)$ is an isomorphism on
  prime-to-$p$ torsion for all extensions $K/k$ of separably closed fields.
  Lecomte's rigidity theorem \cite{lecomte86} states indeed that for an
  extension
  $K/k$ of algebraically closed fields and for any separated scheme $X$ of
  finite
  type over $k$, the base change homomorphism $\operatorname{CH}^i(X) \to
  \operatorname{CH}^i(X_K)$ is an isomorphism on torsion. In fact, the inverse
  is
  given by the specialization homomorphism (\emph{cf.}~\cite[proof of
  Thm.~3.8]{ACMVdcg})
  and so  the base change homomorphism $\operatorname{A}^i(X) \to
  \operatorname{A}^i(X_K)$ is also an isomorphism on torsion. Now one passes
  from
  algebraically closed fields to separably closed fields by utilizing
  Lemma~\ref{L:rad}.
  The proposition then follows from the characterization of surjectivity given
  in item $(iv)$ of Proposition~\ref{prop:surjdef} applied to the diagram of
  Lemma~\ref{L:Pull-Trace}$(ii)$, together with the above rigidity result for
  prime-to-$p$ torsion and the
  rigidity for prime-to-$p$ torsion points on abelian varieties for extensions
  of separably closed fields.   Note that a key point we are using when we
  employ
  Lemma~\ref{L:Pull-Trace}$(ii)$ is that for a regular extension of fields, the
  kernel of the trace is zero-dimensional, and supported at the identity
  \cite[Thm.~6.12]{conradtrace}.
 \end{proof}

 As a corollary of Proposition~\ref{P:surjtrace}, we get yet another
 characterization of surjective regular homomorphisms\,:
 \begin{cor}\label{C:surjdef}
  A regular homomorphism $\Phi:\mathscr A^i_{X/S}\to A$  is surjective if and
  only if one of the following conditions holds\,:

  \begin{enumerate}
   \item[$(v)$] $\Phi_{\Omega}(\Omega) : \mathrm{A}^i(X_{\Omega})\to
   A_\Omega(\Omega)$
   is
   surjective for some
   separably closed point $ s: \spec \Omega \to S$  obtained
   as inverse limits of morphisms to $S$ in $\mathsf {Sm}_\Lambda/S$.

   \item[$(vi)$] $\Phi_{\Omega}(\Omega) : \mathrm{A}^i(X_{\Omega})\to
   A_\Omega(\Omega)$
   is
   surjective on $\ell$-primary torsion for some prime $\ell\neq p$ and for some
   separably closed point $ s: \spec \Omega \to S$ obtained
   as inverse limits of morphisms to $S$ in $\mathsf {Sm}_\Lambda/S$.
  \end{enumerate}
 \end{cor}
 \begin{proof}
  Clearly if $\Phi$ is surjective, then it satisfies $(v)$. Conversely, since
  $s:
  \spec \Omega \to S$ in fact factors through the separable closure
  $\eta_S^{\sep}$
  of the generic point $\eta_S$ and since $\Omega/\kappa(S)^{\sep}$ is separable
  (e.g., \cite[Lem.~3.1]{ACMValb}),
  we can apply Proposition~\ref{P:surjtrace} to get
  that $\Phi_{\eta_S^{\sep}}(\eta_S^{\sep}) : \mathrm{A}^i(X_{\eta_S^{\sep}})\to
  A_{\eta_S^{\sep}}(\eta_S^{\sep})$ is surjective. We conclude with
  Proposition~\ref{prop:surjdef} that $\Phi$ is surjective. Finally, the
  equivalence of $(v)$ and $(vi)$ is proven in exactly the same way as the
  equivalence of $(iii)$ and $(iv)$ in Proposition~\ref{prop:surjdef}.
 \end{proof}

 Likewise, surjectivity for Galois-equivariant regular homomorphisms descends\,:

 \begin{pro}\label{P:surjGal}
  Suppose
  $L/K$
  is a Galois field extension, and $A$ is an abelian variety over $K$.
  Suppose $\Psi:\mathscr A^i_{X_{L}/L}\to A_{L}$
  is a  surjective Galois-equivariant regular homomorphism  (over $\mathsf
  {Sm}_L/L$). Then
  the induced regular homomorphism
  $$\underline \Psi:\mathscr A^i_{X/K}\to A$$
  defined in \eqref{E:underKreg} is surjective.
 \end{pro}

 \begin{proof}
  Since $(\underline \Psi)_{L}=\Psi$ by Lemma~\ref{L:Pull-Desc}$(ii)$, we have
  $(\underline \Psi)(K^{\sep}) = \Psi(K^{\sep}) :
  \operatorname{A}^i(X_{K^{\sep}})
  \to A(K^{\sep})$. Hence, by Proposition~\ref{prop:surjdef}, if $\Psi$ is
  surjective, then $\underline{\Psi}$ is also surjective.
 \end{proof}

 \section{Algebraic representatives}

 \subsection{Algebraic representatives are surjective regular homomorphisms}
 Recall from Definition~\ref{D:fun} that an \emph{algebraic representative} is a
 regular homomorphism that is initial among all regular homomorphisms.

 \begin{pro} \label{P:algrepsurj}
  If $\Phi^i_{X/S}:\mathscr A^i_{X/S}\to \operatorname{Ab}^i_{X/S}$ is an
  algebraic representative, then $\Phi^i_{X/S}$ is a surjective regular
  homomorphism.
 \end{pro}

 \begin{proof}
  This follows immediately from the definition and Lemma~\ref{L:MinVZK}.
 \end{proof}

 \begin{rem}\label{R:Voisin-No-Uni}
  A result of Voisin \cite{voisinUniv} shows that there exist algebraic
  representatives, and therefore surjective regular homomorphisms,  that do not
  admit universal cycles.
  Therefore, using Proposition~\ref{P:univ-epi}, Voisin's example also shows
  that
  there
  are surjective regular homomorphisms that are not epimorphisms of functors.
 \end{rem}

 \subsection{Saito's criterion for the existence of an algebraic representative}

 We can also use Lemma~\ref{L:MinVZK} to give a generalization of
 {\cite[Thm.~2.2]{hsaito} and \cite[Prop.~2.1]{murre83}} to the relative
 setting.

 \begin{pro}[Saito's criterion] \label{P:saito}
  An algebraic representative  $\Phi^i_{X/S}:\mathscr A^i_{X/S}\to
  \operatorname{Ab}^i_{X/S}$ exists if and only if there exists a natural number
  $M$ such that for every  surjective regular homomorphism $\Phi : \mathscr
  {A}^i_{X/S}\to A$,  we have $\dim_S A \leq M$.
 \end{pro}

 \begin{proof}
  The proof is formally the same as the argument  in \cite[Prop.~2.1]{murre83},
  except we use  Lemma~\ref{L:MinVZK} instead of  \cite[Lem.~1.6.2,
  Cor.~1.6.3]{murre83}.  We include the argument for completeness.

  We consider the category of surjective regular homomorphisms $\Phi:\mathscr
  A^i_{X/S}\to A$\,; we will denote these by   $(A,\Phi)$.  We set
  $\dim(A,\Phi)=\dim A$.  A morphism from $(A,\Phi)$ to $(A',\Phi')$ is   a
  morphism $\eta:A\to A'$ of abelian varieties such that $\Phi'=\eta \circ
  \Phi$.

  Let $M$ be the maximal dimension of the $(A,\Phi)$, and fix $(A_0,\Phi_0)$ of
  dimension $M$. Considering products, and
  using the fact that the image of a regular homomorphism is an abelian variety
  (Lemma~\ref{L:MinVZK}), it is easy to see that for any   $(A,\Phi)$ there
  exists
  some $(A',\Phi')$ of dimension $M$, which surjects  onto both $(A,\Phi)$ and
  $(A_0,\Phi_0)$  (set $A'$ to be the  image of the regular homomorphism
  $\Phi_0\times \Phi:\mathscr A^i_{X/K}\to A_0\times A$).

  Therefore, if $(A_0,\Phi_0)$ is not an algebraic representative,  there is an
  isogeny $\eta_1:(A_1,\Phi_1)\to (A_0,\Phi_0)$ of degree $>1$.  Thus,
  inductively,  if there is no algebraic representative, we can construct an
  isogeny $(A_N,\Phi_N)\to (A_0,\Phi_0)$ of arbitrarily large degree.  But this
  would then contradict the existence of the cycle $Z_0\in \mathscr
  A^i_{X/K}(A_0)$ such that $\Phi_0(A_0)(Z):A_0\to A_0$ is $r_0\cdot
  \operatorname{Id}_{A_0}$ for some natural number $r_0$,  established in
  Lemma~\ref{L:MinVZK}.
 \end{proof}

 \subsection{Algebraic representatives and base change}

 \begin{pro}\label{P:surjbasechange}
  In the notation of Lemma~\ref{L:limBC},
  if there exists an algebraic representative
  $$\Phi^i_{X_{S'}/S'}:\mathscr A^i_{X_{S'}/S'/\Lambda'}\to
  \operatorname{Ab}^i_{X_{S'}/S'},$$ then there
  exists an algebraic representative $\Phi^i_{X/S}:\mathscr A^i_{X/S/\Lambda}\to
  \operatorname{Ab}^i_{X/S}$.
 \end{pro}

 \begin{proof}
  This follows immediately from Proposition~\ref{P:surjbasechange0} and
  Proposition~\ref{P:saito}.
 \end{proof}

 \begin{cor}\label{C:surjbasechange}
  Let $s : \spec \Omega \to S$ be a separably closed point obtained
  as an inverse limit of morphisms to $S$ in $\mathsf {Sm}_\Lambda/S$.  If there
  exists an algebraic representative $\Phi^i_{X_{\Omega}/\Omega}:\mathscr
  A^i_{X_\Omega/\Omega/\Omega}\to \operatorname{Ab}^i_{X_\Omega/\Omega}$, then
  there exists an algebraic representative $\Phi^i_{X/S}:\mathscr
  A^i_{X/S/\Lambda}\to \operatorname{Ab}^i_{X/S}$.
 \end{cor}

 \begin{proof}
  This follows immediately from the proposition.
 \end{proof}

 \subsection{Algebraic representatives and descent along separable field
  extensions}

 \subsubsection{Algebraic  representatives and separable extensions of separably
  closed fields}
 With Lemma~\ref{L:Pull-Trace} and Proposition~\ref{P:surjtrace}, we can easily
 prove\,:

 \begin{teo}[Algebraic representatives and separable extensions of separably
  closed fields]\label{T:Omega/k}
  Let $k$ be a separably closed field, and $\Omega/k$ be any separable field
  extension.
  There
  is an
  algebraic representative
  $$
  \Phi^i_{X/k}:\mathscr A^i_{X/k}\to \operatorname{Ab}^i_{X/k}
  $$
  (over $\mathsf {Sm}_k/k$)
  if and only if there is an algebraic representative
  $$
  \Phi^i_{X_{\Omega}/\Omega}:\mathscr A^i_{X_\Omega/\Omega}\to
  \operatorname{Ab}^i_{X_\Omega/\Omega}
  $$
  (over $\mathsf {Sm}_\Omega/\Omega$).
  In addition, if the algebraic representatives exist, then the vertical arrows
  in
  the diagrams below
  \begin{equation}\label{E:StLKtraceThm}
  \begin{multlined}
  \xymatrix@C=5em@R=1em{
   \mathscr {A}_{X/k}^i \ar[r]^<>(0.5){\Phi^i_{X/k}} \ar@{=}[d]&
   \operatorname{Ab}^i_{X/k} \ar@{-->}[d]  &\mathscr {A}^i_{X_\Omega/\Omega}
   \ar[r]^<>(0.5){\Phi^i_{X_\Omega/\Omega}} \ar@{=}[d]&
   \operatorname{Ab}^i_{X_\Omega/\Omega} \ar@{-->}[d]\\
   \mathscr {A}_{X/k}^i\ar[r]^<>(0.5){\LKtrace{\Phi}^i_{X_\Omega/\Omega}}&
   \LKtrace{\operatorname{Ab}}^i_{X_\Omega/\Omega} &\mathscr
   {A}^i_{X_\Omega/\Omega} \ar[r]^<>(0.5){(\Phi^i_{X/k})_\Omega}&
   (\operatorname{Ab}^i_{X/k})_\Omega
  }
  \end{multlined}
  \end{equation}
  defined by the universal property of the algebraic representative with respect
  to the regular homomorphisms in the bottom row are isomorphisms. Moreover,
  denoting $\tau$ the trace morphism, the natural morphism obtained as the
  composition $( \operatorname{Ab}^i_{X/k})_\Omega \longrightarrow (
  \LKtrace{\operatorname{Ab}}^i_{X_\Omega/\Omega})_\Omega
  \stackrel{\tau}{\longrightarrow}  \operatorname{Ab}^i_{X_\Omega/\Omega}$ is an
  isomorphism.
  In particular, the trace for an algebraic representative is an isomorphism.
 \end{teo}

 \begin{proof}
  The first statement of the theorem regarding existence follows from
  Proposition~\ref{P:surjbasechange}, together with the conjunction of Saito's criterion
  (Proposition~\ref{P:saito}) with Proposition~\ref{P:surjtrace}.

  The proof that the vertical arrows of \eqref{E:StLKtraceThm} are isomorphisms
  is
  \emph{via} the commutative diagrams\,:
  \begin{equation}\label{E:PfLKtraceThm}
  \begin{multlined}
  \xymatrix@C=5em@R=1em{
   \mathscr {A}_{X/k}^i \ar[r]^<>(0.5){\Phi^i_{X/k}} \ar@{=}[d]&
   \operatorname{Ab}^i_{X/k} \ar@{-->}[d]
   &\mathscr {A}^i_{X_\Omega/\Omega} \ar[r]^<>(0.5){\Phi^i_{X_\Omega/\Omega}}
   \ar@{=}[d]&  \operatorname{Ab}^i_{X_\Omega/\Omega} \ar@{-->}[d]\\
   \mathscr {A}_{X/k}^i\ar[r]^<>(0.5){\LKtrace{\Phi}^i_{X_\Omega/\Omega}}
   \ar@{=}[d]&  \LKtrace{\operatorname{Ab}}^i_{X_\Omega/\Omega} \ar@{-->}[d]
   &\mathscr {A}^i_{X_\Omega/\Omega} \ar[r]^<>(0.5){(\Phi^i_{X/k})_\Omega}
   \ar@{=}[d]&  (\operatorname{Ab}^i_{X/k})_\Omega \ar@{-->}[d]\\
   \mathscr {A}_{X/k}^i \ar[r]^<>(0.5){\Phi^i_{X/k}}  &
   \operatorname{Ab}^i_{X/k}
   &
   \mathscr A^i_{X_\Omega/\Omega}
   \ar[r]^{\left(\LKtrace{\Phi}^i_{X_\Omega/\Omega}\right)_\Omega} \ar@{=}[d] &
   \left(\LKtrace {\operatorname{Ab}}^i_{X_\Omega/\Omega}\right)_\Omega
   \ar[d]^\tau \\
   &
   &\mathscr {A}^i_{X_\Omega/\Omega} \ar[r]^<>(0.5){\Phi^i_{X_\Omega/\Omega}} &
   \operatorname{Ab}^i_{X_\Omega/\Omega} \\
  }
  \end{multlined}
  \end{equation}
  where on the left-hand side of \eqref{E:PfLKtraceThm} we have applied the
  $\Omega/k$-trace to the right-hand side of \eqref{E:StLKtraceThm} (and used
  Lemma~\ref{L:Pull-Trace}$(i)$) and addended
  it to the bottom of the left-hand side of \eqref{E:StLKtraceThm}, and on the
  right-hand side of \eqref{E:PfLKtraceThm}  we have applied the pull back to
  the
  left-hand side of \eqref{E:StLKtraceThm}, addended it to the bottom of the
  right
  hand side of \eqref{E:StLKtraceThm}, and then finally addended the
  $\Omega/k$-trace from  Lemma~\ref{L:Pull-Trace}$(ii)$.

  The theorem then follows by the universal property of the algebraic
  representative
  (the composition of the right vertical arrows of both diagrams is the
  identity)
  and the fact that the trace morphism $\tau$ has finite kernel
  \cite[Thm.~6.4(4)]{conradtrace}. Indeed, from the former, one finds that the
  top
  vertical dashed arrows of both diagrams are injective homomorphisms. By
  base-change, the injectivity of the top vertical arrow on the left-hand side
  diagram implies that the middle vertical dashed arrow of the right-hand side
  diagram
  is also injective. From the latter, one finds that $ \left(\LKtrace
  {\operatorname{Ab}}^i_{X_\Omega/\Omega}\right)_\Omega $ and $
  \operatorname{Ab}^i_{X_\Omega/\Omega}$ have the same dimension. It follows that
  the
  injective vertical dashed arrows of the  right-hand side diagram are
  isomorphisms, and one also concludes that $\tau$ is an isomorphism. Finally,
  we
  also get that $\operatorname{Ab}^i_{X/k}$ and $\LKtrace
  {\operatorname{Ab}}^i_{X_\Omega/\Omega}$ have the same dimension (since their
  base-changes to $\Omega$ have the same dimension), and we conclude that the
  vertical
  dashed arrows of the left-hand side diagram are isomorphisms.
 \end{proof}

 \begin{rem}
  This strengthens the results in \cite[Thm.~3.7]{ACMVdcg} in two ways.
  First, the vertical arrows of both diagrams of \eqref{E:PfLKtraceThm} and the
  natural morphism $( \operatorname{Ab}^i_{X/k})_\Omega \to
  \operatorname{Ab}^i_{X_\Omega/\Omega}$ were only shown to be purely
  inseparable
  isogenies\,; they are in fact isomorphisms. Second, in
  \cite[Thm.~3.7]{ACMVdcg} we
  did not make precise the relationship
  among the regular homomorphisms $\Phi^i_{X/K}$, $(\Phi^i_{X,k})_{\Omega}$,
  $\Phi^i_{X_\Omega/\Omega}$, and $\LKtrace{\Phi}^i_{X_\Omega/\Omega}$.
 \end{rem}

 \subsection{Algebraic  representatives and Galois field extensions}
 We now consider base change of field where $L/K$ is an algebraic Galois field
 extension.
 First recall from \eqref{E:PhiL} and Lemma~\ref{L:limBC} that if
 $\Phi:\mathscr A^i_{X/K}\to A$ is a
 regular homomorphism (of functors over $\mathsf {Sm}_K/K$), then
 the induced regular homomorphism
 $\Phi_{L}:\mathscr A^i_{X_{L}/L}\to A_{L}$ (over $\mathsf {Sm}_L/L$) is
 Galois-equivariant in the sense of Definition~\ref{D:Galois-equivariant}.
 Conversely,
 if  $A$ is an abelian variety over $K$, and $\Psi:\mathscr A^i_{X_{L}/L}\to
 A_{L}$
 is a  Galois-equivariant regular homomorphism  (of functors over $\mathsf
 {Sm}_L/L$), then there is a   regular homomorphism
 \begin{equation*}
 \underline \Psi:\mathscr A^i_{X/K}\to A
 \end{equation*}
 (over $\mathsf {Sm}_K/K$) defined in \eqref{E:underKreg}.

 \begin{pro}\label{P:algrepGalois}
  Let $L/K$ be an algebraic
  Galois field extension.  An algebraic representative
  $\Psi^i_{X_{L}/L}:\mathscr
  A^i_{X_{L}/{L}}\to
  \operatorname{Ab}^i_{X_{L}/{L}}$, if it exists, is necessarily
  Galois-equivariant. Moreover, it is initial among all Galois-equivariant
  regular
  homomorphisms.
 \end{pro}
 \begin{proof}
  This is \cite[Thm.~4.4]{ACMVdcg}.
 \end{proof}
 With Lemma~\ref{L:Pull-Desc} and Proposition~\ref{P:surjGal} we can prove the
 following theorem\,:

 \begin{teo}[Algebraic representatives and Galois field
  extensions]\label{T:barK/K}
  Let $L/K$ be a Galois field extension.
  There is an algebraic representative
  $$
  \Phi^i_{X/K}:\mathscr A^i_{X/K}\to \operatorname{Ab}^i_{X/K}
  $$
  (over $\mathsf {Sm}_K/K$)
  if and only if there is a (Galois-equivariant) algebraic representative
  $$
  \Psi^i_{X_{L}/L}:\mathscr A^i_{X_{L}/{L}}\to
  \operatorname{Ab}^i_{X_{L}/{L}}
  $$
  (over $\mathsf {Sm}_L/{L}$).
  In addition, if both exist, then the vertical arrows in the diagram below,
  \begin{equation}\label{E:DescThm}
  \begin{multlined}
  \xymatrix@C=5em@R=1em{
   \mathscr {A}_{X/K}^i \ar[r]^<>(0.5){\Phi^i_{X/k}} \ar@{=}[d]&
   \operatorname{Ab}^i_{X/K} \ar@{-->}[d]  &\mathscr {A}^i_{X_{L}/{L}}
   \ar[r]^<>(0.5){\Psi^i_{X_{L}/{L}}} \ar@{=}[d]&
   \operatorname{Ab}^i_{X_{L}/{L}}
   \ar@{-->}[d]\\
   \mathscr {A}_{X/k}^i\ar[r]^<>(0.5){ \underline {\Psi}^i_{X_{L}/{L}}}&
   \underline {\operatorname{Ab}}^i_{X_{L}/{L}} &\mathscr {A}^i_{X_{L}/{L}}
   \ar[r]^<>(0.5){(\Phi^i_{X/K})_{L}}&  (\operatorname{Ab}^i_{X/K})_{L},
  }
  \end{multlined}
  \end{equation}
  defined by the universal property of the (Galois-equivariant) algebraic
  representative with respect to  the (Galois-equivariant) regular homomorphism
  in
  the bottom row,
  are isomorphisms.

 \end{teo}

 \begin{proof}
  The first statement of the theorem regarding existence follows from
  Proposition~\ref{P:algrepGalois}, Proposition~\ref{P:surjbasechange}, together
  with the combination of Saito's criterion (Proposition~\ref{P:saito}) with
  Proposition~\ref{P:surjGal}.
  One then establishes \eqref{E:DescThm} exactly as in the proof of
  Theorem~\ref{T:Omega/k}, using a modified version of
  Diagram~\eqref{E:PfLKtraceThm}.
 \end{proof}

 \subsection{Algebraic representatives and separable field extensions}
 Theorems~\ref{T:Omega/k} and~\ref{T:barK/K} can be combined to provide a proof
 of Theorem~\ref{T2:mainalgrep}, which we recall below, concerning descent of
 algebraic representatives along separable extensions of fields.

 \begin{teo}[Theorem~\ref{T2:mainalgrep}]  \label{T:mainalgrep}
  Let   $\Omega/K$ be a separable field extension.
  Then an algebraic
  representative
  $\Phi^i_{X_\Omega} : \mathscr A^i_{X_\Omega/\Omega/\Omega} \to
  \mathrm{Ab}^i_{X_\Omega/\Omega}$ exists if and only if an algebraic
  representative $\Phi^i_{X} : \mathscr A^i_{X/K/K} \to \mathrm{Ab}^i_{X/K}$
  exists. If this is the case, we have in addition\,:
  \begin{enumerate}[label=(\roman*)]
   \item $\mathrm{Ab}^i_{X_\Omega/\Omega}$ identifies with
   $(\mathrm{Ab}^i_{X/K})_\Omega$ \emph{via} the natural homomorphism
   $\mathrm{Ab}^i_{X_\Omega/\Omega}\to (\mathrm{Ab}^i_{X/K})_\Omega$ induced by
   the
   regular
   homomorphism $(\Phi^i_{X})_\Omega : \mathscr A^i_{X_\Omega/\Omega/\Omega} \to
   (\mathrm{Ab}^i_{X/K})_\Omega$
   and the universal property of~$\Phi^i_{X_\Omega}$\,;
   \item
   $\Phi^i_{X_\Omega}(\Omega) : \mathrm{A}^i(X_\Omega) \to
   \mathrm{Ab}^i_{X_\Omega/\Omega}(\Omega)$ is
   $\mathrm{Aut}(\Omega/K)$-equivariant, relative to the above identification.
  \end{enumerate}
 \end{teo}
 \begin{proof}  Let $\Omega^{\sep} \supseteq K^{\sep}$ be separable closures of
  $\Omega $ and $K$\,; note that $\Omega^{\sep} \supseteq K^{\sep}$  is
  separable \cite[Lemma~3.1]{ACMValb}.
  That an algebraic representative $\Phi^i_{X_\Omega}$ exists if and only if an
  algebraic representative $\Phi^i_{X}$ exists is obtained by applying
  Theorem~\ref{T:barK/K} to the extension $\Omega^{\sep}/\Omega$,
  Theorem~\ref{T:Omega/k} to the extension $\Omega^{\sep}/K^{\sep}$
  and
  Theorem~\ref{T:barK/K} to the extension $K^{\sep}/K$.
  The identification
  $\mathrm{Ab}^i_{X_\Omega/\Omega}\stackrel{=}{\longrightarrow}
  (\mathrm{Ab}^i_{X/K})_\Omega$ comes from the isomorphisms provided by the
  vertical dotted arrows of the right-hand squares of~\eqref{E:StLKtraceThm}
  and~\eqref{E:DescThm}. The $\mathrm{Aut}(\Omega/K)$-equivariance of
  $\Phi^i_{X_\Omega}(\Omega)$ is then given by Lemma~\ref{L:equivariant}.
 \end{proof}

 \begin{rem}
  Slightly more is true\,: Let $\sigma : \Omega_1 \simeq \Omega_2$ be an
  isomorphism of separable fields over~$K$
  Assume that
  $\mathrm{Ab}^i_{X_{\Omega_1}}/{\Omega_1}$ exists.
  Then $\mathrm{Ab}^i_{X_{\Omega_2}}/{\Omega_2}$  exists,
  $\mathrm{Ab}^i_{X_{\Omega_j}}/\Omega_j$ identifies with
  $(\mathrm{Ab}^i_{X/K})_{\Omega_j}$ for $j=1,2$, and relative to these
  identifications
  $\mathrm{A}^i(X_{\Omega_1}) \to
  \mathrm{Ab}^i_{X_{\Omega_1}/{\Omega_1}}({\Omega_1})$ is mapped to
  $\mathrm{A}^i(X_{\Omega_2}) \to
  \mathrm{Ab}^i_{X_{\Omega_2}/{\Omega_2}}({\Omega_2})$ \emph{via} $\sigma$.
 \end{rem}

 \section{Existence results for algebraic representatives}\label{S:existence}

 We now prove the existence statement of
 Theorem~\ref{T2:mainAb2} concerning algebraic representatives for cycles of dimension 0, codimension 1 and codimension 2. In particular, in the codimension-2 case, this generalizes
 Murre's result concerning the existence of algebraic representatives for
 codimension-$2$ cycles
 \cite[Thm.~A]{murre83} on smooth projective varieties defined over
 an algebraically closed field.

 \begin{teo}\label{T:mainAb2}
  Suppose $X$ is a
  scheme of finite type over $S$ with geometric generic fiber
  $X_{\bar{\eta}_S}$,  and let $i$ be a non-negative integer. Suppose
  there exists some $\ell$, invertible in $\kappa(\eta_S)$, such that
  the $\ell$-torsion group $\A^i(X_{\bar \eta_S})[\ell]$ is finite.
  Then there exists an algebraic representative
  $$
  \Phi^i_{X/S}:\mathscr A^i_{X/S/\Lambda}\longrightarrow
  \operatorname{Ab}^i_{X/S}.
  $$
In particular, if $X \to S$ is smooth and proper, then an algebraic representative $
\Phi^i_{X/S}:\mathscr A^i_{X/S/\Lambda}\longrightarrow
\operatorname{Ab}^i_{X/S}$ exists for $i=1,2$ and $\dim_SX$.
 \end{teo}

 \begin{proof}
   In the case where $X$ is a smooth projective scheme over $S=\spec
  K$ for some algebraically closed field $K$, the theorem is classical  and is due to
  Murre~\cite[Proof of Thm.~A]{murre83}.
  We have suitably generalized Saito's criterion and the existence of
  miniversal
  cycles in order to generalize Murre's argument to the general setting where $S$
  is any separated smooth scheme of finite type over a regular Noetherian scheme
  (rather than the spectrum of an algebraically closed field) and where $X$ is
 any scheme of finite type over~$S$ (rather than smooth and projective).
  Via  Corollary~\ref{C:surjbasechange},
  the theorem reduces to $S$ being the spectrum of a separably closed field.
  By Saito's criterion of Proposition~\ref{P:saito}, we have to show the
  existence of a natural number
  $M$ such that for every  surjective regular homomorphism $\Phi :
  \mathscr
  {A}^i_{X_{\eta_S}/\eta_S}\to A$,  we have $\dim_{\eta_S} A \leq
  M$. Let $R$ be the rank of $\A^i(X_{\bar\eta_S})[\ell]$.  It
  suffices to show that,  for
  each such $\Phi$ and $A$, $\#A(\eta_S^\sep)[\ell]$ has rank at most
  $R$.  Note that our hypothesis implies that, for every $e \ge 1$,
  $\#A(\eta_S^\sep)[\ell^e]$ has rank $R$.

To this end, let $\Phi :
  \mathscr
  {A}^i_{X_{\eta_S}/\eta_S}\to A$ be a surjective regular
  homomorphism. Let
  $Z\in \mathscr A^i_{X_{\eta_S}/\eta_S}(A)$ be a miniversal cycle
  provided by Lemma~\ref{L:MinVZK} such that the induced morphism
  $\Phi(A)(Z):A\to
  A$ is  given by  $r\cdot \operatorname{Id}_{A}$ for some natural number
  $r$. By Lemma~\ref{L:Tor}, there is a surjective homomorphism
  $\A^i(X_{\eta_S^\sep})[\ell^{\operatorname{ord}_\ell(r)+1}] \twoheadrightarrow
  A(\eta_S^\sep)[\ell]$, which shows that the latter indeed has rank at most
  $R$.

Now assume $X \to S$ is smooth and proper; we need to show that for
$i=1,2$ and $\dim_SX$, $\A^i(X_{\bar{\eta}_S})[\ell]$ is finite.
The case $i=1$ follows from the identification of $\A^1(X_{\bar{\eta}_S})[\ell]$ with $\pic^0_{X_{\bar{\eta}_S}}[\ell]$, the case $i=2$ from the finiteness of $\chow^2(X_{\bar{\eta}_S})[\ell]$ due to Merkurjev and Suslin~\cite{MS} (see also Theorem~\ref{T:MS} below), and the case $i=\dim_SX$ from Roitman's theorem~\cite{bloch79}.  (In the projective case, the latter can be found in \cite[Thm.~4.1]{bloch79} or
   \cite[Prop.~\ref{ACMVBlochMap:P:Rojtman}]{ACMVBlochMap}, combined with the fact, \emph{e.g.}
   \cite[Prop.~3.11]{ACMVabtriv}, that numerical and algebraic equivalence agree
   on  for $0$-cycles on $X_\Omega$. The proper case follows from Chow's lemma and the facts that
   $\chow_0$ and $\operatorname{Alb}$ are birational invariants for smooth proper
   varieties.)
 \end{proof}

\begin{rem}
In Theorem~\ref{T:AlbAbSe}, we will establish the existence of the algebraic representative $\Phi^d_{X/S}$ under the assumptions that $X \to S$ is proper with reduced and connected geometric fiber and that the composition $X\to S \to \Lambda$ is smooth. We also mention that, in the codimension-1 case, one can further establish the existence of the algebraic representative $\Phi^1_{X/S}$ under the weaker assumption that the geometric generic fiber of $X\to S$ admits an open embedding in a smooth proper scheme over $\bar{\eta}_S$. In that case, it can indeed be established, \emph{via} the localization exact sequence for Chow groups, that, for $\ell$ invertible in $\kappa(\eta_S)$,
the $\ell$-torsion group $\A^i(X_{\bar \eta_S})[\ell]$ is finite.
\end{rem}

 \section{Relation to the Picard scheme and the Albanese scheme}
 \label{S:PicAlb}

The aim of this section is to compare the algebraic representatives for codimension-1 cycles and for dimension-0 cycles 
to the relative Picard scheme and to the relative Albanese scheme, respectively (Theorems~\ref{T:AlgRep-Co1} and~\ref{T:AlbAb}). 
In light of the importance of the (non-)existence of universal cycle classes (\emph{e.g.}, \cite{voisinUniv, ACMVdiag}), we also discuss the existence of  universal codimension-$1$ cycle classes in Corollary~\ref{C:UnivCodim1}.
In Theorem~\ref{T:AlbAbSe} we generalize the existence result for algebraic representatives for dimension-0 cycles of Theorem~\ref{T:mainAb2} to $S$-schemes $X$ that are separated, geometrically reduced, geometrically connected and of finite type over $S$. In relation to \S \ref{SSS:def}, we note that all results of this section hold in the setting where the structure morphisms to $S$ of parameter spaces are not necessarily dominant.

 \subsection{Algebraic representatives for codimension-$1$ cycles}

 \subsubsection{Picard schemes and regular homomorphisms}

    Recall from Theorem \ref{T:mainAb2} that if the morphism  $f:X\to S$ is smooth and proper, then there exists an algebraic representative
  $
  \Phi^1_{X/S}:\mathscr A^1_{X/S/\Lambda}\to
  \operatorname{Ab}^1_{X/S}.
  $ We now compare this to the relative Picard scheme\,:

 \begin{teo}[Picard schemes and algebraic representatives for divisors]\label{T:AlgRep-Co1}
  Suppose $f:X\to S$
  is smooth and projective with geometrically integral fibers, and  $R^2f_*\mathcal O_X=0$.
  \begin{enumerate}[label=(\roman*)]
\item  The Abel--Jacobi map $AJ^1:\operatorname{Div}_{X/S}^0\to
  \operatorname{Pic}^0_{X/S}$ for divisors
  induces a
  regular homomorphism $$\Phi_{X/S}^1:\mathscr A^1_{X/S}\to
  \operatorname{Pic}^0_{X/S},$$ which is an algebraic representative in
  codimension-$1$.

  \item
For any morphism $S'\to S$ obtained as an inverse limit of morphisms in $\mathsf {Sm}_\Lambda/S$,
the natural homomorphism of $S'$-abelian
  schemes $$ \operatorname{Ab}^1_{X_{S'}/S'} \to (
  \operatorname{Ab}^1_{X/S})_{S'}$$ is an isomorphism.

   \item  For any separably closed
  point
  $s : \spec \Omega \to S$  obtained
  as an inverse limit of morphisms to $S$ in $\mathsf {Sm}_\Lambda/S$,
  the homomorphism
  $\Phi^1_{X/S}(\Omega)$ is an isomorphism\,; in particular, the map
  \begin{equation*}
  \xymatrix@C=4em{
   T_l \operatorname{A}^1(X_\Omega) \ar@{->}[r]^<>(0.5){T_l
    \Phi^1_{X/S}(\Omega)}_\cong
   & T_l \operatorname{Ab}^1_{X/S}(\Omega).\\
  }
  \end{equation*}
  is an isomorphism for all primes $l$. 
\end{enumerate}

 \end{teo}

 \begin{proof}
The proof follows readily from the standard results in say \cite{kleimanPic}.
 \end{proof}

 \begin{rem}\label{R:Pic0Kred}
 If $\Lambda=S=\operatorname{Spec}K$, and $X/K$ is proper 
 and geometrically normal,  then
  $(\operatorname{Pic}^0_{X/K})_{\operatorname{red}}$ is an abelian variety
  \cite[Cor.~VI.3.2]{FGA} (see also \cite[Rem.~9.5.6, Thm.~9.5.3, Cor.~9.4.18.3]{kleimanPic} and \cite[Prop.~VI.3.1]{FGA}), and the  Abel--Jacobi map $AJ^1:\operatorname{Div}_{X/S}^0\to
  \operatorname{Pic}^0_{X/K}$ for divisors still
  induces a
  regular homomorphism $$\Phi_{X/K}^1:\mathscr A^1_{X/K}\to
  (\operatorname{Pic}^0_{X/K})_{\operatorname{red}},$$ which is an algebraic
  representative in
  codimension-$1$,  as our parameter spaces are reduced, and therefore every
  morphism from
  a reduced scheme to  $\operatorname{Pic}^0_{X/K}$ factors uniquely through $
  (\operatorname{Pic}^0_{X/K})_{\operatorname{red}}$.
 \end{rem}

 \subsubsection{Some remarks on universal codimension-$1$ cycle classes}\label{S:UnivCodim1}
Bearing in mind that the (non-)existence of universal cycle classes is a subtle and interesting invariant (\emph{e.g.}, \cite{voisinUniv, ACMVdiag}), we note\,:

\begin{cor}[Universal codimension-$1$ cycle classes]\label{C:UnivCodim1}
Assume  $f:X\to S$
  is smooth and projective with geometrically integral fibers, and  $R^2f_*\mathcal O_X=0$.
If    $f$ admits a section, then $\operatorname{Ab}^1_{X/S}$ admits a universal
  codimension-$1$ cycle class.
\end{cor}

\begin{proof}
The proof follows readily from the standard results in say \cite{kleimanPic}.
%
\end{proof}

\begin{rem}\label{R:Pic0/KUniv}
 Assume $\Lambda=S=\operatorname{Spec}K$, and $X/K$ is proper 
and geometrically normal.  If $X$ admits a $K$-point, then $\operatorname{Ab}^1_{X/K}=(\operatorname{Pic}^0_{X/K})_{\operatorname{red}}$ admits a universal codimension-$1$ cycle class.
\end{rem}

\begin{rem}
 Even when working with a smooth projective geometrically integral  variety $X$
 over a perfect field $K$ there are some subtleties regarding universal
 codimension-$1$ cycle classes.  
For instance, if $K=\rat_p$ and 
$X/K$ is any curve of genus one with $X(K) =
\emptyset$,  
then $\operatorname{Ab}^1_{X/K}$ does not admit a universal
 codimension-$1$ cycle class.
 At the same time, we note that in general, the existence of a $K$-point on a smooth projective geometrically integral variety is not
 a
 necessary condition for the existence of a universal codimension-$1$ cycle class\,; \emph{e.g.}, any  smooth projective geometrically integral variety over a finite field admits a universal codimension-$1$ cycle class. 
\end{rem}

 \subsection{Algebraic representatives for dimension-$0$ cycles}
We denote by
   $d=d_{X/S}$ the relative dimension of $X$ over $S$.

 \subsubsection{Albanese schemes and regular homomorphisms}
 \label{SS:AlbRegHom}
 Let $X\to S$ be a proper morphism. An $S$-morphism $g:X\to A$ to an
 abelian scheme induces a regular homomorphism $$\Phi_g:\mathscr
 A^d_{X/S/\Lambda}\to A.$$
 Conversely, suppose in addition that  $X$ belongs to $\mathsf{Sm}_\Lambda/S$, \emph{i.e.}, $X$ is smooth over $\Lambda$ and $X\to S$ is dominant and proper,
  and we assume the generic fiber of $X/\Lambda$ is geometrically connected, and  $X/S$ admits a section $\sigma:S\to X$.  Then the cycle
 $Z=\Delta-(X\times_S \sigma(S))$ is
 in $\mathscr A^d_{X/S/\Lambda}(X)$,
and for any regular
 homomorphism $\Phi:\mathscr A^d_{X/S/\Lambda}\to A$,
 we obtain a morphism $\Phi(X)(Z):X\to A$.
   These
 constructions are inverses of one another, in the sense that  if $g$ sends $\sigma$ to $0$, then
 $\Phi_g(X)(Z)=g$, and
 $\Phi_{\Phi(X)(Z)}=\Phi$.   This gives the following lemma.  Recall
 that if $X/S$ is a scheme over $S$, with section $\sigma:S\to X$,
 then a \emph{pointed $S$-Albanese morphism} is an initial
 $S$-morphism $X\to A$ to an abelian $S$-scheme sending $\sigma$ to
 $0$\,; \emph{i.e.}, $S\to X\to A$ is the zero section ; and the
 abelian scheme is called the pointed $S$-Albanese scheme of $X$.

 \begin{lem} \label{L:AlbAb}
 Let $X$ be in $\mathsf {Sm}_\Lambda/S$ and assume that $X/\Lambda$ has generic fiber that is geometrically connected.
  Assume that $X/S$ is proper and admits a section $\sigma$, and set
  $Z=\Delta-(X\times_S\sigma(S))$ where $\Delta$ is the diagonal on $X\times_S X$.
  If $$\Phi^d_{X/S}:\mathscr A^d_{X/S/\Lambda}\longrightarrow
  \operatorname{Ab}^d_{X/S}$$ is an algebraic representative, then the
  associated
  map $\Phi^d_{X/S}(X)(Z):X\to \operatorname{Ab}^d_{X/S}$ is a pointed
  $S$-Albanese morphism sending
  $\sigma$ to zero, with $\Phi_{\Phi^d_{X/S}(X)(Z)}=\Phi^d_{X/S}$.   Conversely,
  if $$\alpha_\sigma:X\longrightarrow  \operatorname{Alb}_{X/S}$$ is a
  pointed $S$-Albanese morphism
  sending $\sigma$ to zero, then $\Phi_{\alpha_\sigma}:\mathscr
  A^d_{X/S/\Lambda}\to \operatorname{Alb}_{X/S}$ is an algebraic representative
  with $\Phi_{\alpha_\sigma}(X)(Z)=\alpha_\sigma$.
 \end{lem}

 \begin{proof}
 This follows easily from the discussion above.
%
%
 \end{proof}

 We have the following consequence, generalizing results of
 Serre--Grothendieck--Conrad on the existence of  Albanese varieties over
 fields (\cite{serrealb}; see also the discussion in \cite[Appendix]{wittenberg08} and \cite[\S 1]{ACMValb}), and also Grothendieck's relative
 existence
 result (which requires  that  $\mathscr Pic^0_{X/S}$ be represented by an
 abelian scheme\,; see Theorem \ref{T:AlbAb}, below).

 \begin{teo}[Albanese schemes and algebraic representatives for dimension-$0$ cycles]\label{T:AlbAbSe}
  Let $\Lambda$ be a regular Noetherian  scheme, let $S$ be  a scheme that is
  smooth
  separated and of finite type over $\Lambda$, and let $f:X\to S$ be a
  proper surjective
  morphism such that the composition $X\to S \to \Lambda$ is smooth.
  If the generic fiber of $X/\Lambda$ is
  geometrically connected,   then there is an algebraic representative
  $\Phi^d_{X/S}:\mathscr A^d_{X/S/\Lambda}\to
  \operatorname{Ab}^d_{X/S}$.
  If moreover $X/S$ admits a section, then there exists a pointed
  $S$-Albanese morphism $X\to
  \operatorname{Alb}_{X/S}$.
 \end{teo}

 \begin{proof} This is a straightforward application of Lemma \ref{L:AlbAb}.
%
%
 \end{proof}

 \begin{rem}
  The existence of the algebraic representative in Theorem \ref{T:AlbAbSe} follows from Theorem
  \ref{T:mainAb2} in the case where $f:X\to S$ is smooth and proper.
   \end{rem}

 \subsubsection{Grothendieck's Albanese schemes and algebraic representatives for dimension-$0$ cycles}
Recall from Theorem \ref{T:mainAb2} that if the morphism  $f:X\to S$ is smooth and proper, then there exists an algebraic representative
  $
  \Phi^d_{X/S}:\mathscr A^d_{X/S/\Lambda}\to
  \operatorname{Ab}^d_{X/S}$.  We also have Theorem \ref{T:AlbAbSe} above, comparing this with a relative Albanese scheme.
   We now compare this to Grothendieck's construction\,:

 \begin{teo}[Grothendieck's Albanese schemes and algebraic representatives for points]\label{T:AlbAb}

 Suppose that $f:X\to S$
  is smooth and projective with geometrically integral fibers, $R^2f_*\mathcal O_X=0$, and $f$ admits a section.
  \begin{enumerate}[label=(\roman*)]
\item The universal line bundle $\mathcal P$ on
  $\operatorname{Pic}^0_{X/S}\times_S X$, distinguished by being trivialized
  along
  the section, determines a pointed $S$-Albanese morphism $X\to
  \operatorname{Alb}_{X/S}:=(\operatorname{Pic}^0_{X/S})^\vee$,  $(t:T\to
  X)\mapsto ( \operatorname{Id}\times t)^*\mathcal P$, that
  induces  a
  regular homomorphism $$\Phi_{X/S}^d:\mathscr A^d_{X/S}\to
  \operatorname{Alb}_{X/S},$$ which is an algebraic representative for
  dimension-$0$ cycles.

  \item For any morphism $S'\to S$ obtained as an inverse limit of morphisms in $\mathsf {Sm}_\Lambda/S$,
the natural homomorphism of $S'$-abelian
  schemes $$ \operatorname{Ab}^d_{X_{S'}/S'} \to (
  \operatorname{Ab}^d_{X/S})_{S'}$$ is an isomorphism.

  
    \item For any separably closed
  point
  $s : \spec \Omega \to S$  obtained
  as an inverse limit of morphisms to $S$ in $\mathsf {Sm}_\Lambda/S$,
  the homomorphism
  $\Phi^d_{X/S}(\Omega)$ is an isomorphism
  on torsion\,; in particular, the map
  \begin{equation*}
  \xymatrix@C=4em{
   T_l \operatorname{A}^d(X_\Omega) \ar@{->}[r]^<>(0.5){T_l
    \Phi^d_{X/S}(\Omega)}_\cong
   & T_l \operatorname{Ab}^d_{X/S}(\Omega).\\
  }
  \end{equation*}
  is an isomorphism for all primes $l$.
\end{enumerate}

 \end{teo}

 \begin{proof}
 This follows readily from the standard arguments in \cite[VI,
  Thm.~3.3(iii)]{FGA}.
%
%
%
%
%
 \end{proof}

 \begin{rem} \label{R:mumford}
  Mumford's examples (e.g., a smooth K3 surface over the complex numbers) show that $\Phi^d_{X/S}(\Omega)$ need not be
  an isomorphism. Precisely, if the generic fiber $X_{\eta_S}$ of the smooth
  projective morphism $f:X\to S$ is  such that some $\ell$-adic cohomology group
  $H^i(X_{{\eta_S^{\sep}}},\rat_\ell)$ is not supported on a divisor for some
  $i>1$, then $\Phi^d_{X/S}(\Omega)$ is not injective for any separable field
  extension $\Omega/\eta_S$ that is separably closed and of infinite
  transcendence
  degree over its prime subfield. However, if $\Omega$ is $\bar \rat$ or a
  separably closed subfield of $\overline{\mathbb{F}_p(T)}$, then the Beilinson
  conjecture predicts that $\Phi^d_{X/S}(\Omega)$ should be an isomorphism.
 \end{rem}

 \begin{rem} We are unaware if $\Phi^d_{X/S}$ is an epimorphism of functors,
  \emph{i.e.}, in view of Proposition~\ref{P:univ-epi}, if it admits a universal codimension-$d$ cycle class.
 \end{rem}

\begin{rem}\label{R:0cycDeg1}
If $\Lambda=S=\operatorname{Spec}K$, and $X/K$ is proper 
and geometrically normal,  then the Abel--Jacobi map induces an algebraic representative 
 $\Phi_{X/K}^1:\mathscr A^1_{X/K}\to
  (\operatorname{Pic}^0_{X/K})_{\operatorname{red}}=\operatorname{Ab}^1_{X/K}$  (Remark~\ref{R:Pic0Kred}).
If in addition $X$ admits a zero-cycle $\alpha_0$ of degree-one (but not necessarily a $K$-point), then $X$ still admits a universal divisor $\mathcal P$,
which induces, as in Theorem \ref{T:AlbAb}(i), a regular homomorphism $\Phi:\mathscr A^d_{X/K}\to (\operatorname{Pic}^0_{X/K})_{\operatorname{red}}^\vee$ that   on $\bar K$ points is given by $\alpha\mapsto \mathcal P_\alpha$, where
if  $\alpha=\sum n_ip_i$, then  $\mathcal P_\alpha= \bigotimes \mathcal P_{p_i}^{\otimes n_i}$, viewed as a line bundle on $(\operatorname{Pic}^0_{X/K})_{\operatorname{red}}$.    The same arguments one uses to prove  Theorem~\ref{T:AlbAb}(i) (using the map $X\to \mathscr A^d_{X/K}$, $x\mapsto x-\alpha_0$) show that $\Phi$ is an algebraic representative in codimension-$d$\,; \emph{i.e.},  under these hypotheses
$$\operatorname{Ab}^d_{X/K} =  ((\operatorname{Pic}^0_{X/K})_{\operatorname{red}})^\vee=(\operatorname{Ab}^1_{X/K})^\vee.
$$
Note that if in addition $X$ admits a $K$-point (in which case $X$ is also geometrically connected) then the morphism $X\to \operatorname{Ab}^d_{X/K}$ is an Albanese morphism. 
\end{rem}

In \S \ref{S:dvr} below, we study the behavior of algebraic representatives
under base change.  Granting Theorems \ref{T:dvr} and
\ref{T:basechangechar0} for now -- their proofs are
independent of the present section -- we can secure a duality like
that of Theorem \ref{T:AlbAb}$(i)$ under slightly weaker hypotheses.

\begin{pro}\label{P:AlbPicDual}   Assume $S$ is the spectrum of a DVR, or $\Lambda = \operatorname{Spec} K$ for a field $K\subseteq \mathbb C$. 
Suppose that $f:X\to S$
  is smooth and proper and  that $\Phi^1_{X/S}:\mathscr A^1_{X/S}\to \operatorname{Ab}^1_{X/S}$ and $\Phi^d_{X/S}:\mathscr A^d_{X/S}\to \operatorname{Ab}^d_{X/S}$   are algebraic representatives (Theorem \ref{T:mainAb2}).
 \begin{enumerate}[label=(\roman*)]
\item If the generic fiber  admits a zero-cycle $\alpha_0$ of degree-one, then it induces a regular homomorphism   $\Phi^d_{X/S}:\mathscr A^d_{X/S}\to ((\operatorname{Ab}^1_{X/S}))^\vee$ that is an algebraic representative; i.e., $((\operatorname{Ab}^1_{X/S}))^\vee\cong \operatorname{Ab}^d_{X/S}$. 

\item   If $X/S$ admits a section and the generic fiber is geometrically connected, then there is an induced map $X\to \operatorname{Ab}^d_{X/S}=((\operatorname{Ab}^1_{X/S}))^\vee$ making $((\operatorname{Ab}^1_{X/S}))^\vee$ a pointed Albanese for $X/S$ (Lemma \ref{L:AlbAb}).  

\item 
If, moreover, there exists a universal codimension-$1$ cycle class $Z$ (see Corollary \ref{C:UnivCodim1}), then it induces a morphism
 $X\to  ((\operatorname{Ab}^1_{X/S}))^\vee$ that agrees with the pointed Albanese map in (ii). 
\end{enumerate}

\end{pro}

\begin{proof}
The proofs are straightforward applications of Theorems \ref{T:dvr} and
\ref{T:basechangechar0}.
\end{proof}

 \section{Surjective regular homomorphisms, algebraic representatives and
  cohomology} \label{S:dvr}

 \subsection{The $\ell$-adic Bloch map}
 Recall that, for $X$ a smooth projective variety over an algebraically closed
 field $K$, Bloch \cite{bloch79} defined for $\ell$ a prime different from the
 characteristic of $K$ and for all nonnegative integers $i$ a functorial map
 $$\lambda^i :
 \operatorname{CH}^i(X)[\ell^\infty] \to H^{2i-1}(X,\rat_\ell/\integ_\ell(i)),$$
 which we will refer to as the \emph{Bloch map}.  Here, for an abelian group $M$,
 $M[\ell^\infty]$ denotes the $\ell$-primary torsion subgroup  $\varinjlim {M[\ell^n]}$.
 In fact, Bloch's map is defined for smooth projective varieties defined over a
 separably closed field (\emph{e.g.}, \cite[Rem.~\ref{ACMVBlochMap:R:D:ell-Bloch79Sep}]{ACMVBlochMap}), and one
 may
 define an \emph{$\ell$-adic Bloch map}
 $$\lambda^i :
 T_\ell\operatorname{CH}^i(X)[\ell^\infty] \to H^{2i-1}(X,\integ_\ell(i))_\tau$$
 for all smooth projective varieties over a separably closed field.
 Here, for an abelian group $M$, $T_\ell M$ denotes the Tate module $\varprojlim
 M[\ell^n]$ and $M_\tau$ denotes the quotient by the torsion subgroup. Here we
 use the observation, originally due to Suwa \cite{Suwa} that by applying
 $T_\ell$
 to the usual Bloch map one obtains the $\ell$-adic Bloch map.
 Classically, the usual Bloch map is bijective for codimension-1 cycles and for
 dimension-0 cycles\,; \emph{cf.}~\cite[Prop.~3.6, Prop.~3.9,
 Thm.~4.1]{bloch79}. The corresponding statement for the $\ell$-adic Bloch map
 follows by applying $T_\ell$\,; \emph{cf.}~\cite[Prop.~\ref{ACMVBlochMap:P:Kummer},
 Prop.~\ref{ACMVBlochMap:P:Rojtman}]{ACMVBlochMap}.

 \begin{teo}[\cite{MS}, \cite{CTR}]  \label{T:MS}
  Let $X$ be a smooth proper
  variety over a separably closed
  field $K$. Then the $\ell$-adic Bloch map
  \begin{equation*}\label{E:BlochMap}
  \xymatrix{
   T_\ell \operatorname{CH}^{2}(X)\ar[r]^{\lambda^{2} \quad}& H^{3}(X,\mathbb
   Z_\ell(2))_\tau
  }
  \end{equation*}
  is injective.
 \end{teo}
 \begin{proof}
  The injectivity of the usual second Bloch map   $\lambda^2 :
  \operatorname{CH}^2(X)[\ell^\infty] \to H^3(X,\rat_\ell/\integ_\ell(2))$ for
  $X$
  smooth projective is due to Merkurjev and Suslin \cite{MS}, and the
  corresponding statement for the second $\ell$-adic Bloch map follows by
  applying
  $T_\ell$\,; \emph{cf.}~\cite[Prop.~\ref{ACMVBlochMap:P:M-9.2}]{ACMVBlochMap}.
  In fact,  the usual second Bloch map
  is shown to be injective for $X$ smooth and proper over a separably closed
  field
  in \cite[Prop.~3.1 and
  Rmk.~3.2]{CTR}. Applying $T_\ell$ gives the theorem.
 \end{proof}

 \subsection{Surjective regular homomorphisms and cohomology}
 Let $X$ be a smooth proper variety over a
 field $K$ of characteristic exponent $p$. Recall from the proof of
 Lemma~\ref{L:Tor} that any cycle $Z\in \mathscr A^{i}_{X/K}(A)$ induces a
 homomorphism of $\mathrm{Gal}(K)$-representations on $N$-torsion
 $$w_Z : A(K^{\sep})[N] \to \A^i(X_{K^{\sep}})[N], \quad a \mapsto
 Z_*([a]-[0])$$
 for all $N$ coprime to $p$\,; and hence a homomorphism of
 $\mathrm{Gal}(K)$-representations
 \begin{equation}\label{E:w}
 w_{Z,\ell} :  T_\ell A
 \to T_\ell \operatorname{A}^{i}(X_{K^{\sep}})
 \end{equation} for all primes $\ell \neq p$.\medskip

 Let now $\Phi:\mathscr A^{i}_{X/K}\to A$ be a surjective regular homomorphism
 and assume $Z\in \mathscr A^{i}_{X/K}(A)$ is a miniversal cycle, provided by
 Lemma~\ref{L:MinVZK}, with
 $\Phi(A)(Z):A\to A$ given by $r\cdot \operatorname{Id}_A$ for some natural
 number $r$. Then, still by the arguments of the proof of Lemma~\ref{L:Tor}, the
 homomorphism of $\mathrm{Gal}(K)$-representations $w_{Z,\ell}$ of \eqref{E:w}
 is
 injective for all primes $\ell\neq p$ coprime to~$r$. The following proposition
 is then an immediate consequence of Theorem~\ref{T:MS} (and the discussion preceding it concerning the cases $i=1, \dim X$).

 \begin{pro}\label{P:BlochMap} With the notation above,  let $\ell$ be a
  prime number different from $p$ and coprime to~$r$.
  Then,  for $i=1,2,\dim X$, the
  morphism of $\operatorname{Gal}(K)$-representations
  $\iota_{Z,\ell}:T_\ell A \to H^{2i-1}(X_{K^{\sep}},\mathbb Z_\ell(i))$  given
  as
  the composition
  \begin{equation*}
  \xymatrix{
   \iota_{Z,\ell}:
   T_\ell A
   \ar@{^(->}[r]^<>(0.5){w_{Z,\ell}} & T_\ell
   \operatorname{A}^{i}(X_{K^{\sep}})\ar@{^(->}[r]& T_\ell
   \operatorname{CH}^{i}(X_{K^{\sep}})\ar[r]^{\lambda^{i} \quad}&
   H^{2i-1}(X_{K^{\sep}},\mathbb Z_\ell(i))_\tau
  }
  \end{equation*}
  is injective.  Moreover, if $\Phi$ is an algebraic
  representative, then $\iota_{Z,\ell}$ is an isomorphism for $i=1,\dim X$.\qed
 \end{pro}

 \subsection{Algebraic representatives  over a DVR}

 \begin{teo}\label{T:dvr}
  Let $S = \Lambda$ be the spectrum of a DVR and denote $\eta$ its generic
  point. Suppose $X$ is a smooth and proper scheme over $S$. Fix $i=1,2$ or $\dim_S X$. Then the natural
  homomorphism of $\eta$-abelian varieties $$\mathrm{Ab}^i_{X_{\eta}/\eta}
  \longrightarrow (\mathrm{Ab}^i_{X/S})_{\eta}$$ is an isomorphism.
 \end{teo}

 \begin{proof}
  We start by observing that, due to Proposition~\ref{P:BlochMap}
  and the N\'eron--Ogg--Shafarevich
  criterion, $\mathrm{Ab}^i_{X_{\eta}/\eta}$
  admits a model $A$ over $S$. Moreover, there is a regular homomorphism $\Phi:
  \mathscr A^i_{X/S} \to A$ whose base-change to $\eta$ is
  $\Phi^i_{X_{\eta}/\eta}$, defined in the following way. If $T$ is a separated
  smooth scheme of finite type over $S=\Lambda$ and if $Z\in \mathscr A^i_{X/S}(T)$,
  then
  $\Phi(T)(Z)$ is the unique morphism
  $T\to A$ lifting $\Phi^i_{X_{\eta}/\eta}(T_{\eta})(Z_{\eta}) : T_{\eta} \to
  A_{\eta} = \mathrm{Ab}^i_{X_{\eta}/\eta}$. Such a morphism exists due to the
  fact that any rational map defined over a base $S$ from a regular scheme to an
  abelian $S$-scheme extends to a morphism (Proposition~\ref{P:BLR}$(i)$,
  \cite[Cor.~6, \S 8.4]{BLR}). (Note that
  the above argument in fact establishes that any surjective regular
  homomorphism
  $ \mathscr A^i_{X_\eta/\eta} \to B$ extends to a surjective regular
  homomorphism
  $\mathscr A^i_{X/S} \to \mathcal{B}$.)

  By the universal property of the algebraic representatives, there exist a
  unique homomorphism $f : \mathrm{Ab}^i_{X_{\eta}/\eta} \rightarrow
  (\mathrm{Ab}^i_{X/S})_{\eta}$
  such that $(\Phi^i_{X/S})_{\eta} = f \circ \Phi^i_{X_{\eta}/\eta}$
  as well as a unique homomorphism $g : \mathrm{Ab}^i_{X/S} \to A$ such that $\Phi =
  g\circ \Phi^i_{X/S}$. By base-changing the latter to $\eta$, one obtains
  $\Phi^i_{X_{\eta}/\eta} = g_{\eta} \circ (\Phi^i_{X/S})_{\eta}$.
  In other words, we have a commutative diagram
  $$\xymatrix@R=1.3em{\mathscr A^i_{X_\eta/\eta} \ar[rr]^{\Phi_{X_\eta/\eta}}
   \ar[dr]_{(\Phi^i_{X/S})_{\eta}} & & \mathrm{Ab}^i_{X_{\eta}/\eta}
   \ar@/^/@{-->}[dl]^f\\
   &  (\mathrm{Ab}^i_{X/S})_{\eta} \ar@/^/@{-->}[ur]^{g_\eta}
  }
  $$
  This yields the identity $\Phi^i_{X_{\eta}/\eta} = g_{\eta} \circ f \circ
  \Phi^i_{X_{\eta}/\eta}$, and hence by the universal property of
  $\Phi^i_{X_{\eta}/\eta}$ that $g_{\eta} \circ f  =
  \mathrm{id}_{\mathrm{Ab}^i_{X_{\eta}/\eta}}$. In particular,  $f$ is an 
  injective homomorphism. On the other hand, we also obtain the identity $f\circ
  g_\eta \circ (\Phi_{X/S})_\eta = (\Phi_{X/S})_\eta$. The surjectivity of
  $\Phi_{X/S}$ (Proposition~\ref{P:algrepsurj}) implies then that $f\circ g_\eta
  :
  (\mathrm{Ab}^i_{X/S})_{\eta} \to (\mathrm{Ab}^i_{X/S})_{\eta}$ is the identity
  on $\eta^{\sep}$-points, in particular that $f$ is surjective on
  $\eta^{\sep}$-points. In conclusion $f$ is an isomorphism.
 \end{proof}

 \begin{rem} \label{R:specialfiber}
  We remind the reader that, in contrast, the \emph{special} fiber of the algebraic
  representative is typically \emph{not} the algebraic representative of the
  special fiber (\emph{e.g.}, \cite[Ex.~6.5]{ACMVdcg}).
 \end{rem}

On the other hand, we have the following, which gives a condition for 
 the special fiber of the algebraic
  representative to be isogenous to a sub-abelian variety of the algebraic representative of the
  special fiber\,:  

\begin{cor}\label{C:algrepDVR}
	Let $S=\Lambda$ be the spectrum of a discrete valuation ring with generic
	point $\eta = \spec K$ and closed point $\circ = \spec \kappa$.  Let
	$X/S$ be a smooth projective scheme, and let $\Gamma \in \mathscr A^2_{X_\eta/K}(\operatorname{Ab}^2_{X_\eta/K})$ be a miniversal cycle of minimal degree $r$. 
	Suppose that for some prime $\ell \nmid r\cdot \operatorname{char}(K)$ we have  $\phi^2_{X_\circ/\kappa}[\ell^\infty]$ is injective.  
	Then the morphism 
	$(\Ab^2_{X/S})_\circ \to \Ab^2_{X_\circ/\kappa}$ induced by $\Gamma$ is an isogeny 
	onto its image. 
\end{cor}
\begin{proof} By Theorem \ref{T:dvr} 
	we have  $(\Ab^2_{X/S})_\eta \iso
	\Ab^2_{X_\eta/\eta}$.  
	Let $\Gamma_{X/S} \in \mathcal A^2_{X/S}(\Ab^2_{X/S})$ be a
	miniversal cycle of minimal degree $r$ induced by the one in the assumption of the lemma. Its specialization induces a group homomorphism
	$w_{\Gamma_{X/S},\circ}: \Ab^2_{X/S}(\bar \kappa) \to
	\A^2(X_{\circ,\bar\kappa})$, and thus a homomorphism
	$\psi_{\Gamma_{X/S,\circ}}: (\Ab^2_{X/S})_\circ \to \Ab^2_{X_\circ/\circ}$. 
	On $\ell$-primary torsion, we have a
	commutative diagram
	\begin{equation}
	\label{D:specializeA2}
	\xymatrix{
		\Ab^2_{X/S}[\ell^\infty](\bar K) \ar@{->}[r]^{\sim} \ar[d]^{w_{\Gamma_{X/S}}}&
		(\Ab^2_{X/S})_\circ[\ell^\infty](\bar\kappa)
		\ar[d]^{w_{\Gamma_{X/S},\circ}}
		\ar@/^5pc/[dd]^{\psi_{\Gamma_{X/S},\circ}} \\
		\A^2(X_{\bar K})[\ell^\infty] \ar@{^{(}->}[r] & \A^2(X_0)[\ell^\infty]
		\ar@{^{(}->}[d]^{\phi^2_{X_\circ/\bar\kappa}[\ell^\infty]} \\
		& \Ab^2_{X_\circ/\kappa}[\ell^\infty](\bar\kappa)
	}
	\end{equation}
	Both the top and bottom horizontal 
	arrows are the specialization maps; the fact that the specialization map on torsion cycle classes is injective in codimension-$2$ follows from the fact that the second Bloch map is an inclusion.  Choose a prime $\ell\ne \operatorname{char}(K)$
	relatively prime to $r$.  Then $w_{\Gamma_{X/S}}$ and  is injective, and by commutativity, this implies $w_{\Gamma_{X/S,\circ}}$ is injective. 
	Then since we assume that $\phi^2_{X_\circ/\kappa}[\ell^\infty]$ is injective, it follows that all arrows in
	\eqref{D:specializeA2} are injective.
	In particular,  it follows that 
	$\ker \psi_{\Gamma_{X/S},\circ}$ is trivial, and therefore has no nontrivial $\ell$-torsion, so that the 
	morphism of abelian varieties is an isogeny onto its image. 
\end{proof}

\begin{rem} The condition that 
 $\phi^2_{X_\circ/\kappa}[l^\infty]$ be injective holds  for all primes $l$ if  $\operatorname{char}(\kappa) = 0$  \cite[Thm.~10.3]{murre83}, or if $X$ is geometrically rationally chain connected with $\kappa$ perfect \cite[Prop.~3.8]{ACMVBlochMap}.  In other words, under these hypotheses,  the morphism 
	$(\Ab^2_{X/S})_\circ \to \Ab^2_{X_\circ/\kappa}$ induced by $\Gamma$ is an isogeny 
	onto its image. 
We are unaware of an example of a smooth projective variety $X_\circ$ for which $\phi^2_{X_\circ/\kappa}[l^\infty]$ is not an isomorphism\,; in other words, we are unaware of an example where the morphism 
	$(\Ab^2_{X/S})_\circ \to \Ab^2_{X_\circ/\kappa}$ induced by $\Gamma$ is \emph{not} an isogeny 
	onto its image. 
\end{rem}

 \subsection{Algebraic representatives and base change
  in characteristic zero} Our aim is to prove Theorem~\ref{T2:mainAb2}$(i)$\,:

 \begin{teo} \label{T:basechangechar0} Suppose that $\Lambda=\operatorname{Spec}K$ for a field $K\subseteq \mathbb C$, and
 $X$ is a
  smooth proper scheme over $S$. Fix $i=1,2$ or $\dim_S X$.
 If $S'\to S$ is a morphism  obtained as an inverse limit of morphisms in $\operatorname{Sm}_K/S$,
 then the natural homomorphism of $S'$-abelian
  schemes $$ \operatorname{Ab}^i_{X_{S'}/S'} \to (
  \operatorname{Ab}^i_{X/S})_{S'}$$ is an isomorphism.
 \end{teo}
 \begin{proof}
  By rigidity (\emph{e.g.}, \cite[Prop.~6.1]{mumfordGIT})
  the natural homomorphism
  $\mathrm{Ab}^i_{X_{S'}/S'} \to (\mathrm{Ab}^i_{X/S})_{S'}$ is an isomorphism if
  and only if it is an isomorphism when restricted to the generic point
  $\eta_{S'}$.
  By Theorem~\ref{T:mainalgrep}$(i)$ applied to $X_{\eta_S}$ together with the
  field extension $\kappa(S')/\kappa(S)$, we are reduced to showing that the
  natural homomorphism
  $$\mathrm{Ab}^i_{X_{\eta_S}/\eta_S} \to (\mathrm{Ab}^i_{X/S})_{\eta_S}$$ is an
  isomorphism for any smooth proper scheme over a smooth separated scheme $S$ of
  finite type over~$K$.
  This is achieved exactly as in the proof of
  Theorem~\ref{T:dvr}\,: the crux consists in showing that any surjective regular
  homomorphism
  $\Phi:  \mathscr A^i_{X_{\eta_S}/\eta_S} \to B$ extends to a surjective regular
  homomorphism
  $\widetilde \Phi: \mathscr A^i_{X/S} \to \mathcal{B}$. For that purpose,
  we first use,  as in \cite[p.~19]{ACMVsigma}, both the
  N\'eron--Ogg--Shafarevich
  criterion \emph{and} the Faltings--Chai criterion~\cite[Cor.~6.8 p.~185]{FC}, to extend $B$ to an abelian scheme $\mathcal{B}$ over $S$. Second, if $T$ is a smooth separated scheme of finite type over $K$ with a \emph{dominant} morphism to $S$  and if $Z\in \mathscr A^i_{X/S}(T)$,
  then we use  Proposition~\ref{P:BLR}$(i)$,
  \cite[Cor.~6, \S 8.4]{BLR} to define
  $\widetilde \Phi(T)(Z)$ as the unique morphism
  $T\to A$ lifting $\Phi(T_{{\eta_S}})(Z_{{\eta_S}}) : T_{{\eta_S}} \to
  \mathcal{B}_{{\eta_S}} = B$.
 \end{proof}

\begin{rem}
Vasiu and Zink \cite{vasiuzinkpurity} have investigated the structure
of mixed-characteristic rings over which a suitable version of the
Faltings--Chai extension theorem holds.  If $S \to \integ_{(p)}$
satisfies the hypotheses of \cite[Cor.~5]{vasiuzinkpurity} (for
example, if $S$ is smooth over an unramified discrete valuation ring
of mixed characteristic $(0,p)$), and if
$X\to S$ is smooth and proper, then $\Ab^i_{X_{\eta_S}}$ spreads out
to an abelian scheme over $S$, and we again have $\Ab^i_{X_{\eta_S}}
\iso (\Ab^i_{X/S})_{\eta_S}$ for $i = 1$, $2$ or $\dim_SX$.
\end{rem}

Let $X/K$ be a smooth projective variety.  
In \cite[\S 4]{ACMVdiag} we investigate geometric conditions, such as a
decomposition of the diagonal in Chow, which imply that there exists a
correspondence $\Gamma \in \chow^{d_{\Ab^2_{X/K}}+1}(\widehat
\Ab^2_{X/K}\times X)$ which, in any Weil cohomology theory $\mathcal
H$, induces an isomorphism $\Gamma_{\mathcal H, *}:\mathcal
H^1(\Ab^2_{X/K}) \stackrel{\sim}{\to} \mathcal H^3(X)(1)$.  In
this setting, even in positive characteristic, we can obtain an
analogue of Theorem \ref{T:basechangechar0}.

In the following statement, we denote by $W = W(K)$ the ring of Witt
vectors of a field $K$, and let $B = B(K) = \operatorname{Frac} W$ be its fraction
field.  The hypothesis that the crystalline cohomology be locally free
is satisfied by, for example, families of complete intersections
\cite[Ex.~3.13]{morrow19}, or families which admit a decomposition of
the diagonal \cite[\S~\ref{ACMVdiag:S:DD-Tor}]{ACMVdiag}.  In such a case, inducing an
inclusion of crystals is equivalent to inducing an inclusion of isocrystals.

\begin{lem}
Let $\Lambda = \spec K$ be the spectrum of a perfect field of
characteristic $p>0$, let $S/\Lambda$ be smooth, and let $f:X \to S$ be
a smooth proper morphism.  Suppose that, over the generic point
$\eta_S$ of $S$, there exists a correspondence $\Gamma \in
\chow^{d_{\Ab^2_{X_{\eta_S}/\eta_S}}+2}(\Ab^2_{X_{\eta_S}/\eta_S} \times X_{\eta_S})$ such that
$\Gamma_{\cris,*}: H^1_\cris(\Ab^2_{X_{\eta_S}/\eta_S}) \to
H^3_\cris(X_{\eta_S})(1)$ is an inclusion of crystals.
Suppose further that the crystal $R^3 f_{\cris, *} \mathcal O_{S/W(K)}$
is locally free.
Then $\Ab^2_{X_{\eta_S}/\eta_S}$ extends to an abelian scheme
$\underline\Ab^2_{X_{\eta_S}/\eta_S} \to S$.
\end{lem}

\begin{proof}
  Take a spread $g: \underline \Ab^2_{X_{\eta_S}/\eta_S} \to U$ of the
  generic algebraic representative to an abelian scheme over a  nonempty open subscheme $U$ of $S$. Using
  the N\'eron--Ogg--Shafarevich criterion, we may and do assume that
  the complement of $U$ in $S$ has codimension at least two. To slightly ease notation, let
  $\mathcal M = R^1g_{\cris,*}\mathcal O_{\underline
    \Ab^2_{X_{\eta_S}/\eta_S}/W}$ and $\mathcal N = R^3
  f_{\cris,*}\mathcal O_{X/W}(1)$; they are locally free crystals of
  $\mathcal O_{U/W}$- and $\mathcal O_{S/W}$-modules, respectively.

Since $S$ is smooth and thus regular, $\Gamma_{\cris,*}$ extends to an
inclusion of crystals $\gamma: \mathcal M \hookrightarrow \mathcal
N_U$ \cite[Thm.~1.1]{dejongBT}.  Using the elementary
  divisors of the image of $\gamma$, we choose an
  endomorphism $\beta \in \End(\mathcal N_U)$ such that
  $\beta(\mathcal N_U) = \gamma(\mathcal M)$.
  Restriction of F-isocrystals from $S$ to $U$ is fully faithful
  (e.g, \cite[Thm.~5.1]{kedlayanotes}); since by hypothesis $\mathcal
  N$ is locally free, the restriction map of (integral) endomorphisms from
  $\End(\mathcal N)$ to $\End(\mathcal N_U)$ is a bijection.  Let $\widetilde \beta$ be the extension of
  $\beta$ to $S$.  Then $\widetilde\beta(\mathcal N)$ is a crystal on $S$
  whose restriction to $U$ is isomorphic to $\mathcal M$. Thus,
  $\mathcal M := \widetilde\beta(\mathcal N)$ is a locally free
  $F$-crystal on $S$.  Since the Dieudonn\'e functor over $S$ is an
  equivalence of categories \cite[Thm.~4.6]{dejongmessing}, there is a
  $p$-divisible group $\mathcal G \to S$ with $\mathcal G_U$
  isomorphic to $\underline \Ab^2_{X_{\eta_S}/\eta_S}[p^\infty]$.
  Using the proof of \cite[Prop.~4.1]{vasiu04} (the argument given
  there is valid over arbitrary regular Noetherian rings), the abelian
  scheme $\underline \Ab^2_{X_{\eta_S}/\eta_S}$ extends to an abelian
  scheme over all of $S$.
\end{proof}

 \section{Regular homomorphisms and  intermediate
  Jacobians}\label{S:jac}

 The point of this section is to show that Abel--Jacobi maps in the relative
 setting  induce regular homomorphisms in the sense we define here\,; indeed,
 this
 was one of our main motivations in proving \cite[Thm.~1]{ACMVsigma}.
 To explain briefly, recall that given a smooth projective morphism $f:X\to S$
 of
 smooth complex varieties, a result of Griffiths states that the Abel--Jacobi
 map
 induces a regular homomorphism in the analytic category.  In other words, given
 any dominant
 morphism of complex manifolds $T\to S$
 and any cycle class $Z\in \operatorname{CH}^i(X_T)$ such that every refined
 Gysin fiber $Z_t$ is homologically trivial, the associated map of sets $\nu_Z:T
 \to \mathrm J^{2i-1}(X_T/T)=\mathrm J^{2i-1}(X/S)_T$, $t\mapsto AJ_{X_t}(Z_t)$,
 is induced by a holomorphic map of analytic spaces\,; $\nu_Z$ is called the
 motivated normal function associated to $Z$.

 Our recent result \cite[Thm.~1]{ACMVsigma} shows that assuming $f$, $X$, and $S$ are algebraic, the Abel--Jacobi map
 actually induces a regular homomorphism (in the algebraic sense we define here) if we restrict to algebraically trivial
 cycle classes.  More precisely,
 there is an algebraic relative subtorus $\mathrm J^{2i-1}_a(X/S)\subseteq
 \mathrm J^{2i-1}(X/S)$ such that
 if we assume further that $T$ and $Z$ are algebraic with
 every refined Gysin fiber $Z_t$ algebraically trivial, then the normal
 function factors as $\nu_Z:T\to \mathrm J^{2i-1}_a( X/S)_T\subseteq \mathrm
 J^{2i-1}(X/S)_T$, and is algebraic (not just holomorphic).  In particular,
 there
 is a regular homomorphism $\Phi_{AJ}:\mathscr A^i_{X/S/\mathbb C}\to \mathrm
 J^{2i-1}_a(X/S)$, with $\Phi_{AJ}(T)(Z)=\nu_Z:T\to \mathrm J^{2i-1}_a(X/S)_T$
 given by the associated normal function.
 We now explain this in more detail, and recall some of the arithmetic
 properties of these normal functions.

 \subsection{The Abel--Jacobi map over subfields of $\cx$} \label{SS:AJequi}
 We reformulate the main result of \cite{ACMVdmij} within our functorial
 setting. We recall the definition of the \emph{algebraic intermediate
  Jacobian}.
 Let $X$ be a smooth projective variety over $\cx$.
 Griffiths defined an
 Abel--Jacobi map $AJ : \operatorname{CH}^i(X)_{\mathrm{hom}} \to
 \mathrm{J}^{2i-1}(X)$ on homologically trivial cycles. Here
 $\mathrm{J}^{2i-1}(X)$ is the so-called \emph{intermediate Jacobian}\,; it is
 the
 complex torus defined by
 $$ \mathrm{J}^{2i-1}(X) := \mathrm{F}^i\mathrm{H}^{2i-1}(X,\cx)\backslash
 \mathrm{H}^{2i-1}(X,\cx) / \mathrm{H}^{2i-1}(X,\integ)_{\tau},$$ with
 $\mathrm{F}^\bullet$ being the Hodge filtration and with the subscript
 ``$\tau$''
 referring to the quotient by the maximal torsion subgroup. Given a
 homologically
 trivial cycle class $\gamma = \partial  \Gamma$, the Abel--Jacobi map assigns,
 \emph{via} Poincar\'e duality, the linear form $\int_{\Gamma}(-)$ to $\gamma$.
 It is a theorem of Griffiths that this assignment is well-defined. Moreover,
 the
 image of the restriction of the Abel--Jacobi map to
 algebraically trivial cycles defines a subtorus $\mathrm{J}^{2i-1}_a(X)$ which
 is naturally
 endowed with a polarization and hence defines a complex abelian variety. This
 complex abelian variety will be called the \emph{algebraic intermediate
  Jacobian}.
 A basic result of Griffiths shows that the induced map $AJ: \mathrm{A}^i(X)
 \to \mathrm{J}_a^{2i-1}(X)(\cx)$ (which we will henceforth refer to by abuse as
 the Abel--Jacobi map) is a regular homomorphism in the classical sense, and
 hence by \S \ref{SS:reg} provides a regular homomorphism  $AJ:
 \mathscr{A}^i_{X/\cx} \to \mathrm{J}_a^{2i-1}(X)$ in the sense of
 Definition~\ref{D:reghom} whose evaluation at $\spec \cx$ gives the classical
 Abel--Jacobi map $AJ: \mathrm{A}^i(X_\cx) \to \mathrm{J}_a^{2i-1}(X_\cx)(\cx)$.

 \begin{teo}[{\cite{ACMVdmij}}]  \label{T:dmij}
  Let $X$ be a smooth projective variety over a field $K\subseteq \cx$. Then the
  algebraic intermediate Jacobian $\mathrm{J}_a^{2i-1}(X_\cx)$ admits a
  distinguished
  model $\mathrm{J}^{2i-1}_{a,X/K}$ over $K$ and there exists a surjective
  regular
  homomorphism
  $$
  \Phi_{AJ_{X/K}}:\mathscr A^i_{X/K} \longrightarrow \mathrm{J}^{2i-1}_{a,X/K}
  $$
  whose evaluation at $\spec \cx$ is the Abel--Jacobi map $AJ:
  \mathrm{A}^i(X_\cx) \to \mathrm{J}_a^{2i-1}(X_\cx)(\cx)$. In particular, the
  Abel--Jacobi map is
  $\mathrm{Aut}(\cx/K)$-equivariant.
 \end{teo}

  \begin{rem}\label{R:DistMeaning}
  Since $AJ:\operatorname{A}^{i}(X_{\mathbb C})\to
  \mathrm{J}_a^{2i-1}(X_{\mathbb C})$ is surjective,
  the abelian variety
  $\mathrm{J}_a^{2i-1}(X_{\mathbb C})$ admits at most one  structure of a scheme
  over $K$
  such that $AJ$ is
  $\operatorname{Aut}(\cx/K)$-equivariant.  This is
  the sense in which  $\mathrm{J}_a^{2i-1}(X_{\mathbb C})$ admits a
  \emph{distinguished model}
  over~$K$.
  See also \cite[Thm.~4.1]{HTCycClaMaps}.
 \end{rem}

 \begin{proof}
  The following proof fixes a gap in the proof of the
  $\mathrm{Aut}(\cx/K)$-equivariance of the Abel--Jacobi map given
  in~\cite{ACMVdmij}.
  The starting point is \cite[Prop.~1.1]{ACMVdmij} which provides a smooth
  projective curve $C$ over $K$ and a surjective homomorphism
  $(\operatorname{Pic}^0_{C/K})_\cx = J(C_\cx) \to \mathrm{J}_a^{2i-1}(X_\cx)$,
  and hence shows that the $\cx/\bar
  K$-trace
  homomorphism $\tau : \mathrm{tr}_{\cx/\bar
   K}\big(\mathrm{J}_a^{2i-1}(X_\cx)\big)_\cx \to \mathrm{J}_a^{2i-1}(X_\cx)$ is
  an
  isomorphism  (\emph{cf.}~\cite[\S 2.1]{ACMVdmij}). We set $\mathrm{J}^{2i-1}_{a,X_{\bar K}/\bar K} =
  \mathrm{tr}_{\cx/\bar K}\big(\mathrm{J}_a^{2i-1}(X_\cx)\big)$. By
  Lemma~\ref{L:Pull-Trace}, we obtain a regular homomorphism $\LKtrace{AJ} :
  \mathscr{A}^i_{X_{\bar K}/{\bar K}} \to \mathrm{J}^{2i-1}_{a,X_{\bar K}/\bar
   K}$, which satisfies $AJ = \tau \circ (\LKtrace{AJ})_\cx$ and which is
  surjective by Proposition~\ref{P:surjtrace}.

  Next we claim that the surjective regular homomorphism $\LKtrace{AJ}$ is
  $\mathrm{Gal}(\bar K/K)$-equivariant. This is achieved in two steps in
  \cite{ACMVdmij}\,: first one shows \cite[\S 2.2]{ACMVdmij} that $
  \mathrm{J}^{2i-1}_{a,X_{\bar K}/\bar K}$ descends to an abelian variety
  $\mathrm{J}^{2i-1}_{a,X/{K}}$ over $K$ and that $\LKtrace{AJ}(\bar K) :
  \mathrm{A}^i(X_{\bar K}) \to (\mathrm{J}^{2i-1}_{a,X/{K}})(\bar K)$ is
  $\mathrm{Gal}(\bar K/K)$-equivariant on torsion\,; then one establishes the
  general fact \cite[Prop.~3.8]{ACMVdmij} that a regular homomorphism over $\bar
  K$ that is $\mathrm{Gal}(\bar K/K)$-equivariant on torsion is in fact
  $\mathrm{Gal}(\bar K/K)$-equivariant. By  Lemma~\ref{L:Pull-Desc}, we obtain a
  regular homomorphism $\Phi : \mathscr{A}^i_{X/{K}} \to
  \mathrm{J}^{2i-1}_{a,X/{K}}$, which satisfies $\Phi_{\bar K} = \LKtrace{AJ}$
  and
  which is surjective by Proposition~\ref{P:surjGal}.

  Combining the above, Lemma~\ref{L:equivariant}$(i)$ gives $\Phi_{\cx} = AJ$,
  where  $(\mathrm{J}^{2i-1}_{X/{K}})_\cx$ is identified with
  $\mathrm{J}_a^{2i-1}(X_\cx)$ \emph{via} $\tau$.  We conclude with
  Lemma~\ref{L:equivariant}$(ii)$ that relative to the above identification $AJ$
  is  $\mathrm{Aut}(\cx/K)$-equivariant.
 \end{proof}

 \subsection{The Abel--Jacobi map in families}
 In this section, we assume that $\Lambda = \spec K$ for some subfield $K$ of
 $\cx$ and that $S$ is an irreducible smooth quasi-projective variety over~$K$.
 Theorem~\ref{T:geomNF} below is a reformulation of the main result of
 \cite{ACMVsigma} within our functorial setting. This result, which in fact was
 motivated by the present work, shows the existence of non-trivial regular
 homomorphisms in families.
 Before we state the theorem, recall that Griffiths' Abel--Jacobi map can be
 constructed family-wise\,: given a smooth projective family $X$ of complex
 varieties over a smooth quasi-projective complex variety $S$, there is a family
 of complex tori $J^{2i-1}(X/S)$ over $S$ which consists fiberwise of
 the intermediate Jacobians. Moreover, a classical result of Griffiths says
 that,
 given a cycle class $Z\in \operatorname{CH}^i(X)$ which is fiberwise
 homologically trivial, the \emph{normal function} $\nu_Z : S \to
J^{2i-1}(X/S)$ defined by $t\mapsto AJ(Z_t)$ is holomorphic.

 \begin{teo}[\cite{ACMVsigma}]  \label{T:geomNF}
  Suppose $f: X \to S$ is a smooth projective morphism of smooth varieties over
  a
  field $K\subseteq \mathbb C$. Then the Abel--Jacobi map
  $\Phi_{AJ_{X_{\eta_S}/\eta_S}}:\mathscr A^i_{X_{\eta_S}/\eta_S} \to
  \mathrm{J}^{2i-1}_{a,X_{\eta_S}/\eta_S}$ of Theorem~\ref{T:dmij} extends
  uniquely to a surjective regular homomorphism
  $$
  \Phi_{AJ_{X/S}}: \mathscr{A}^i_{X/S} \to \mathrm{J}^{2i-1}_{a,X/S}
  $$
  such that the base change $(\mathrm{J}^{2i-1}_{a,X/S})_{\mathbb C}$ is
  canonically identified with an algebraic subtorus $\mathrm
  J^{2i-1}_a(X_{\mathbb
   C}/S_{\mathbb C})\subseteq \mathrm {J}^{2i-1}(X_{\mathbb C}/S_{\mathbb C})$
  of
  the intermediate Jacobian.
  Moreover,
  for each $T$ in $\mathsf {Sm}_K/S$  and each $Z\in
  \mathscr{A}^i_{X_{X/S}}(T)$,
  we have $\nu_{Z_{\mathbb C}}=(\Phi_{AJ_{X/S}}(T)(Z))_{\mathbb C}:T_{\mathbb
   C}\to (\mathrm J^{2i-1}_{a,X/S})_{\mathbb C}\subseteq \mathrm
  J^{2i-1}(X_{\mathbb C}/S_{\mathbb C})$ is the associated normal function.
 \end{teo}

 \begin{proof}
  The  content of \cite[Thm.~1]{ACMVsigma} and its proof is to show that  there
  is
  an abelian scheme $\mathrm{J}^{2i-1}_{a,X/S}$ over $S$ such that the
  Abel--Jacobi map $\Phi_{AJ_{X_{\eta_S}/\eta_S}}:\mathscr
  A^i_{X_{\eta_S}/\eta_S}
  \to \mathrm{J}^{2i-1}_{a,X_{\eta_S}/\eta_S}$ of Theorem~\ref{T:dmij}  defines
  a
  map
  $
  \delta: \mathscr{A}^i_{X/S}(S) \to \mathrm{J}^{2i-1}_{a,X/S}(S)
  $
  such that the base change $(\mathrm{J}^{2i-1}_{a,X/S})_{\mathbb C}$ is
  canonically identified with an algebraic subtorus of the intermediate Jacobian
  $\mathrm {J}^{2i-1}(X_{\mathbb C}/S_{\mathbb C})$,  and such that for each
  $Z\in
  \mathscr{A}^i_{X_{X/S}}(S)$ we have $\nu_{Z_{\mathbb C}}=(\delta(Z))_{\mathbb
   C}:S_{\mathbb C}\to (\mathrm J^{2i-1}_{a,X/S})_{\mathbb C}\subseteq \mathrm
  J^{2i-1}(X_{\mathbb C}/S_{\mathbb C})$.  Since this holds also for $X_T\to T$
  for any $T\to S$ in $\mathsf {Sm}_K/S$, in order to define the regular
  homomorphism
  $$\Phi_{AJ_{X/S}}: \mathscr{A}^i_{X/S} \to \mathrm{J}^{2i-1}_{a,X/S},$$
  it suffices to show $\mathrm J^{2i-1}_{a,X_T/T}= \mathrm
  J^{2i-1}_{a,X/S}\times_ST$.  Working on connected components, we may as well
  assume $S$ and $T$ are integral, and
  the claim is that in fact we have a fibered product diagram
  \begin{equation}\label{E:AlgIntBC}
  \xymatrix@R=.2em{
   \mathrm J^{2i-1}_{a}(X_{\mathbb C})\ar@{->}[rr] \ar[dd] \ar@{->}[rd]&&\mathrm
   J^{2i-1}_{a,X_{\eta_S}/\eta_S}\ar@{-}[d]  \ar[rr]&&\mathrm J^{2i-1}_{a,X/S}
   \ar[dd]\\
   &\mathrm J^{2i-1}_{a,X_{\eta_T}/\eta_T} \ar[dd] \ar[ru]
   \ar[rr]&\ar[d]&\mathrm
   J^{2i-1}_{a,X_T/T} \ar[dd] \ar[ru]&\\
   \operatorname{Spec}\mathbb C  \ar[rd]&\ar@{-}[l] \ar[r]&\eta_S
   \ar@{-}[r]&\ar@{->}[r]&S\\
   &\eta_T    \ar[ru] \ar@{->}[rr]&&T \ar[ru]&\\
  }
  \end{equation}
  where $\operatorname{Spec}\mathbb C\to T\to S$ is any very general point of
  $T$\,;
  \emph{i.e.}, maps to the generic point of $T$.
  That the left-hand triangular prism is a fibered product diagram is
  \cite[Rem.~1.6]{ACMVsigma}.    To show the right-hand cube is a fibered
  product
  diagram we argue as follows.  The only thing to show is that the far
  right-hand
  vertical face of the cube is cartesian.  Using that the rest of the diagram is
  cartesian, we see that $J^{2i-1}_{a,X/S}$ pulls back to
  $J^{2i-1}_{a,X_{\eta_T}/\eta_T}$ over $\eta_T$.  Thus both
  $J^{2i-1}_{a,X_S/S}\times_ST$ and $J^{2i-1}_{a,X_T/T}$ are abelian schemes
  over
  $T$ that pull back along $\eta_T\to T$ to give
  $J^{2i-1}_{a,X_{\eta_T}/\eta_T}$\,; \emph{i.e.}, they have the same generic
  fiber.
  Thus they agree (Proposition~\ref{P:BLR}$(i)$,  \cite[Cor.~6, \S 8.4]{BLR}),
  and therefore we
  have defined the regular homomorphism $\Phi_{AJ_{X/S}}$.

  That $\Phi_{AJ_{X/S}}$ is surjective is simply due to
  Proposition~\ref{prop:surjdef} and  to the fact that
  $(\Phi_{AJ_{X/S}})_{\eta_S}$ is the Abel--Jacobi map
  $\Phi_{AJ_{X_{\eta_S}/\eta_S}}: \mathscr A^i_{X_{\eta_S}/\eta_S} \to
  \mathrm{J}^{2i-1}_{a,X_{\eta_S}/\eta_S}$, which is surjective by
  Theorem~\ref{T:dmij}.
  That $\Phi_{AJ_{X/S}}$ is uniquely determined by its restriction
  $\Phi_{AJ_{X_{\eta_S}/\eta_S}}$ is simply due to the fact that if $T$ is an
  element of $\mathsf{Sm}_K/ S$, then a $\eta_S$-morphism $T_{\eta_S} \to
  (\mathrm{J}^{2i-1}_{a,X/S})_{\eta_S} = \mathrm{J}^{2i-1}_{X_{\eta_S}/\eta_S}$
  is the restriction of at most one $S$-morphism $T \to
  \mathrm{J}^{2i-1}_{a,X/S}$.
 \end{proof}

 For clarity we extract the following assertion, which was established in the
 proof
 above.

 \begin{cor}
  Suppose $T\to S$ is in $\mathsf {Sm}_K/S$.  Then we have a fibered product
  diagram\,:
  $$
  \xymatrix@R=1.5em@C=4em{
   \mathscr A^i_{X_T/T}\ar[r]^{\Phi_{AJ_{X_T/T}}} \ar[d]& \mathrm
   J^{2i-1}_{a,X_T/T} \ar[r] \ar[d]& T \ar[d]\\
   \mathscr A^i_{X/S}\ar[r]^{\Phi_{AJ_{X/S}}}& \mathrm J^{2i-1}_{a,X_S/S}\ar[r]&
   S\\
  }
  $$
 \end{cor}

 \begin{proof}
  That the right-hand square is cartesian comes from \eqref{E:AlgIntBC}.  That
  the
  outer rectangle is cartesian is \eqref{E:BC-2}.
 \end{proof}

 \subsection{Algebraic representatives and intermediate Jacobians}
 \label{SS:jac2}

 We  show that
 $\mathrm{J}^{2i-1}_{a,X/S}$ is an algebraic representative for codimension-$i$ cycles with $i=1,2,\dim_SX$.

 \begin{teo}[{Algebraic representatives over $\mathbb C$}]
  \label{T:MK}
  Suppose $\Lambda = \operatorname{Spec} K$ for some field $K\subseteq \mathbb C$, and that
$f: X \to S$ is a smooth projective.  Fix $i=1,2,\dim_S X$.
  Then the homomorphism $ \operatorname{Ab}^i_{X/S}\to
  \mathrm{J}^{2i-1}_{a,X/S}$ induced by the universal property of the algebraic
  representative
  is an isomorphism. In other words, $\mathrm{J}^{2i-1}_{a,X/S}$ is an algebraic
  representative for codimension-$i$ cycles.
 \end{teo}

 \begin{proof}
 The cases $i=1,\dim_SX$ can be proven in the same way as the case $i=2$, and so we focus on this last case.
 In the classical situation where $S = \spec \cx$, this is
  \cite[Thm.~C]{murre83}. Our strategy is to reduce to that known case.
  By base-changing to the generic point of $S$ and by the universal property of
  $\operatorname{Ab}^2_{X_{\eta_S}/\eta_S}$ we obtain a composition of
  homomorphism of abelian varieties over $\eta_S$\,:
  \begin{equation}\label{E:fac}
  \operatorname{Ab}^2_{X_{\eta_S}/\eta_S} \twoheadrightarrow
  (\operatorname{Ab}^2_{X/S})_{\eta_S} \to (\mathrm{J}^3_{a,X/S})_{\eta_S}  =
  \mathrm{J}^3_{a,X_{\eta_S}/\eta_S},
  \end{equation}
  where the identification on the right is provided by Theorem~\ref{T:geomNF}
  and the surjectivity of the arrow on the left is provided by
  Proposition~\ref{P:surjbasechange0}. In order to show that
  $\operatorname{Ab}^2_{X/S}\to
  \mathrm{J}^3_{a,X/S}$ is an isomorphism, it is enough to show that it is
  generically an isomorphism by Proposition~\ref{P:BLR}$(i)$ (\cite[Cor.~6, \S
  8.4]{BLR}).
  Hence, in view of \eqref{E:fac}, it is enough to see that the natural
  homomorphism $\operatorname{Ab}^2_{X_{\eta_S}/\eta_S} \to
  \mathrm{J}^3_{a,X_{\eta_S}/\eta_S}$ is an isomorphism. This is achieved as
  follows. Fix an embedding of $\kappa(S)$ in~$\cx$. Thanks to
  Theorems~\ref{T:Omega/k} and~\ref{T:barK/K}, we obtain after base-changing the
  latter homomorphism along the inclusion $\kappa(X)\subset \cx$ a
  homomorphism
  $$\operatorname{Ab}^2_{X_{\cx}/\cx} \to \mathrm{J}^3_{a,X_\cx},$$
  which by
  \cite[Thm.~C]{murre83} is an isomorphism.
 \end{proof}

 \begin{rem}
  As already mentioned in Example~\ref{Ex:AJ}, for $2<i<\dim X_{\mathbb C}$, even if an algebraic
  representative exists, the canonical morphism
  $\operatorname{Ab}^{i}_{X_{\mathbb C}/\mathbb C}\twoheadrightarrow
  \mathrm{J}_a^{2i-1}(X_{\mathbb C})$ need not be injective\,;
  \emph{cf.}~\cite[Cor.~4.2]{OttemSuzuki}.
 \end{rem}

 Building on the $K=\cx$ version of Theorem~\ref{T:MK}, Murre established the
 following theorem in the special case when $S = \spec \cx$\,:

 \begin{teo}\label{T:mainAb2char0}
  Suppose
  $\operatorname{char}(\kappa(\Lambda))=0$ and let $X\to S$ be
  a smooth projective
   morphism. Fix $i=1,2,\dim_SX$.
  Let $\Phi_{X/S}^i:\mathscr A^i_{X/S}\to \operatorname{Ab}^i_{X/S}$
  be the algebraic representative for codimension-$i$ cycles (whose existence is
  provided by Theorem~\ref{T:mainAb2}).
  Then
  for any separably closed
  point
  $s : \spec \Omega \to S$  obtained
  as an inverse limit of morphisms in $\mathsf {Sm}_\Lambda/S$,
  the homomorphism
  $\Phi^i_{X/S}(\Omega)$ is an isomorphism
  on  torsion.
  In particular, for all primes $\ell$, the map
  \begin{equation}\label{E:MT2}
  \xymatrix@C=4em{
   T_\ell \operatorname{A}^i(X_\Omega) \ar@{->}[r]^<>(0.5){T_\ell
    \Phi^i_{X/S}(\Omega)}
   & T_\ell \operatorname{Ab}^i_{X/S}(\Omega)\\
  }
  \end{equation}
  is an isomorphism.
 \end{teo}
 \begin{proof} This is standard over $\mathbb C$ for $i=1,\dim_SX$.  For $i=2$,
  this is \cite[\S 10]{murre83} in the case $S = \spec \cx$\,: in
  that case the natural morphism $\mathrm{Ab}^2_{X/\cx} \to \mathrm{J}_a^3(X)$
  is
  an isomorphism (Theorem~\ref{T:MK}) and the corresponding statement with $
  \mathrm{J}_a^3(X)$ in place of $\mathrm{Ab}^2_{X/\cx}$ is
  \cite[Thm.~10.3]{murre83}.
  One reduces the general case to the previous case \emph{via} Theorem \ref{T:basechangechar0}
  and
  Theorem~\ref{T:Omega/k}.
 \end{proof}


 \bibliographystyle{hamsalpha}
 \bibliography{DCG}
\end{document}

 \appendix

 \section{Albanese varieties and base change}\label{appendix}
 In this appendix we consider the problem of compatibility of the
 formation of the
 Albanese variety with base change of field.   A classical result of Serre implies the
 existence of Albanese varieties for geometrically integral varieties over fields, but it
 turns out that Albanese varieties are not stable under base change of
 field.   A result
 of Grothendieck, later generalized by Conrad, implies that if one  restricts to
 proper geometrically integral varieties, then Albanese varieties are stable under base
 change of field.  In this appendix, we show that alternatively, if one is
 willing to restrict to \emph{separable} base change of field, then
 the Albanese variety
 is stable under separable base change of field for all geometrically integral
 varieties.

 \subsection{Definition of the Albanese variety}
 Let
 $V$ be a geometrically reduced and geometrically connected separated scheme of
 finite type over a field $K$.
 Recall that an \emph{Albanese datum for $V$} consists of a triple $(\alb_{V/K}, \
 \alb^1_{V/K},a:V\to \alb^1_{V/K})$ with $\alb_{V/K}$ an abelian variety over
 $K$, $\alb^1_{V/K}$ a torsor under $\alb_{V/K}$ over $K$, and $a:V\to
 \alb^1_{V/K}$ a morphism of $K$-schemes which is universal in the sense that
 given any triple $(A, P,f:V\to P)$ with $A$ an abelian variety over $K$, $P$ a
 torsor under $A$ over $K$, and $f:V\to P$ a morphisms of $K$-schemes, there is a
 unique homomorphism $\operatorname{Alb}_{V/K}\to A$ and a unique morphism of
 torsors $\operatorname{Alb}^1_{V/K}\to P$ making the following diagram
 commute\,:
 $$
 \xymatrix{
  V \ar[d]^a \ar[r]^f& P\\
  \operatorname{Alb}^1_{V/K} \ar@{-->}[ru]_{\exists !}
 }
 $$
 We will respectively call the three objects in this datum the Albanese variety, the Albanese torsor, and the Albanese morphism of $V/K$ (although of course the torsor is itself a variety, too).

 When $V/K$ is equipped with a $K$-point $v:\operatorname{Spec}K\to V$ over $K$,
 then one can define a pointed Albanese variety and morphism, by requiring all the maps in the
 previous paragraph to be pointed.  This reduces to the following situation\,: a
 \emph{pointed Albanese datum for $(V,v)$} is a pair $(\operatorname{Alb}_{V/K},a:V\to
 \operatorname{Alb}_{V/K})$ where $\operatorname{Alb}_{V/K}$ is an abelian
 variety, and $a:V\to \operatorname{Alb}_{V/K}$ is a morphism of $K$-schemes
 taking $v$ to the zero section, which is universal in the sense that given any
 $K$-morphism $f:V\to A$  to an abelian variety $A$ taking $v$ to the zero
 section, there is a unique homomorphism $\operatorname{Alb}_{V/K}\to A$ making
 the following diagram commute\,:
 $$
 \xymatrix{
  V \ar[d]^a \ar[r]^f& A\\
  \operatorname{Alb}_{V/K} \ar@{-->}[ru]_{\exists !}
 }
 $$

 \subsection{Existence of Albanese varieties}
 The existence of the Albanese variety goes back essentially to Serre.  We
        direct the reader to \cite[Thm.~A.1 and p.836]{wittenberg08}
        for an exposition valid over an arbitrary field\,; the
        assertion there is made for $V/K$ a geometrically integral
        scheme of finite type over a field~$K$, although
        argument holds under the weaker hypothesis that $V$ is a
        geometrically reduced and geometrically connected 
        scheme of finite type over $K$.

 \begin{teo}[Serre]\label{T:Serre-Alb}
  Let $V$ be a geometrically reduced and geometrically connected  scheme
  of finite type over a field $K$.  Then $V$ admits an Albanese torsor, and if $V$ admits
  a $K$-point, then $V$ admits a pointed Albanese variety. \qed
 \end{teo}

 \subsection{Base change for Albanese varieties}
 If $L/K$ is any field extension, then the universal property provides
 a base change morphism of torsors
 \[
 \xymatrix{\beta^1_{V,L/K}: \alb^1_{V_L/L} \ar[r] &(\alb^1_{V/K})_L}
 \]
 over a base change morphism of abelian varieties
 \[
 \xymatrix{\beta_{V,L/K}: \alb_{V_L/L} \ar[r] &(\alb_{V/K})_L}
 \]
 Our goal in this section is to
 understand $\beta_{V,L/K}$ and $\beta^1_{V,L/K}$, \emph{i.e.}, to understand the
 behavior of the
 Albanese variety under base change.
 We say that the Albanese variety of $V$ is \emph{stable under (separable) base change of
  field}  if $\beta_{V,L/K}$ and $\beta^1_{V,L/K}$ are isomorphisms for all
 (separable) field extensions $L/K$. \medskip

 The first observation is that Albanese varieties are not stable under base change of
 field\,:

 \begin{exa}[Albanese varieties are not stable under base change]
  \label{E:badbc}
  Let $L/K$ be a finite, purely inseparable extension.  Let $A/L$ be an
  abelian variety whose $L/K$-image, $\im_{L/K}A$, is trivial.  Let $G =
  \res_{L/K}A$ be the Weil restriction, which is a smooth connected commutative algebraic
  group over $K$.  Then $\alb_{G_L/L} = A$, while
  $\alb_{G/K}=\operatorname{Spec}K$, so that  $\alb_{G_L/L}\ne
  (\alb_{G/K})_L$\,;
  see \cite[Ex.~4.2.7]{brionstructure}, which goes back to Raynaud.
 \end{exa}

 There are two possible pathologies to focus on in this example.  First, the
 variety $G/K$ is not proper, and second, the extension $L/K$ is not
 separable.  Regarding the former pathology, it has been understood that if one
 assumes $V$ is proper, then the Albanese variety is stable under base change\,:

 \begin{teo}[Grothendieck--Conrad]\label{T:GrConBC-1}
  Let $V$ be a geometrically reduced and geometrically connected  \emph{proper}
  scheme over a field $K$.  Then the Albanese torsor of $V$ is stable under base change
  of field, and if $V$ admits a $K$-point, then the pointed Albanese
  variety is stable
  under base change of field. \qed
 \end{teo}

 \begin{rem}
  Recall that Grothendieck provides an Albanese torsor for any
  geometrically normal proper scheme $X$ over a field $K$ in the following way.
  As $X/K$ is proper and geometrically normal, one has that
  $\operatorname{Pic}^0_{X/K}$ is proper \cite[Thm.~VI.2.1(ii)]{FGA}\,; then by
  \cite[Prop.~VI.2.1]{FGA}, one has that $(\operatorname{Pic}^0_{X/K})_{\operatorname{red}}$ is a
  group scheme (\emph{i.e.}, without the usual hypothesis that $K$ be perfect).
  It then follows from \cite[Thm.~VI.3.3(iii)]{FGA} that
  $(\operatorname{Pic}^0_{X/K})^\wedge$ 
\Yano{Better to write $((\operatorname{Pic}^0_{X/K})_{\operatorname{red}})^\vee$.}
   is an Albanese variety.
  Conrad has generalized Grothendieck's argument to show that any
  geometrically reduced geometrically proper
  \Yano{Typo -this should be ``geometrically reduced proper scheme $X$'' -- geometrically connected is just needed for base change of the Albanese.  This paragraph essentially explains, but I will write out all of the details soon.}
   scheme $X$ over a field $K$ admits an
  Albanese torsor and Albanese variety. For lack of a better reference, we direct the reader to
  \cite[Thm.]{ConradMathOver}.
  His argument is to show that the Albanese variety is the dual abelian variety to the
  maximal abelian subvariety of the (possibly non-reduced and non-proper) Picard
  scheme  $\operatorname{Pic}_{X/K}$.  Grothendieck's theorem can then be
  summarized in this context by saying that his additional hypothesis that $X$ be
  geometrically normal implies that the maximal abelian subvariety of
  $\operatorname{Pic}_{X/K}$ is $\operatorname{Pic}^0_{X/K}$.
  That Grothendieck's and Conrad's Albanese varieties are stable under arbitrary field
  extension 
\Yano{when $X$ is assumed further to be geometrically connected (a necessary hypothesis for base change).} 
  is  \cite[Thm.~VI.3.3(iii)]{FGA} and \cite[Prop.]{ConradMathOver},
  respectively.  In fact, the Albanese variety enjoys an even
                stronger universal property; see \S \ref{S:appuniv} below.
 \end{rem}

 Returning to Example \ref{E:badbc}, and the second pathology, namely that the
 field extension $L/K$ in the example is not separable, our goal in this appendix
 is to use Theorem \ref{T:GrConBC-1} to prove\,:

 \begin{teo}[Base change for separable extensions]\label{T:SepBC-Alb}
  Let $V$ be a geometrically integral separated
  scheme of finite type over a field $K$. Then the Albanese torsor of $V$ is stable under
  \emph{separable} base change of field, and if $V$ admits a $K$-point, then the
  pointed Albanese variety is stable under \emph{separable} base change of field.
 \end{teo}

 This generalizes the abelian (as opposed to semi-abelian) case
        of \cite[Cor.~A.5]{wittenberg08}, which requires $V$ to be an open
 subset of a smooth geometrically integral proper scheme over $K$.
 The remainder of this appendix is devoted to proving this theorem.

 \subsection{Proof of Theorem \ref{T:SepBC-Alb}}

 Our first observation is that it turns out that  $\beta_{V,L/K}$ is an
 isomorphism if and only if  $\beta^1_{V,L/K}$ is an isomorphism\,:

 \begin{lem}
  \label{L:betavsbeta1}
  Let $V/K$ be a geometrically reduced and geometrically connected separated
  scheme of finite type over a field $K$, and let $L/K$ be an extension of fields.
  Then $\beta_{V,L/K}$ is an isomorphism if and only if $\beta^1_{V,L/K}$ is an
  isomorphism.
 \end{lem}

 \begin{proof}
  On one hand, let $T$ be a torsor under an abelian variety $A$\,; then $A$ can
  be recovered from~$T$ as  $\aut_T^0$, the connected component of identity of the
  automorphism group scheme of $T$.  Consequently, if $\alb^1_{V_L/L}$ and
  $(\alb^1_{V/K})_L$ are isomorphic, then so are $\alb_{V_L/L}$ and
  $(\alb_{V/K})_L$.

  On the other hand, suppose $\beta_{V,L/K}$ is an isomorphism.  Then
  $\beta^1_{V,L/K}$ is a nontrivial map of torsors over an isomorphism of abelian
  varieties.  Since the category of torsors under a given abelian variety is a
  groupoid, $\beta^1_{V,L/K}$ is an isomorphism.
 \end{proof}

 For \emph{finite} separable extensions, an easy argument shows\,:

 \begin{lem}
  \label{L:albBCfinsep}
  Let $V/K$ be a geometrically reduced and geometrically connected separated
  scheme of finite type over a field $K$.
  If $L/K$ is finite and separable, then $\beta^1_{V,L/K}$ and
  $\beta_{V,L/K}$ are isomorphisms.
 \end{lem}

 \begin{proof}
  By the universal property, it suffices to show that if $A/L$ is any
  abelian variety, $T$ a torsor under $A$, and $\alpha: V_L \to T$ a
  morphism, then $\alpha$ factors through $(\alb^1_{V/K})_L$.

  Since $L/K$ is finite and separable, the Weil restriction $\res_{L/K}(A)$ is an
  abelian
  variety, and $\res_{L/K}(T)$ is a torsor under $\res_{L/K}(A)$.
  By the adjoint property of Weil restriction \cite[p.~191, Lemma 1]{BLR},
  there's a map $V \to \res_{L/K}(T)$.  By the universal property of
  $\alb_{V/K}$, this factors as
  \[
  \xymatrix{
   V \ar[r] \ar[rd]& \alb^1_{V/K} \ar[d] \\
   &\res_{L/K}(T)
  }
  \]
  Again by the adjoint property of $\res_{L/K}$, this induces a diagram
  \[
  \xymatrix{
   V_L \ar[r] \ar[rd]& (\alb^1_{V/K})_L \ar[d] \\
   &T
  }
  \]
  and $(\alb^1_{V/K})_L$ is the universal torsor receiving a map from
  $V$.  One concludes with Lemma \ref{L:betavsbeta1}.
 \end{proof}

 The following result only requires that $V$ be regular and geometrically
 connected\,; but since we
 insist on working with geometrically reduced varieties, it is
 equivalent to assume that $V$ is smooth and irreducible.

 \begin{lem}
  \label{L:albopen}
  Let $V/K$ be a  geometrically integral smooth separated scheme of finite type
  over a field $K$.
  If $U\hookrightarrow V$ is an open immersion, then $\alb_{U/K} \iso \alb_{V/K}$
  and $\alb^1_{U/K} \iso \alb^1_{V/K}$.
 \end{lem}

 \begin{proof} It suffices to show that, if $f:U \to T$ is a morphism
  to a torsor under an abelian variety, then $f$ extends to a morphism
  $\tilde f: V \to T$.  If $T$ is an abelian variety, this is a special
  case of \cite[\S 8.4, Cor.~6]{BLR} (recalled in Proposition~\ref{P:BLR}$(i)$).
  Otherwise, let $L/K$ be a finite
  extension which splits~$T$, so that $f_L$ extends to $\tilde f_L: V_L
  \to T_L$\,; the statement then follows immediately from fpqc descent, as
  recalled below (Lemma \ref{L:fpqc}).
 \end{proof}

 \begin{lem}
  \label{L:fpqc}
  Let $S$ be a separated scheme over $K$ and let $X$ and $Y$ be separated
  schemes  over $S$.
  Let $U\subset S$ be an open, dense subscheme, and let $S' \to S$ be
  faithfully flat and quasicompact.  Suppose $f: X_U \to Y_U$ is a morphism of
  schemes over $U$.  If $f_{S'}: X_U \times_S S' \to Y_U\times_S S'$ extends to a
  morphism $\til f':X_{S'} \to Y_{S'}$, then $\til f'$ descends to a morphism
  $\til f: X \to Y$ over $S$, and $\til f_U = f$.
 \end{lem}

 \begin{proof}
  Let $S'' = S' \times_S S'$, equipped with the two projections $p_i:S'' \to
  S'$. Let $\Gamma_{\til f'} \subset X_{S'} \times_{S'} Y_{S'}$ be the graph of
  $\til f'$.  By Grothendieck's theory of fpqc descent (\emph{e.g.}, \cite[\S
  6.1]{BLR} or \cite[Thm.~3.1]{conradtrace}), it suffices to demonstrate an
  equality of closed subschemes $p_1^*(\Gamma_{\til f'}) = p_2^*(\Gamma_{\til
   f'})$. However, $p_i^*(\Gamma_{\til f'})$ contains $p_i^*(\Gamma_{f_{S'}})$ as a
  dense set\,; and $p_1^*(\Gamma_{f_{S'}}) = p_2^*(\Gamma_{f_{S'}})$, because
  $f_{S'}$ descends to $f$.
 \end{proof}

 For a variety $V$ which admits a smooth
 alteration, formation of the Albanese torsor commutes with separable base
 change\,:

 \begin{lem}
  \label{L:albBCautLK}
  Let $V/K$ be a geometrically integral
          separated
  scheme  of finite type over a field $K$.  Suppose that there is a
  diagram
  \[
  \xymatrix{
   U \ar@^{^{(}->}[r]^\iota \ar[d]^\pi & X \\
   V
  }
  \]
  of $K$-schemes  with $\pi$ surjective, $\iota$ an open
                immersion, and $X$ a geometrically integral separated scheme
                which is smooth and proper over $K$.  Let $L/K$ be any field
  extension such that $L^{\aut(L/K)} = K$.  Then $\beta_{V,L/K}$ and
  $\beta^1_{V,L/K}$ are isomorphisms.
 \end{lem}

 \begin{proof}
  We adapt Mochizuki's strategy \cite[App.~A]{mochizuki12} to
  understand the
  Albanese torsor of $V$
  in terms of that of its alteration $U$. Note that, by Lemma \ref{L:albopen} and
  Theorem \ref{T:GrConBC-1}, $\beta^1_{U,L/K}$ and $\beta_{U,L/K}$ are
  isomorphisms.
  Let $\mathcal I(V)$ denote the collection of all maps $\alpha: V \to T_\alpha$,
  where $T_\alpha$ is a torsor under an abelian variety $A_\alpha$.
  By the universal property, each such $\alpha$ induces a diagram
  \[
  \xymatrix@R=1.5em{
   U \ar[r]^a \ar[d]^\rho & \alb^1_{U/K}\ar[d]^{\delta^1_\alpha}\\
   V \ar[r]^\alpha& T_\alpha
  }
  \]
  and $\delta^1_\alpha$ is equivariant with respect to the induced
  morphism $\delta_\alpha: \alb_{U/K} \to A_\alpha$ of abelian
  varieties.  Let $\Xi_\alpha = \ker \delta_\alpha$, and let
  $\Xi_{V/K,U} = \cap_\alpha \Xi_\alpha$.  A Noetherian argument
  (\emph{e.g.}, \cite[Lem.~A.9]{mochizuki12}) shows that there exists some
  $\alpha\in \mathcal I(V)$ with $\Xi_{V/K,U} = \Xi_\alpha$, and then
  $\alb_{V/K} \iso \alb_{U/K}/\Xi_{V/K,U}$
  \cite[Cor.~A.11]{mochizuki12}.

  Base change yields $\cali(V)_L \hookrightarrow \mathcal I(V_L)$, and
  thus $(\Xi_{V/K,U})_L\supseteq \Xi_{V_L/L,U}$\,; to prove the lemma for
  the Albanese abelian variety, it suffices to prove the opposite
  containment.  Suppose $\alpha\in \mathcal I(V_L)$ and $\sigma \in
  \operatorname{Aut}(L/K)$.  Then we have $(\alpha^\sigma: V_L \to
  (A_\alpha)^\sigma)\in \mathcal I(V_L)$, and $\Xi_{\alpha^\sigma} =
  (\Xi_\alpha)^\sigma$.  In particular, $\Xi_{V_L/L,U} = \cap_\alpha
  \Xi_\alpha$ is stable under $\operatorname{Aut}(L/K)$, and thus
  descends to $K$.  Base changing back to $L$ (and then invoking Lemma
  \ref{L:betavsbeta1}) shows that $\beta_{V,L/K}$ and $\beta^1_{V,L/K}$ are
  isomorphisms.
 \end{proof}

 We now collect a few elementary results on separable extensions which will allow
 us to extend Lemma \ref{L:albBCautLK}
 to arbitrary separable extensions.

 \begin{lem}
  \label{L:autLKsepclosed}
  Let $\Omega/k$ be an extension of separably closed fields. Then
  $\Omega/k$ is separable if and only if $\Omega^{\aut(\Omega/k)} = k$.
 \end{lem}

 \begin{proof}
  In general, if $L/K$ is a separable extension, then $L$ is separable
  over $L^{\operatorname{Aut}(L/K)}$\,; see \emph{e.g.} \cite[\S 15.3,
  Prop.~7]{bourbakifields}. In particular, if $L^{\operatorname{Aut}(L/K)} =
  K$, then $L/K$ is separable.

  Conversely, assume that $\Omega/k$ is a separable extension of
  separably closed fields. Since $k$ is separably closed and since any
  sub-extension $\Omega/K/k$ satisfies $K/k$ separable, in order to show
  that $\Omega^{\operatorname{Aut}(\Omega/k)} =k$, it is enough to show
  that $\Omega^{\operatorname{Aut}(\Omega/k)} /k$ is algebraic. Let
  $\alpha \in \Omega$ be a transcendental element over $k$.  Since
  $\alpha$ extends to a transcendence basis of $\Omega/k$, the map
  $\alpha \mapsto \alpha+1$ extends to an automorphism of $\Omega$ which
  fixes $k$.  Consequently, no element of $\Omega$ transcendental over
  $k$ is fixed by all of $\aut(\Omega/k)$, and
  $\Omega^{\operatorname{Aut}(\Omega/k)} /k$ is algebraic, as desired.
 \end{proof}

 Recall that if $L'/L/K$ is a tower of field extensions with $L'/K$ separable,
 then $L/K$ is separable \cite[Prop.~8, p.V.116]{bourbakifields} but that $L'/L$
 may not be separable (\emph{e.g.},
 $\mathbb{F}_p(T)/\mathbb{F}_p(T^p)/\mathbb{F}_p$). However, we have\,:

 \begin{lem}\label{L:L/KsepCl}
  Suppose that $L/K$ is a separable extension of fields.  Then $L^{\sep}/K^{\sep}$
  is separable.
 \end{lem}

 \begin{proof} In characteristic $0$ there is nothing to show.  So let
  $p=\operatorname{char}K>0$.
  We start with a small observation \cite[Exe.~4 p.V.165]{bourbakifields}\,:   If
  $F/E/K$ is a tower of field extensions, with $F/K$ separable, then if $E^pK=E$,
  then $F/E$ is separable.  To prove this, it suffices to show that the natural
  map $F^p\otimes_{E^p} E \to F^pE$ is injective. By the assumption $E=E^pK$, we
  therefore must show $F^p \otimes_{E^p}E^pK\to F^p(E^pK)$ is injective.   Since
  $E^p/K$ is separable~\cite[Prop.~8, p.V.116]{bourbakifields},
  we have that $E^p\otimes_{K^p}K\hookrightarrow E^pK$ is injective.   Since
  field extensions are (faithfully) flat, tensoring by $F^p\otimes_{E^p}(-)$ we
  obtain
  \begin{equation}\label{E:L/KsepCl}
  F^p \otimes_{E^p} (E^p \otimes_{K^p}K) \hookrightarrow F^p\otimes_{E^p}E^pK \to
  F^p(E^pK)
  \end{equation}
  The composition is identified with the map $F^p\otimes_{K^p}K\to F^pK\subseteq
  F^p(E^pK)$, which is injective since $F/K$ is assumed to be separable.  However,
  since $E^pK$ is the field of fractions of $E^p \otimes_{K^p} K$ under the
  inclusion $E^p \otimes_{K^p} K\hookrightarrow E^pK$, we see that the right hand
  map $ F^p \otimes_{E^p} KE^p \to F^p (E^pK)$ in \eqref{E:L/KsepCl} is injective,
  as claimed, since it is obtained from the composition   $F^p \otimes_{E^p} (E^p
  \otimes_{K^p}K)  \to F^p(E^pK)$ in~\eqref{E:L/KsepCl} by localization.

  To prove the lemma, we apply the observation in the previous paragraph with
  $F=L^{\sep}$ and $E=K^{\sep}$.  Thus we just need to show that
  $(K^{\sep})^pK=K^{\sep}$.    Thus we have reduced to the following\,: if $E/K$
  is a separable algebraic extension\,;  then $E=E^pK$. Indeed, we have a tower of
  extensions $E/E^pK / K$. The extension $E/E^pK$ is purely inseparable (the
  $p$-th power of every element of $E$ belongs to $E^pK$) while the extension
  $E/K$ is separable. This implies $E=E^pK$.
 \end{proof}

 \begin{lem}
  \label{L:albBCsmoothalteration}
  Suppose $U$, $V$, and $X$ are as in Lemma \ref{L:albBCautLK}.  If $L/K$
  is any separable extension, then $\beta_{V,L/K}$ is an isomorphism.
 \end{lem}

 \begin{proof}
  Let $K^{\sep}$ and $L^{\sep}$ be separable closures of, respectively, $K$ and
  $L$, and consider the diagram of separable (thanks to Lemma \ref{L:L/KsepCl})
  field extensions
  \[
  \xymatrix@R=.5em{
   L^{\sep} \ar@{-}[dd] \ar@{-}[dr] \\
   & L \ar@{-}[dd]\\
   K^{\sep}\ar@{-}[dr]&\\
   &K
  }
  \]
  We have $(K^{s})^{\operatorname{Aut}(K^{s}/K)} = K$,
  $(L^{s})^{\operatorname{Aut}(L^{s}/L)} = L$ and by
  Lemma \ref{L:autLKsepclosed} we also have
  $(L^{s})^{\operatorname{Aut}(L^{s}/K^{s})} = K^{s}$. Hence we can
  apply Lemma \ref{L:albBCautLK} to all three extensions.
  Using the
  universal property of the Albanese morphism, we obtain a diagram with canonical
  arrows
  $$\xymatrix{\operatorname{Alb}_{V/L^{\sep}} \ar[r]^\cong \ar[d]_\cong &
   (\operatorname{Alb}_{V/L})_{L^{\sep}} \ar[d] \\
   ((\operatorname{Alb}_{V/K})_{K^{\sep}})_{L^{\sep}} \ar@{=}[r] &
   ((\operatorname{Alb}_{V/K})_L)_{L^{\sep}}.  }$$ It follows that
  $\beta_{V, L/K}$  becomes an isomorphism after base-change to $L^{\sep}$, and
  hence that it
  is an isomorphism.
 \end{proof}

 Recall \cite{conradtrace} that if $L/K$ is a primary extension of fields and if
 $A$ is an
 abelian variety, then the $L/K$-image of $A$, $\im_{L/K}(A)$, is the
 abelian variety over $K$ which is universal for maps from $A$ to
 abelian varieties which are defined over $K$.
 The formation of the image is insensitive to separable field
        extensions.  Indeed, a
 special case of \cite[Thm.~5.4]{conradtrace} states\,:

 \begin{lem}
  \label{L:image}
  If $L/K$ is purely inseparable, if $M/K$ is separable, and if $A/L$ is
  an abelian variety, then
  \[
  (\im_{L/K}(A))_{M} \iso \im_{LM/M}(A_{LM}).
  \]
 \end{lem}

 Armed with this, we can explain the Raynaud Example \ref{E:badbc}\,:

 \begin{lem}
  \label{L:albBCinsep}
  Let $V/K$ be a geometrically reduced and geometrically connected separated
  scheme of finite type over a field $K$.  Suppose $L/K$ is a finite purely inseparable extension.
  Then $\alb_{V/K} \iso \im_{L/K}\alb_{V_L/L}$.
 \end{lem}

 \begin{proof}
  Since $V$ is geometrically reduced, it admits a point over some finite
  separable extension $M/K$.  By Lemmas \ref{L:image} and
  \ref{L:albBCfinsep}, it suffices to verify the lemma after base change
  to $M$.  Thus, we may and do assume that $V$ admits a $K$-point, and
  consequently  that the Albanese torsor and the Albanese abelian variety
  coincide.

  To ease notation slightly let $\ubar A = \im_{L/K}(\alb_{V_L/L})$, and let
  $\ubar a$ be the composite map
  \[
  \xymatrix{
   \ubar a: V_L \ar[r] & \alb_{V_L/L} \ar[r] & \ubar A_L
  }
  \]
  It is universal for pointed maps from $V_L$ to abelian varieties which
  are defined over $K$. Since the base change of any morphism $V\to B$
  from $V$ to a $K$-abelian variety is such a map, it suffices to show
  that $\ubar a$ descends to a
  morphism $V \to \ubar A$ over $K$.

  As before, we argue using fpqc descent.  Let $p_1,p_2: \spec(L\otimes_K L) \to
  \spec L$ be the two projections.  It suffices to show that
  $p_1^*(\ubar a) = p_2^*(\ubar a)$ as morphisms from $V_{L\otimes_KL}$
  to $\ubar A_{L\otimes_KL}$.
  This follows from rigidity\,; see Lemma \ref{L:rigidity} below, where
  $(V,K,R,f,g) = (V, L, L\otimes_KL,
  p_1^*(a), p_2^*(a))$, and $L\otimes_KL$ is local because $L/K$ is
  purely inseparable.  (See also Remark \ref{R:AlbDefArg}.)
 \end{proof}

 \begin{lem}
  \label{L:rigidity}
  Let $K$ be a field and let  $(V,P)/K$ be a  geometrically reduced and
  geometrically
  connected separated scheme of finite type over a field $K$ with a $K$-point $P$.
  Let $R$ be a Noetherian local $K$-algebra, and let $\circ$ be
  the closed point of $\spec R$.
  Let $A/R$ be an abelian
  scheme.    Given two pointed morphisms
  $f,g: V_R \to A$, if $f_\circ = g_\circ$, then $f=g$.
 \end{lem}

 \begin{proof}
  Using Nagata compactification and \cite[Lem~2.2]{lutkebohmert93}, or
  \cite[Rem.~2.5]{conradNagata} directly, we find a proper scheme $X \to
  \spec R$ admitting an open immersion $V_R \hookrightarrow X$, and such
  that $f$ and $g$ extend to morphisms $X \to A$.  Consider the
  composite morphism
  \[
  \xymatrix{
   \delta = (f-g): X \ar[r]^-{(f,g)} & A \times A \ar[r]^-{-} & A
  }
  \]
  of proper $R$-schemes.  By hypothesis, $\delta_\circ(X_\circ)$ is the
  zero section.  By rigidity (\emph{e.g.}, \cite[Prop.~6.1]{mumfordGIT}),
  $\delta$ is constant, and thus $f =g$ schematically.  Restricting to $V_R$ gives
  the desired result.
 \end{proof}

 \begin{rem}\label{R:AlbDefArg}
  In the proof of Lemma \ref{L:albBCinsep}, the fact that there is an equality of morphisms
  $p_1^*(\ubar a) = p_2^*(\ubar a):V_{L\otimes_KL}\to \ubar A_{L\otimes_KL}$ can
  also be shown \emph{via} deformation theory.
  Set $R = L \otimes_K L$ viewed as a local Artinian $L$-algebra.
  We have that  $p_1^*a,p_2^*a:V_R \to A_R$ are both deformations of the morphism
  $a : V_L \to A_L$.
  Using Nagata's compactification theorem,
   and then
  \cite[Lem~2.2]{lutkebohmert93} or
  \cite[Rem.~2.5]{conradNagata},  we can   embed $V_L$ as an open subset of  a
  (geometrically \cite[Lem~32.6.8 (04KS)]{stacks-project})
        reduced proper scheme $W$ of finite type over $L$ admitting an $L$-morphism
  $a:W\to A_L$ extending the morphism $a:V_L\to A_L$.
  We can now consider $p_1^*a, p_2^*a:W_R \to  A_R$, both deformations of the map
  $a : W \to A_L$, and extending the maps $p_1^*a, p_2^*a:V_R \to A_R$.
  The first order deformations of the map a are given by $H^0(W, a^*T_{A_L})$
  (\cite[Thm.~6.4.12]{FGAexplained}, or  \cite[Prop.~24.9]{HartshorneDef}).
  If $\dim (A) = g$, then $T_{A_L}$ is a trivial rank $g$ vector bundle, and so
  the first order deformation space, $H^0(W,a^*T_{A_L})$, is a $g$ dimensional
  vector space over $L$.
  Since there is a $g$-dimensional space of first order  deformations given by the
  translations in $A_L$, every first order deformation of $a:W\to A_L$ that
  preserves the base points is trivial.
  The standard arguments reducing to square-zero extensions show that every
  infinitesimal deformation of $a:W\to A_L$ that preserves the base points  is
  trivial.
  One then concludes that $p_1^*a=p_2^*a:W_R\to A_R$, which forces
  $p_1^*a=p_2^*:V_R\to A_R$, which was what we had to prove.
 \end{rem}

 Finally, we can prove Theorem \ref{T:SepBC-Alb}\,:

 \begin{proof}[Proof of Theorem \ref{T:SepBC-Alb}]
  The plan is to chase Albanese varieties among the diagram of fields
  \[
  \xymatrix@R=.5em{
   LM \ar@{-}[dd]_{\text{insep}} \ar@{-}[dr]^{\text{sep}} \\
   & L \ar@{-}[dd]^{\text{insep}}\\
   M\ar@{-}[dr]_{\text{sep}}&\\
   &K
  }
  \]
  Using Nagata compactification \cite{conradNagata}, embed $V \hookrightarrow Y$
  into a
  proper geometrically integral variety.  Using \cite{deJong}, there is a
  diagram
  \[
  \xymatrix{
   U \ar[d] \ar@{^{(}->}[r] & X \ar[d] \\
   V  \ar@{^{(}->}[r] & Y
  }
  \]
  in which the vertical arrows are alterations\,;  moreover, there is a
  finite, purely inseparable $L/K$ such that the structural morphism $X
  \to \spec K$ factors through $\spec L$, and $X \to \spec L$ is
  smooth.  Note that $U$, $X$ and $V_L$ satisfy the hypotheses of Lemma
  \ref{L:albBCautLK}.

  We then compute canonical isomorphisms
  \begin{align*}
  \alb_{V_M/M} &\iso \im_{LM/M}(\alb_{V_{LM}/LM}) \quad \text{(Lemma
   \ref{L:albBCinsep})}\\
  &\iso \im_{LM/M}((\alb_{V_L/L})_{LM})  \quad \text{(Lemma \
   \ref{L:albBCsmoothalteration})}\\
  &\iso (\im_{L/K} \alb_{V_L/L})_M \quad \text{(Lemma \ref{L:image})}\\
  &\iso ( \alb_{V/K})_M \quad \text{(Lemma \ref{L:albBCinsep})}.
  \end{align*}
 \end{proof}

\subsection{The universal property of Albanese varieties}
\label{S:appuniv}

In fact, Theorem \ref{T:GrConBC-1} as stated above is weaker than what
Grothendieck \cite[Thm.~VI.3.3(iii)]{FGA} and Conrad
\cite[Thm.]{ConradMathOver} actually prove\,:

 \begin{teo}[Grothendieck--Conrad]
  \label{T:groconrad}
  Let $V/K$ be a geometrically reduced and geometrically connected proper scheme
  over $K$.  Then for any $S
  \to \spec K$, and any $f: V_S \to T$ to a torsor under an abelian
  scheme $B/S$, there exists a unique $S$-map $g: (\alb^1_{V/K})_S \to
  T$ such that $g\circ a_S = f$. \qed
 \end{teo}

If one is willing to retrict to base change by smooth morphisms, one
can derive a similar statement for varieties without a properness
hypothesis.

 \begin{teo}
  \label{T:groconrad-open}
  Let $V/K$ be a geometrically integral separated
  scheme of finite type over a field $K$.  Then for any (inverse limit of) smooth
  morphism $S
  \to \spec K$, and any $f: V_S \to T$ to a torsor under an abelian
  scheme $B/S$, there exists a unique $S$-map $g: (\alb^1_{V/K})_S \to
  T$ such that $g\circ a_S = f$.
 \end{teo}
 \begin{proof} By assumption on the morphism $S \to \operatorname{Spec} K$,
  $\kappa(S)/K$ is a separable extension.
  Consider then the restriction $f_{\eta_S}$ of  $f: X_S \to T$ to the generic
  point $\eta_S$ of $S$. By Theorem~\ref{T:SepBC-Alb}, $f_{\eta_S}$ factors
  through $(\mathrm{alb}^1_{V/K})_{\eta_S}$. This gives a canonical
  $\eta_S$-morphism of torsors $(\mathrm{Alb}^1_{V/K})_{\eta_S} \to T_{\eta_S}$
  over $(\mathrm{Alb}_{V/K})_{\eta_S}$.  Let $U\subset S$ be an open dense
  subscheme to which these morphisms extend as $g^1: (\alb^1_{V/K})_U \to T_U$
  over $g: (\alb_{V/K})_U \to B_U$.
  By Proposition \ref{P:BLR}$(ii)$ \cite[I.2.7]{FC},
  $g$ extends to a morphism of abelian schemes over $S$.
  Let $S' \to S$ be an fpqc morphism such that $(\alb^1_{V/K})_{S'}\to S'$ and
  $T_{S'}\to S'$ admit sections.  Then $(\alb^1_{V/K})_{S'}$ and $T_{S'}$ are
  trivial torsors under abelian schemes over $S'$, and so $g^1_{U\times_S S'}$
  extends to a morphism $(g^1)': (\alb^1_{V/K})_{S'} \to T_{S'}$.  By fpqc descent
  (Lemma~\ref{L:fpqc}), $(g^1)'$ descends to a morphism $g^1: (\alb^1_{V/K})_S \to
  T$, as desired.
 \end{proof}

\newcommand{\etalchar}[1]{$^{#1}$}
\def\cprime{$'$}
\providecommand{\bysame}{\leavevmode\hbox to3em{\hrulefill}\thinspace}
\providecommand{\MR}{\relax\ifhmode\unskip\space\fi MR }
\providecommand{\MRhref}[2]{%
	\href{http://www.ams.org/mathscinet-getitem?mr=#1}{#2}
}
\providecommand{\href}[2]{#2}

\end{document}